\documentclass[cag]{ipart_v1}

\Year{0000}
\Vol{00}
\Issue{0}
\firstpage{1}

\usepackage{amscd}
\usepackage{amsmath,amstext,amsthm,amssymb,amsfonts,amscd}
\usepackage{mathrsfs}
\usepackage{graphicx,tensor}
\usepackage[T1]{fontenc}
\usepackage[latin1]{inputenc}
\usepackage{color,tocvsec2}
\usepackage{enumerate}
\usepackage{lscape}
\usepackage[all]{xy}

\theoremstyle{plain}
\newtheorem{lem}{Lemme}[section]
\newtheorem{thm}[lem]{Theorem}
\newtheorem{prop}[lem]{Proposition}
\newtheorem{cor}[lem]{Corollary}
\newtheorem{as}[lem]{Assumption}

\theoremstyle{definition}
\newtheorem{defin}[lem]{Definition}
\newtheorem{exa}[lem]{Example}

\theoremstyle{remark}
\newtheorem{re}[lem]{Remark}

\numberwithin{equation}{section}
\numberwithin{figure}{section}

\newcommand{\cD}{\mathcal{D}}
\newcommand{\cE}{\mathcal{E}}
\newcommand{\cF}{\mathcal{F}}
\newcommand{\cG}{\mathcal{G}}
\newcommand{\cH}{\mathcal{H}}

\newcommand{\cM}{\mathcal{M}}

\newcommand{\cO}{\mathcal{O}}

\newcommand{\cS}{\mathcal{S}}

\newcommand{\cU}{\mathcal{U}}
\newcommand{\cV}{\mathcal{V}}

\newcommand{\bN}{\mathbf{N}}
\newcommand{\bZ}{\mathbf{Z}}
\newcommand{\bR}{\mathbf{R}}
\newcommand{\bC}{\mathbf{C}}

\newcommand{\fb}{\mathfrak{b}}
\newcommand{\fg}{\mathfrak{g}}
\newcommand{\fh}{\mathfrak{h}}
\newcommand{\fk}{\mathfrak{k}}

\newcommand{\fm}{\mathfrak{m}}
\newcommand{\fn}{\mathfrak{n}}
\newcommand{\fp}{\mathfrak{p}}

\newcommand{\ft}{\mathfrak{t}}
\newcommand{\fz}{\mathfrak{z}}
\newcommand{\fu}{\mathfrak{u}}

\newcommand{\sF}{\mathscr{F}}
\newcommand{\sG}{\mathscr{G}}

\DeclareMathOperator{\Trs}{\mathrm Tr_s}
\DeclareMathOperator{\Tr}{\mathrm Tr}
\DeclareMathOperator{\Ad}{\mathrm Ad}

\DeclareMathOperator{\vol}{\mathrm vol}

\DeclareMathOperator{\ad}{\mathrm ad}
\DeclareMathOperator{\Sp}{\mathrm{Sp}}
\DeclareMathOperator{\Supp}{\mathrm Supp}
\renewcommand{\Re}{\mathrm{Re}\,}
\DeclareMathOperator{\im}{\mathrm Im}
\DeclareMathOperator{\End}{\mathrm End}
\DeclareMathOperator{\Hom}{\mathrm Hom}

\DeclareMathOperator{\rk}{\mathrm rk}
\DeclareMathOperator{\GL}{\mathrm GL}

\newcommand{\<}{\langle}
\renewcommand{\>}{\rangle}
\newcommand{\ol}{\overline}
\newcommand{\ul}{\underline}

\newcommand{\p}{\partial}

\renewcommand{\(}{\left(}
\renewcommand{\)}{\right)}
\renewcommand{\[}{\left[}
\renewcommand{\]}{\right]}

\newcommand{\g}{\geqslant}
\newcommand{\e}{\epsilon}

\newcommand{\bbS}{\mathbb{S}}

\begin{document}

\title{Flat vector bundles and analytic torsion on  orbifolds}
\author[S. Shen and J. Yu]{Shu Shen and Jianqing Yu}

\begin{abstract}
This article is devoted to a study of flat orbifold  vector bundles. 
We construct a bijection between the isomorphic classes of proper 
flat orbifold  vector bundles and the equivalence classes of 
representations of the orbifold fundamental groups of  base orbifolds. We establish a Bismut-Zhang like  anomaly formula for the Ray-Singer metric on the determinant line of the cohomology of  a compact orbifold with coefficients in an orbifold  flat vector bundle. We show that the analytic torsion of  an acyclic unitary flat orbifold vector bundle is equal to the  value at zero of a dynamical zeta function when the underlying orbifold is  a compact locally symmetric space of reductive type, which extends one of the results obtained by  the first author for compact locally symmetric manifolds.
\end{abstract}

\maketitle
\tableofcontents

\settocdepth{section}
\section*{Introduction}
 Orbifolds were introduced by Satake \cite{Satake_gene_mfd} under name of $V$-manifold as manifolds with quotient singularities.  They appear naturally, for example,  in the geometry of 3-manifolds,  in the  symplectic reduction, in the problems on moduli spaces, and  in string theory, etc.

It is natural to consider  the index theoretic problem and the associated secondary invariants on  orbifolds. 
 Satake \cite{SatakeGaussB} and Kawasaki \cite{Kawasaki_Orb_sign, Kawasaki_RR} extended   
the classical Gauss-Bonnet-Chern Theorem,  the Hirzebruch signature Theorem and the Riemann-Roch-Hirzebruch Theorem.
For the secondary invariants,  Ma \cite{Ma_Orbifold_immersion} studied  the  holomorphic torsions and Quillen metrics associated with  holomorphic orbifold vector bundles, and  Farsi \cite{Farsi_eta} introduced an orbifold version eta invariant and extended  the Atiyah-Patodi-Singer Theorem.  In this article, we study  flat orbifold vector bundles and the associated secondary invariants, i.e., analytic torsions or more precisely   Ray-Singer metrics.

Let us recall   some results on  flat vector bundles on manifolds.
Let $Z$ be a connected smooth manifold, and let $F$ be a complex  flat vector bundle on $Z$. 
Equivalently, $F$ can be obtained via a complex representation of the fundamental group $\pi_1(Z)$ of $Z$, which 
is called the holonomy representation. Denote by $H^\cdot(Z,F)$ the cohomology of 
the sheaf of locally constant sections of $F$. 

Assume that $Z$ is compact.
Given  metrics $g^{TZ}$ and $g^F$ on  $TZ$ and $F$, the Ray-Singer metric \cite{RSTorsion} on the determinant line $\lambda$ of $H^\cdot(Z,F)$ is defined by the product of  the analytic torsion 
with an $L^2$-metric on $\lambda$. 

If $g^F$ is flat, or equivalently if the holonomy representation is 
unitary,  then the celebrated Cheeger-Müller Theorem 
\cite{Ch79,Muller78} tells us that in this case the Ray-Singer metric 
coincides with the so-called Reidemeister metric 
\cite{ReidemeisterTorsion}, which is a topological invariant of the unitarily  flat vector bundles constructed with the help of a triangulation on $Z$.  Bismut-Zhang \cite{BZ92} and Müller \cite{Muller2} simultaneously considered generalizations of this result. In  \cite{Muller2}, Müller extended it to the case 
where $g^F$ is unimodular or equivalently  the holonomy representation  is unimodular.  In \cite{BZ92}, Bismut and Zhang studied the dependence of the Ray-Singer metric on $g^{TZ}$ and  $g^F$.  They gave an anomaly formula \cite[Theorem 0.1]{BZ92} for the variation of the logarithm of the Ray-Singer metric on   $g^{T Z}$ and $g^F$ as an integral of a locally calculable Chern-Simons  form on $Z$. They
  generalized the original Cheeger-Müller Theorem to arbitrary flat vector bundles with arbitrary Hermitian metrics \cite[Theorem 0.2]{BZ92}. 
  In \cite{BZ94},  Bismut and Zhang  also considered the extensions to the equivariant case. Note that both in \cite{BZ92, BZ94}, the existence of a Morse function whose gradient satisfies the Smale transversality condition \cite{S61,S67} plays an important role. 
  
  From the dynamical side,  motivated by  a remarkable similarity \cite[Section 3]{MilnorZcover} between  the analytic  torsion and Weil's zeta function, 
  Fried \cite{FriedRealtorsion}  showed that, when the underlying manifold is hyperbolic,
   the analytic torsion of an acyclic unitarily flat vector bundle is equal to
   the  value at zero  of the Ruelle dynamical zeta function.  In \cite[p.66 Conjecture]{Friedconj}, he conjectured similar results hold true for more general spaces.


   In \cite{Shfried}, following the early contribution of  Fried 
   \cite{FriedRealtorsion} and  Moscovici-Stanton \cite{MStorsion}, 
   the  author  showed the Fried conjecture on closed locally 
   symmetric manifolds of the reductive type. The proof is based on  
   Bismut's explicit   semisimple orbital integral formula 
   \cite[Theorem 6.1.1]{B09}. (See  Ma's talk \cite{Ma_bourbaki} at Séminaire Bourbaki for an introduction.)

  In this article,  we extend most   of the above results  to 
  orbifolds.  
    Now, we will describe our results in more details and explain 
	the techniques used in the proof.

\subsection{Orbifold fundamental group and holonomy representation} 
Let $Z$ be a connected  orbifold with the associated groupoid $\cG$. Following Thurston \cite{Thurston_geo_3_maniflod}, 
let $X$ be the universal covering orbifold of $Z$ with the deck transformation group $\Gamma$, which is called  orbifold fundamental group of $Z$. Then, $Z=\Gamma\backslash X$. In an analogous way as in the classical homotopy theory of ordinary paths on topological  spaces,   Haefliger \cite{Haefliger_orb} introduced the $\cG$-paths and their homotopy theory.  He  gave an explicit construction of  $X$ and $\Gamma$ following the classical methods.

%
%
%
%

If  $F$ is a  complex proper flat orbifold vector bundle of rank $r$, in Section \ref{Sec:Top}, we constructed a parallel transport  along a $\cG$-path. In this way, we obtain a representation  $\rho:\Gamma\to \GL_r(\bC)$ of $\Gamma$, which is called the holonomy representation of $F$.  
Denote by  $\cM^{\rm pr}_r(Z)$ the  isomorphic classes  of complex proper flat orbifold  vector bundles of rank $r$ on $Z$, and denote  by $ {\rm Hom}(\Gamma,\mathrm{GL}_r (\bC))/_\sim$  the equivalence classes  of complex  representations of  $\Gamma$ of dimension $r$. We show the following theorem. 

%
%
%
%

\begin{thm}
\label{thm:1}
The above construction descends to  a well-defined bijection  
\begin{align}
\cM^{\mathrm pr}_r(Z)\simeq {\rm{Hom}}\big(\Gamma,\GL_r(\bC)\big)/_\sim.
\end{align}
\end{thm}

The difficulty of the proof  lies in the injectivity, which consists in  showing that $F$ is isomorphic to the quotient of $X\times  \bC^r$ by the $\Gamma$-action induced by the deck transformation on $X$ and by the holonomy representation on $\bC^r$. Indeed, applying  Haefliger's construction, in subsection \ref{sec:hol}, we  show directly that the  universal covering orbifold of the total space  of $F$ is  $X\times  \bC^r$. Moreover, its deck transformation group is isomorphic to  $\Gamma$ with the desired  action on $X\times \bC^r$. 


We remark that  on the universal covering orbifold there exist non trivial and non proper flat orbifold vector bundles. Thus, Theorem \ref{thm:1} no longer holds true for non proper orbifold vector bundles.  

On the other hand,  for a general orbifold vector bundle $E$ which is not necessarily proper, there exists a proper  subbundle $E^{\rm pr}$ of $E$ such that
\begin{align}\label{eq:inEpr}
C^\infty(Z,E)=C^\infty\(Z,E^{\rm pr}\).
\end{align}
Moreover, if $E$ is flat, $E^{\rm pr}$ is also flat.  For a $\Gamma$-space $V$, we denote by $V^\Gamma$ the set of  fixed points in $V$. By Theorem \ref{thm:1} and \eqref{eq:inEpr}, we get:
\begin{cor}\label{cor:1}
 For any  (possibly non proper) flat orbifold vector bundle $F$ on a connected orbifold $Z$, there exists  a representation of the orbifold fundamental group $\rho:\Gamma\to \GL_r(\bC)$ such that 
 \begin{align}
 C^\infty(Z,F)=C^\infty\(X,\bC^r\)^\Gamma.
 \end{align}
\end{cor}
By abuse of notation, in this case, although $\rho$ is not unique, we still call  $\rho$ a holonomy representation of $F$.

Waldron  informed us that in his PhD thesis \cite{Waldron} he proved  Theorem 
\ref{thm:1} in  a more abstract setting using
differentiable stacks.

\subsection{Analytic torsion on orbifolds}
%
%
%
%
Assume that $Z$ is a compact orbifold of dimension $m$. Let $\Sigma 
Z$ be the strata of $Z$, which has a natural orbifold structure. 
Write $Z\coprod \Sigma Z=\coprod_{i=0}^{l_0}Z_i$ as a disjoint union 
of connected components. We denote  $m_i\in \bN$ the multiplicity of  $Z_i$ (see \eqref{eq:mi25}).
Let $F$ be a complex flat orbifold  vector bundle on $Z$. Let 
$\lambda$ be the determinant line  of the cohomology $H^\cdot(Z,F)$ (see \eqref{eq:r=detH}).

Let $g^{TZ}$ and $g^F$ be   metrics on $TZ$ and $F$. Denote by $\Box^Z$ the associated Hodge Laplacian acting on the space $\Omega^\cdot(Z,F)$ of smooth forms with values in $F$. By the orbifold Hodge Theorem \cite[Proposition 2.1]{Daiyu}, we have the canonical isomorphism
\begin{align}\label{eq:hodgein}
H^\cdot(Z,F)\simeq \ker \Box^Z. 
\end{align}

As in the case of smooth manifolds, by \cite{Kawasaki_Orb_sign} (or by the short time asymptotic expansions of the heat trace \cite[Proposition 2.1]{Ma_Orbifold_immersion}),  the analytic torsion $T(F)$ is still well-defined. It is a real positive number defined  by the following  weighted product of the zeta regularized determinants 
\begin{align}
T(F)=\prod_{i=1}^m {\rm det}\( \Box^{Z}|_{\Omega^i(Z,F)}\)^{(-1)^ii/2}.
\end{align}
Let $|\cdot|^{\rm RS,2}_\lambda$ be the $L^2$-metric on $\lambda$ induced by $g^{TZ},g^F$ via \eqref{eq:hodgein}. 
The Ray-Singer metric on $\lambda$ is then given by 
\begin{align}
\|\cdot\|^{\rm RS}_\lambda=T(F)|\cdot|^{\rm RS}_\lambda.
\end{align}
We remark that as in the smooth case, if $Z$ is of even dimension and 
orientable, if $F$ is unitarily flat, in Proposition \ref{prop:48}, we  show that  $T(F)=1$.

In Section \ref{Sec:Tor}, we study the dependence of $\|\cdot\|^{\rm RS,2}_\lambda$ on $g^{TZ}$ and $g^F$. To state our result, let us introduce some notation. Let $(g^{\prime TZ},g^{\prime F})$ be another pair of metrics.  Let $\|\cdot\|^{\prime \rm RS,2}_\lambda$ be  the   Ray-Singer metric for $(g^{\prime TZ},g^{\prime F})$. 
Let $\nabla^{TZ}$ and $\nabla^{\prime TZ}$ be the respective Levi-Civita connections on $TZ$ for  $g^{TZ}$ and $g^{\prime TZ}$. Denote by $o(TZ)$  the orientation line of $Z$. Consider the Euler form
$e(TZ,\nabla^{TZ})\in \Omega^{m}(Z,o(TZ))$
 and the first odd Chern form $  
 \frac{1}{2}\theta(\nabla^F,g^F)=\frac{1}{2}\Tr[(g^{F})^{-1}\nabla^F g^F]\in \Omega^1(Z).$
  Denote by 
  \begin{align}
 & e\(Z_i,\nabla^{TZ_i}\)\in \Omega^{\dim 
  Z_i}(Z_i,o(TZ_i)),&\theta_i\(\nabla^F,g^F\)\in \Omega^1(Z_i)
  \end{align}
  the canonical extensions of  $e(TZ,\nabla^{TZ})$ and 
$\theta(\nabla^F,g^F)$ to $Z_i$ (see subsection \ref{sec:chara}). 
Let  
\begin{align}
\widetilde{e}(TZ_i,\nabla^{TZ_i},\nabla^{\prime TZ_i})\in 
\Omega^{\dim Z_i-1}(Z_i,o(TZ_i))/d \Omega^{\dim Z_i-2}(Z_i,o(TZ_i))
\end{align}
and $\widetilde{\theta}_i(\nabla^F,g^F,g^{\prime F})\in C^\infty(Z_i)$ be the associated Chern-Simons  forms such that 
\begin{align}
\begin{aligned}
d\,\widetilde{e}\(TZ_i,\nabla^{TZ_i},\nabla^{\prime 
TZ_i}\)&=e\(Z_i,\nabla^{\prime TZ_i}\)-e\(Z_i,\nabla^{TZ_i}\),\\
d\,\widetilde{\theta}_i\(\nabla^F,g^F,g^{\prime 
F}\)&=\theta_i\(\nabla^F,g^{\prime F}\)-\theta_i\(\nabla^F,g^F\).
\end{aligned}
\end{align}
In Section \ref{Sec:Tor}, 
we show:
\begin{thm}
\label{thm:2}The following identity holds:
\begin{multline}\label{eq:ano1}
\log\(\frac{\|\cdot\|^{\prime {\rm RS},2}_{\lambda}}{\|\cdot\|^{ \rm{RS},2}_{\lambda}}\)=\sum^{l_0}_{i=0}\frac{1}{m_i}\bigg(\int_{Z_i}\widetilde{\theta}_i\(\nabla^F,g^F,g^{\prime F}\)e\(Z_i,\nabla^{TZ_i}\)\\
- \int_{Z_i}\theta_i\(\nabla^F,g^{\prime F}\)\widetilde{e}\(TZ_i,\nabla^{TZ_i},\nabla^{\prime TZ_i}\)\bigg).
\end{multline}
\end{thm}
The arguments in Section \ref{Sec:Tor} 
are inspired by 
Bismut-Lott \cite[Theorem 3.24]{BLott}, who gave a unified proof for the family local index  theorem and the anomaly formula \cite[Theorem 0.1]{BZ92}. Conceptually, their proof is  simpler and more natural than the original proof given by Bismut-Zhang \cite[Section IV]{BZ92}. 
Also, our proof relies on  the finite propagation speeds for the solutions of hyperbolic equations on orbifolds, which is originally due to Ma \cite{Ma_Orbifold_immersion}.

If $Z$ is of odd dimension  and orientable, then all the $Z_i$, for $0\le i\le l_0$, is of odd dimension. By Theorem \ref{thm:2}, 
 the Ray-Singer metric $\|\cdot\|^{{\rm RS},2}_{\lambda}$ does not 
 depend on the metrics $g^{TZ}, g^F$; it becomes a topological invariant. 

\subsection{A solution of Fried conjecture on  locally symmetric orbifolds}
In \cite[p. 537]{FriedRealtorsion}, Fried 
raised the question of extending his result \cite[Theorem 1]{FriedRealtorsion} to hyperbolic orbifolds on the equality between  the analytic torsion   and the zero value of the Ruelle dynamical zeta function associated to a unitarily flat  acyclic  vector bundle
on hyperbolic manifolds. In Section \ref{Sec:Fried}, we extend 
Fried's  result to more general compact odd dimensional \footnote{The even dimensional case is trivial.} locally symmetric  orbifolds of the reductive type. 

Let $G$ be a linear connected real  reductive group with Cartan involution $\theta\in {\rm Aut}(G)$. Let 
$K\subset G$ be the set of fixed points of $\theta$ in $G$, so that $K$ is a maximal compact subgroup of $G$. Let $\fg$ and $\fk$ be the Lie algebras of $G$ and $K$. Let $\fg=\fp\oplus \fk$ be the Cartan decomposition. Let $B$ be an $\Ad(G)$-invariant and $\theta$-invariant non degenerate bilinear form on $\fg$ such that $B|_{\fp}>0$ and $B|_{\fk}<0$. 
Recall that  an element $\gamma\in G$ is said to be semisimple if and only if $\gamma$ can be conjugated to $e^ak^{-1}$ with $a\in \fp$, $k\in K$, $\Ad(k)a=a$. And $\gamma$ is said to be elliptic if and only if  $\gamma$ can be conjugated into $K$. Note that if $\gamma$ is semisimple,  its centralizer  $Z(\gamma)$ in $G$ is still reductive with maximal compact subgroup $K(\gamma)$.

Take  $X=G/K$ to be the associated symmetric space. Then, $B|_{\fp}$ induces a $G$-invariant Riemannian metric $g^{TX}$ on $X$ such that $(X,g^{TX})$ is of non positive sectional  curvature.  Let $d_X$ be the Riemannian distance on $X$. 
%

Let $\Gamma\subset G$  be a cocompact discrete subgroup of $G$. Set $Z=\Gamma\backslash G/K$. Then $Z$ is a compact orbifold with universal covering orbifold $X$. 
To simplify the notation in Introduction, we assume that $\Gamma$ acts effectively on $X$. 
Then $\Gamma$ is the orbifold fundamental group of $Z$.  Clearly, 
$\Gamma$ contains only semisimple elements. 
Let $\Gamma_+$ be the subset of $\Gamma$ consisting of non elliptic 
elements. Take $[\Gamma]$ to be the set of conjugacy classes of 
$\Gamma$. Denote by  $[\Gamma_+]\subset [\Gamma]$  the set of non 
elliptic conjugacy classes.


Proceeding as in the proof for the manifold case \cite[Proposition 5.15]{DuistermaatKolkVaradarajan}, 
 the set of closed geodesics  of  positive lengths  consists of a 
 disjoint union of smooth connected compact orbifolds
$\coprod_{[\gamma]\in [\Gamma_+]} B_{[\gamma]}$.
Moreover,  $B_{[\gamma]}$ is diffeomorphic to  $ \Gamma\cap Z(\gamma)\backslash Z(\gamma)/K(\gamma)$. Also, all the elements in  $B_{[\gamma]}$ have the same length $l_{[\gamma]}>0$. Clearly, the 
geodesic flow induces a locally free $\bbS^1$-action on 
$B_{[\gamma]}$. By an analogy to the  multiplicity $m_i$ of  $Z_i$ in 
$Z\coprod \Sigma Z$, we can define the multiplicity   $m_{[\gamma]}$ 
of the quotient orbifold $\mathbb{S}^1\backslash B_{[\gamma]}$ (see 
\eqref{eq:mulr}). Denote by  $\chi_{\rm 
orb}\big(\mathbb{S}^1\backslash B_{[\gamma]}\big)\in \mathbf{Q}$  the 
orbifold Euler characteristic number \cite[Section 3.3]{SatakeGaussB} (see also \eqref{eq:chiorb}) of $\mathbb{S}^1\backslash B_{[\gamma]}$. In Section \ref{Sec:Fried},  we show:

\begin{thm}
\label{thm:3}
If $\dim Z$ is odd, and if $F$ is a unitarily flat  orbifold  vector 
bundle on $Z$ with holonomy $\rho:\Gamma\to \mathrm{U}(r)$, then the dynamical zeta function 
\begin{align}
R_\rho(\sigma)=\exp\(\sum_{[\gamma]\in [\Gamma_+]}\Tr[\rho(\gamma)]\frac{\chi_{\mathrm{orb}}\(\mathbb{S}^1\backslash B_{[\gamma]}\)}{m_{[\gamma]}}e^{-\sigma l_{[\gamma]}}\)
\end{align}
is  well-defined and holomorphic  on $\Re(\sigma)\gg1$, and extends 
meromorphically to $\bC$. There exist explicit constants $C_\rho\in 
\bR$ with $C_\rho\neq0$ and $r_\rho\in \bZ$ (see \eqref{eq:Cr}) such that as $\sigma\to0$, 
\begin{align}
R_\rho(\sigma)=C_\rho T(F)^2\sigma^{r_\rho}+\cO(\sigma^{r_\rho+1}).
\end{align}
Moreover, if $H^\cdot(Z,F)=0$, we have 
\begin{align}
&C_\rho=1,&r_\rho=0,
\end{align}
so that 
\begin{align}
R_\rho(0)=T(F)^2.
\end{align}
\end{thm}

The proof of Theorem \ref{thm:3} is  similar to the one  given in \cite{Shfried}, except that in the current  case, we also need to take account of the contribution of elliptic orbital integrals in the analytic torsion. On the other hand, let us note that a  priori elliptic  elements do not contribute to the  dynamical zeta function. This seemingly contradictory phenomenon has already appeared in the smooth case. In fact, 
in the current case, the elliptic and non elliptic orbital integrals  are related  via functional equations of certain Selberg zeta functions. 

We refer the readers to the papers of  Giulietti-Liverani-Pollicott \cite{GLP2013} and Dyatlov-Zworski \cite{DyatlovZworski,zworski_zero} for other points of view on the dynamical zeta function on negatively curved manifolds. 

Let us also mention  Fedosova's recent work \cite{Fedosova_orb_zeta, Fedosova_orb_hyper,Fedosova_orb_finite} on the Selberg zeta function and  the asymptotic behavior of the analytic torsion of unimodular flat orbifold vector bundles on   hyperbolic orbifolds.

\subsection{Organisation of the article}
This article  is organized as follows.
In Section \ref{Sec:1}, we introduce some basic notation  on the determinant line and characteristic forms. Also we recall some standard terminology  on  group actions on topological spaces.

In Section \ref{Sec:Top}, we recall the definition of orbifolds,   orbifold vector bundles, and the $\cG$-path theory of Haefliger \cite{Haefliger_orb}.   We show Theorem \ref{thm:1}.

In Section \ref{Sec:diff}, we explain  how to extend the usual differential calculus and Chern-Weil theory on manifolds to orbifolds. 

In Section \ref{Sec:Tor}, we  study    the analytic torsion and 
Ray-Singer metric  on orbifolds.   Following \cite{BLott}, we show in 
a unified way an orbifold version of  Gauss-Bonnet-Chern Theorem and Theorem \ref{thm:2}. Some estimates on heat kernels are postponed to Section \ref{Sec:estheat}.

In Section \ref{Sec:Fried}, we study the analytic torsion on locally symmetric orbifold  using the Selberg trace formula and Bismut's semisimple orbital integral formula. We show Theorem \ref{thm:3}.

\subsection{Notation}
In the whole paper, we use the superconnection formalism of Quillen 
\cite{Quillensuper} (see also \cite[Section 1.3]{BGV}).
 Here we just briefly recall that if $A$ is a $\bZ_2$-graded algebra, 
 if $a,b\in A$,  the supercommutator $[a, b]$ is 
given by $[a,b]=ab-(-1)^{\deg a \deg b}ba$. 
If $B$ is another $\bZ_2$-graded algebra, we denote by $A\widehat{\otimes} B$ the super tensor algebra.
If $E = E^+ \oplus E^-$ is a $\bZ_2$-graded vector space, the algebra $\End(E)$ is $\bZ_2$-graded. If $\tau = \pm 1$ on
$E^\pm$, if $a \in \End(E)$, the supertrace $\Trs[a]$ is defined by 
$  \Trs[a]=\Tr[\tau a]$.
We make the convention that $\mathbf{N}=\{0,1, 2,\cdots\}$ and $\mathbf{N}^*=\{1,2,\cdots\}$.  If $A$ is a finite set, we denote by $|A|$ its cardinality. 
\subsection*{Acknowledgement}
The work  started while S.S. was visiting the
University of Science and Technology of China in July, 2016.
He would  like to thank the School of Mathematical Sciences for hospitality.
S.S. was  supported by a grant from the European Research Council (E.R.C.) under European Union's Seventh Framework Program (FP7/2007-2013) /ERC grant agreement (No.~291060) and by Collaborative Research Centre ``Space-Time-Matte" (SFB 647) founded by the German Research Foundation (DFG).
J.Y. was partially
supported by NSFC (No.~11401552, No.~11771411). 

\settocdepth{subsection}

\section{Preliminary}\label{Sec:1}
The purpose of this section is to recall some basic definitions and constructions.
%
This section is organized as follows.
In subsection \ref{sec:det}, we  introduce the basic conventions on  determinant lines.

In subsection \ref{sec:group}, we recall some standard terminology 
 of group actions on topological spaces, which will be used in the whole paper.

In subsection \ref{sec:chm}, we recall the  Chern-Weil construction 
on  characteristic forms and  the associated secondary classes of Chern-Simons forms on manifolds. 
\subsection{Determinants}\label{sec:det}
Let $V$ be a complex  vector space of finite dimension. We denote by $V^*$ the dual space of $V$, and by $\Lambda^\cdot V$ the exterior algebra of $V$.
 Set $  \det V=\Lambda^{\dim V} V.$
 Clearly, $\det V$ is a  line. We use the convention that $\det 0=\bC.$
If  $\lambda$  is a  line, we denote by $\lambda^{-1}=\lambda^*$ the dual line.
%
%


\subsection{Group actions}\label{sec:group}
Let $L$ be a topological  group  acting continuously on a topological 
space $S$. The action of $L$ is  said to be  free if for any $g\in L$ and $g\neq 1$, the set of fixed points of $g$ in $S$ is empty. The action of $L$  is said to be effective if the morphism of groups $L\to {\rm Homeo}(S)$ is injective, where ${\rm Homeo}(S)$ is the group of homeomorphisms of $S$. The action of $L$ is said to be properly discontinuous if for any  $x\in S$ there is a neighborhood $U$ of $x$ such that the set 
\begin{align}
\{g\in L: gU\cap U\neq \varnothing\}
\end{align}
is finite.

If $L$ acts on the right (resp.~left) on the topological space $S_0$ (resp.~ $S_1$), denote by $S_0/L$ (resp.~ $L\backslash S_1$) the quotient space, and by $S_0\times_{L} S_1$ the quotient of $S_0\times S_1$ by the left action defined by 
\begin{align}
&g(x_0,x_1)=(x_0g^{-1},gx_1),& \hbox{for } g\in L \hbox{ and } (x_0,x_1)\in S_0\times S_1.
\end{align}
If $S_2$ is another left $L$-space, denote by  $S_1 \tensor[_L]{\times}{} S_2$ the quotient of $S_1\times S_2$ by the evident left action of $L$. 

\subsection{Characteristic forms on manifolds}\label{sec:chm}
Let $S$ be a  manifold. Denote by  $\big(\Omega^\cdot(S),d^S\big)$  
the de Rham complex of $S$, and by  $H^\cdot(S)$  its de Rham cohomology.


Let $E$ be a real  vector bundle of rank $r$ equipped  with a  Euclidean metric $g^E$. Let $\nabla^E$ be a metric connection, and  let  $R^E=(\nabla^E)^2$ be the curvature of $\nabla^E$. Then $R^E$ is a $2$-form on $S$ with values in  antisymmetric endomorphisms of $E$.               

If $A$ is an antisymmetric matrix, denote by ${\rm Pf}[A]$ the Pfaffian \cite[(3.3)]{BZ92} of $A$. Then ${\rm Pf}[A]$ is a polynomial function of $A$, which is a square root of $\det[A]$. Let $o(E)$ be the orientation line of $E$. The Euler  form  of $\(E,\nabla^{E}\)$ is given by
\begin{align}\label{eq:eE}
  e\(E,\nabla^{E}\)=\mathrm{Pf}\[\frac{R^{E}}{2\pi}\]\in \Omega^{r}\big(S,o(E)\big).
\end{align}

The cohomology class $e(E)\in H^r(S,o(E))$ of $e(E,\nabla^E)$ does not depend on the choice of $(g^E,\nabla^E)$. More precisely,  if $g^{\prime E}$ is another metric on $E$, and if $\nabla^{\prime E}$ is another connection on $E$ which preserves 
$g^{\prime E}$, we can define a class of Chern-Simons $(r-1$)-form 
\begin{align}
\widetilde{e}\(E,\nabla^{E}, \nabla^{\prime E}\)\in \Omega^{r-1}\big(S,o(E)\big)/d\Omega^{r-2}\big(S,o(E)\big)
\end{align}
 such that 
\begin{align}\label{eq:wie}
d^S\widetilde{e}\(E,\nabla^{E}, \nabla^{\prime E}\)=e\(E,\nabla^{\prime E}\)-e\(E,\nabla^{E}\).
\end{align}
Note that  if $r$ is odd, then $e(E,\nabla^{E})=0$ and 
$\widetilde{e}(E,\nabla^{E}, \nabla^{\prime E})=0$. 

Let us describe the construction of $\widetilde{e}(E,\nabla^{E}, \nabla^{\prime E})$.
 Take a smooth family $(g^{ E}_s,\nabla^{E}_s)_{s\in \bR}$   of metrics and metric connections  such that 
\begin{align}\label{eq:E01}
&\(g^{ E}_0,\nabla^{E}_0\)=\(g^{E},\nabla^{E}\), &\(g^{ E}_1,\nabla^{ E}_1\)=\(g^{\prime E},\nabla^{\prime E}\).
\end{align}
Set
\begin{align}\label{eq:pi}
\pi:\bR\times S\to S.
\end{align}
We equip  $\pi^*E$ with a Euclidean metric $g^{\pi^* E}$ and with a metric connection $\nabla^{\pi^*E} $ defined  by 
\begin{align}\label{eq:pugE}
&g^{\pi^* E}|_{\{s\}\times S}=g^{E}_s,
&\nabla^{\pi^*E} =ds\wedge \(\frac{d}{ds}+\frac{1}{2}g^{ E,-1}_s\frac{d}{ds} g^{ E}_s\)+\nabla^{E}_s.
\end{align}
Write 
\begin{align}\label{eq:aass}
e\(\pi^*E,\nabla^{\pi^*E}\)=e\(E,\nabla^{E}_s\)+ds\wedge \alpha_s\in \Omega^{r}\big(\bR\times S, \pi^*o(E)\big).
\end{align}
Since  $e\(\pi^*E,\nabla^{\pi^*E}\)$ is closed, by \eqref{eq:aass}, for $s\in \bR$, we have
\begin{align}\label{eq:aseu}
\frac{\p}{\p s} e\(E,\nabla^{E}_s\)=d^S  \alpha_s.
\end{align}
Then, $\widetilde{e}(E,\nabla^{E}, \nabla^{\prime E})\in\Omega^{r-1}\big(S,o(E)\big)/d\Omega^{r-2}\big(S,o(E)\big) $ is defined  by  the class of
\begin{align}\label{eq:intalpha}
\int_0^1\alpha_s ds\in \Omega^{r-1}\big(S,o(E)\big).
\end{align}
Note that  $\widetilde{e}(E,\nabla^{E}, \nabla^{\prime E})$
does not depend on the choice of smooth family  $(g^{ 
E}_s,\nabla^{E}_s)_{s\in \bR}$ (c.f. \cite[Proposition 2.7]{BruningMa06}). Also, \eqref{eq:wie} is a consequence 
of  \eqref{eq:E01} and \eqref{eq:aseu}.

Let us recall the definition of  the $\widehat{A}$-form of $(E,\nabla^E)$. For $x\in \bC$, set
\begin{align}\label{eq:defAhat}
  \widehat{A}(x)=\frac{x/2}{\sinh(x/2)}.
\end{align}
The $\widehat{A}$-form  of $(E,\nabla^E)$  is given   by
\begin{align}\label{eq:Ahatg}
\widehat{A}\(E,\nabla^E\)=\[\det\(\widehat{A}\(-\frac{R^{E}}{2i\pi}\)\)\]^{1/2}\in \Omega^\cdot(S).
\end{align}

Let $L$ be a compact Lie group. 
Assume that  $L$ acts fiberwisely and linearly on the vector bundle $E$ over $S$, which preserves $(g^E,\nabla^E)$. 
Take $g\in L$. Assume that $g$ preserves the orientation of $E$. Let $E(g)$ be the subbundle of $E $ defined by the 
fixed points  of $g$. Let  $\pm\theta_1,\cdots,\pm\theta_{s_0}$, $0<\theta_i\le \pi$ be the district nonzero angles of the action of $g$ on $E$. Let $E_{\theta_i}$ be the subbundle of $E$ on which $g$ acts by a rotation of angle $\theta_i$. The subbundles $E(g)$ and $E_{\theta_i}$ are canonically equipped with Euclidean metrics and  metric connections $\nabla^{E(g)},\nabla^{E_{\theta_i}}$.

For $\theta\in \bR-2\pi \bZ$, set 
\begin{align}
\widehat{A}_\theta(x)=\frac{1}{2\sinh\(\frac{x+i\theta}{2}\)}.
\end{align}
Given $\theta_i$, let $\widehat{A}_{\theta_i}(E_{\theta_i},\nabla^{E_{\theta_i}})$ be  the corresponding multiplicative genus. The equivariant  $\widehat{A}$-form of $(E,\nabla^E)$ is given by 
\begin{align}\label{eq:Agpm1}
\widehat{A}_g\(E,\nabla^E\)=\widehat{A}\(E(g),\nabla^{E(g)}\)\prod_{i=1}^{s_0} \widehat{A}_{\theta_i}\(E_{\theta_i},\nabla^{E_{\theta_i}}\)\in \Omega^\cdot(S).
\end{align}

Let $E'$ be  a complex  vector bundle  carrying a  connection $\nabla^{E'}$ with curvature $R^{E'}$. Assume that 
$E'$ is equipped with a fiberwise linear action of $L$, which preserves $\nabla^{E'}$. For $g\in L$, the equivariant 
 Chern character form of $(E',\nabla^{E'})$ is given by 
\begin{align}\label{eq:chg}
  \mathrm{ch}_g\(E',\nabla^{E'}\)=\Tr\[g\exp\(-\frac{R^{E'}}{2i\pi}\)\]\in \Omega^{\rm even}(S).
\end{align}

As before, $\widehat{A}_g(E,\nabla^E)$, $\mathrm{ch}_g(E',\nabla^{E'})$ are closed. Their cohomology classes do not depend on the choice of  connections. The closed forms in \eqref{eq:Agpm1} and $\eqref{eq:chg}$ on $S$ are exactly the ones that appear in the Lefschetz fixed point formula of Atiyah-Bott \cite{AtiyahBott67, AtiyahBott68}. Note that there are questions of signs to be taken care of, because of the need to distinguish between $\theta_i$ and $-\theta_i$. We refer to the above references for more detail.

Let $F$ be a  flat vector bundle on $S$ with flat connection 
$\nabla^F$. Let $g^F$ be a Hermitian metric  on $F$. Assume that $F$ 
is equipped with a fiberwise and linear action of $L$ which preserves 
$\nabla^{F}$ and $g^F$. Following \cite[Definition 4.1]{BZ92}, put
\begin{align}\label{eq:omeF}
\omega\(\nabla^{F},g^{F}\)=\(g^{F}\)^{-1}\nabla^{F}g^{F}.
\end{align}
Then, $\omega(\nabla^{F},g^{F})$ is a $1$-form on $S$ with values in symmetric endomorphisms of $F$. For $x\in \bC$, set 
\begin{align}\label{eq:h}
h(x)=xe^{x^2}.
\end{align}
Following  \cite[Defintion 1.7]{BLott} and \cite[Definition 1.7]{BG01}, for $g\in L$, 
the equivariant  odd Chern character form of $(F,\nabla^F)$ is given by 
\begin{align}\label{eq:defh}
h_g\(\nabla^F,g^F\)=\sqrt{2i\pi}\Tr\[gh\(\frac{\omega\(\nabla^{F},g^{F}\)/2}{\sqrt{2i\pi}}\)\]\in \Omega^{\rm odd}(S).
\end{align}
When $g=1$, we denote by  $h\(\nabla^F,g^F\)=h_1\(\nabla^F,g^F\)$.

By \cite[Theorem 1.11]{BLott} and \cite[Theorem 1.8]{BG01}, we know 
that the cohomology class  $h_g(\nabla^F)\in H^{\rm odd}(S)$ of $h_g(\nabla^F,g^F)$ does not depend on $g^F$.  
If $g^{\prime F}$ is another $L$-invariant Hermitian metric on $F$, we can define the class of Chern-Simons form
$\widetilde{h}_g(\nabla^F,g^F,g^{\prime F})\in\Omega^{\rm even }(S)/d \Omega^{\rm odd}(S)$  such that 
\begin{align}\label{eq:dehgg}
d^S\widetilde{h}_g(\nabla^F,g^F,g^{\prime F})=h_g\(\nabla^{ F},g^{\prime F}\)-h_g\(\nabla^F,g^F\).
\end{align}
More precisely, choose a smooth family of $L$-invariant metrics $(g^F_s)_{s\in \bR}$ such that 
\begin{align}
&g_0^F=g^F,& g_1^F=g^{\prime F}.
\end{align}
Consider the projection $\pi $ defined in \eqref{eq:pi}. Equip 
$\pi^*F$  with the following flat connection and Hermitian metric
\begin{align}\label{eq:pucon}
&\nabla^{\pi^*F}=d^{\bR}+\nabla^F,
&g^{\pi^* F}\Big|_{\{s\}\times S}=g^F_s.
\end{align}
Write
\begin{align}\label{eq:bbss}
h_g\(\nabla^{\pi^* F},g^{\pi^* F}\)=h_g\(\nabla^{F},g^{ F}_s\)+ds \wedge \beta_s\in \Omega^{\rm odd}(\bR\times S).
\end{align}
As \eqref{eq:intalpha},  $\widetilde{h}_g\(\nabla^{F},g^{ F},g^{\prime F}\)\in \Omega^{\rm even }(S)/d \Omega^{\rm odd}(S)$ is defined  by the class of 
\begin{align}\label{eq:intbeta}
\int_0^1\beta_s ds \in \Omega^{\rm even}( S).
\end{align}
By \cite[Theorem 1.11]{BLott} and \cite[Theorem 1.11]{BG01}, $\widetilde{h}_g\(\nabla^{F},g^{ F},g^{\prime F}\)$ does not depend on the choice of the smooth family of metrics $(g^F_s)_{s\in \bR}$. Also,  $\widetilde{h}_g\(\nabla^{F},g^{ F},g^{\prime F}\)$ satisfies \eqref{eq:dehgg}.

\section{Topology of orbifolds}\label{Sec:Top}
The purpose of this section is to introduce some basic definitions and related constructions for orbifolds. We show Theorem \ref{thm:1}, which claims a bijection  between the isomorphism classes of proper flat  orbifold  vector bundles and the equivalent classes of representations of the orbifold fundamental group.

This section is organized as follows. In subsection \ref{sec:Orb},  we recall the definition of orbifolds and the associated groupoid $\cG$.

In subsection \ref{sec:sing},  we introduce  the resolution for the 
singular set of an orbifold.

In subsection \ref{sec:vec}, we recall the definition of   orbifold vector bundles. 

In subsection \ref{sec:cov}, the orbifold fundamental group and the universal covering orbifold are constructed  
using the $\cG$-path theory of Haefliger \cite{Haefliger_orb,Bridson_Haefliger}. 

Finally, in subsection \ref{sec:hol}, we define  the holonomy representation for a  proper flat  orbifold vector bundle. We  restate and show Theorem \ref{thm:1}.

\subsection{Definition of orbifolds}\label{sec:Orb}In this subsection, we recall the definition of orbifolds 
following \cite[Section 1]{SatakeGaussB} and \cite[Section 1.1]{Ruan_Orbifold}.
Let $Z$ be a topological space, and let $U\subset Z$ be a connected open subset of $Z$. Take $m\in \bN$.
\begin{defin}\label{def:chart}
  An $m$-dimensional orbifold chart for  $U$ is given by a triple  $(\widetilde{U}, G_U, \pi_U)$, where
  \begin{itemize}
    \item  $\widetilde{U}\subset \mathbf{R}^m$ is a connected open subset of $\mathbf{R}^m$;
    \item $G_U$ is a finite group acting smoothly and effectively on the left on  $\widetilde{U}$;
    \item  $\pi_U:   \widetilde{U}\to  U$ is a $G_U$-invariant continuous  map 
    which induces a homeomorphism of topological spaces   
\begin{align}
 G_U\backslash\widetilde{U}\simeq U.
\end{align}
  \end{itemize}
\end{defin}

\begin{re}In \cite[Section 1]{SatakeGaussB}, it is assumed that the codimension of the fixed point set of $G_U$ in $\widetilde{U}$ is  bigger than or equal to $2$. In this article,  we do not make this assumption.  
\end{re}

 Let $U\hookrightarrow V$ be an embedding of connected open subsets of $Z$, and let $(\widetilde{U}, G_U, \pi_U)$ and $(\widetilde{V}, G_V, \pi_V)$ be respectively orbifold charts for $U$ and $V$.

\begin{defin}
  An embedding of orbifold charts is a smooth embedding $\phi_{VU}: \widetilde{U}\to \widetilde{V}$ such that the diagram
  \begin{align}
  \begin{aligned}
        \xymatrix{
    \widetilde{U}\ar[r]^{\phi_{VU}} \ar[d]^{\pi_U}  & \widetilde{V}\ar[d]^{\pi_V}\\
     U\ar@{^{(}->}[r]& V
    }
  \end{aligned}
  \end{align}
   commutes.
\end{defin}

We recall the following  proposition. The proof was given by   Satake \cite[Lemmas 1.1, 1.2]{SatakeGaussB} under the assumption that  the codimension of the fixed point set is  bigger than or equal to $2$. For general cases,  see \cite[Appendix]{MoPr_orbifold} for example.

\begin{prop}\label{prop:lift}
Let $\phi_{VU}:(\widetilde{U}, G_U, \pi_U)\hookrightarrow (\widetilde{V}, G_V, \pi_V)$ be an embedding of  orbifold charts. The following statements hold:
\begin{enumerate}
\item if $g\in  G_V$, then $x\in \widetilde{U}\to g\phi_{VU}(x)\in \widetilde{V}$ is another embedding of orbifold charts. Conversely, any embedding of orbifold charts $(\widetilde{U}, G_U, \pi_U)\hookrightarrow (\widetilde{V}, G_V, \pi_V)$ is of  such form;
\item there exists a unique injective morphism $\lambda_{VU}:G_U\to G_V$ of groups such that 
$\phi_{VU}$ is $\lambda_{VU}$-equivariant;
\item if $g\in G_V$ such that $\phi_{VU}(\widetilde{U})\cap g\phi_{VU}(\widetilde{U})\neq \varnothing$, then $g$ is in the image of $\lambda_{VU}$, and so
 $\phi_{VU}(\widetilde{U})= g \phi_{VU}(\widetilde{U})$.
\end{enumerate}
\end{prop}

Let $U_1, U_2\subset Z$ be  two connected open subsets of $Z$ with orbifold charts $(\widetilde{U}_1, G_{U_1}, \pi_{U_1})$ and $(\widetilde{U}_2, G_{U_2}, \pi_{U_2})$.

\begin{defin}
  The orbifold charts  $(\widetilde{U}_1, G_{U_1}, \pi_{U_1})$ and $(\widetilde{U}_2, G_{U_2}, \pi_{U_2})$ are said to be compatible if for any $z\in U_1\cap U_2$, there is an open connected neighborhood $U_0\subset U_1\cap U_2$ of $z$  with orbifold chart $(\widetilde{U}_0, G_{U_0}, \pi_{U_0})$ such that there exist two embeddings   of orbifold charts
  \begin{align}
    &\phi_{U_iU_0}: (\widetilde{U}_0, G_{U_0}, \pi_{U_0})\hookrightarrow (\widetilde{U}_i, G_{U_i}, \pi_{U_i}),   &\hbox{for } i=1,2.
  \end{align}
  The diffeomorphism  $\phi_{U_2U_0}\phi^{-1}_{U_1U_0}:\phi_{U_1U_0}(\widetilde{U}_0)\to \phi_{U_2U_0}(\widetilde{U}_0)$ is  called a coordinate transformation.
\end{defin}

\begin{defin}  An orbifold  atlas on $Z$ consists  of an open connected  cover $\cU=\{U\}$ of $Z$ and   compatible orbifold charts $\widetilde{\cU}=\{(\widetilde{U}, G_{U}, \pi_{U})\}_{U\in \cU}$. 

An orbifold atlas $(\cV,\widetilde{\cV})$ is called a refinement of $(\cU,\widetilde{\cU})$, if $\mathcal{V}$ is a refinement of $\cU$ and if every orbifold chart in $\widetilde{\mathcal{V}}$ has an embedding into some orbifold chart in 
$\widetilde{\cU}$.

Two orbifold  atlases are said to be equivalent if they have a common refinement.

The equivalent class of an orbifold atlas is called an orbifold structure on $Z$. 
\end{defin}

\begin{defin}
  An orbifold is a second countable Hausdorff space  equipped with an 
  orbifold structure.  It said to have dimension $m$,   if all the orbifold charts which define the orbifold structure are of dimension $m$.
\end{defin}

\begin{re}\label{re:exichart}
Let $U, V$ be  two connected open subsets of an orbifold with respectively orbifold charts
$(\widetilde{U}, G_U, \pi_U)$ and $(\widetilde{V}, G_V, \pi_V)$, which  are compatible with the orbifold structure. If $U\subset V$, and if $\widetilde{U}$ is simply connected, then there exists an embedding of orbifold charts $(\widetilde{U}, G_U, \pi_U)\hookrightarrow (\widetilde{V}, G_V, \pi_V)$. 
\end{re}

\begin{re}\label{re:lch}
For any point $z$ of an orbifold, there exists an open connected 
neighborhood  $U_z\subset Z$ of $z$ with a compatible orbifold chart 
$(\widetilde{U}_z,G_z,\pi_z)$ such that $\pi_z^{-1}(z)$ contains only 
one point $x\in \widetilde{U}_z$. Such a chart is called to be  centered at $x$. Clearly,  $x$ is a fixed point of $G_z$. The isomorphism class of the group $G_z$ does not depend on the different choices  of centered orbifold charts, and is called the local group at $z$. 

Moreover,  we can choose $(\widetilde{U}_z,G_z,\pi_z)$ to be a linear chart centered at $0$, which means 
\begin{align}
&\widetilde{U}_z=\bR^m, &x=0\in \bR^m,&&G_z\subset {\rm O}(m).
\end{align}
\end{re}

In the sequel, let $Z$ be an orbifold with orbifold atlas 
$(\cU,\widetilde{\cU})$. We assume that  $\cU$ is countable and  that 
each $\widetilde{U}\in \widetilde{\cU}$ is simply connected.  When we talk of an orbifold chart, we mean the one which is compatible with   $\widetilde{\cU}$.
%
%
%

Let us introduce a groupoid $\cG$ associated to the orbifold $Z$ with orbifold atlas $(\cU,\widetilde{\cU})$. Recall that a groupoid is a category whose morphisms, which are called arrows,  are isomorphisms. We define  $\cG_0$,  the objects of $\cG$, to be the countable disjoint union  of smooth manifold 
\begin{align}\label{eq:Lie}
\cG_0=\coprod_{U\in \cU} \widetilde{U}.
\end{align}
An arrow  $g$ from $x_1\in \cG_0$ to $x_2\in \cG_0$, denoted by $g:x_1\to x_2$, is a  germ of coordinate transformation $g$ defined near $x_1$ such that $g(x_1)=x_2$. We denote by $\cG_1$ the set of  arrows.  This way defines a groupoid $\cG=(\cG_0,\cG_1)$. By \cite[Section 1.4]{Ruan_Orbifold}, $\cG_1$ is equipped with a topology such that $\cG$ is a proper, effective, \'etale Lie groupoid.

For $x_1,x_2\in \cG_0$, we call $x_1$ and $x_2$ in the same orbit if there is an arrow  $g\in \cG_1$ from $x_1$ to $x_2$. 
We denote by $\cG_0/\cG_1$ the orbit space equipped with the quotient  topology. The projection $\pi_U:\widetilde{U}\to U$ induces a homeomorphism of topological spaces
\begin{align}\label{eq:Z=GG}
   \cG_0/{{\cG_1}}\simeq Z.
\end{align}

Let $Y$ and $Z$ be two orbifolds. Following \cite[p. 
361]{Satake_gene_mfd}, we introduce:
\begin{defin}\label{def:fYZ}
  A continuous map $f:Y\to Z$ between orbifolds  is called smooth if for any $y\in Y$, there exist
  \begin{itemize}
\item an open connected neighborhood $U\subset Y$ of $y$,  an open connected neighborhood $V\subset Z$ of $f(y)$ such that $f(U)\subset V$,
\item orbifold charts $(\widetilde{U},G_U,\pi_U)$ and $(\widetilde{V},G_V,\pi_V)$ for $U$ and $V$,
\item a smooth map $\widetilde{f}_U:\widetilde{U}\to \widetilde{V}$
   such that the following diagram
 \begin{align}
  \begin{aligned}
        \xymatrix{
    \widetilde{U}\ar[r]^{\widetilde{f}_U} \ar[d]^{\pi_U}  & \widetilde{V}\ar[d]^{\pi_V}\\
     U\ar[r]^{f|_U}& V
    }
  \end{aligned}
  \end{align}
  commutes.
  \end{itemize}
We denote by $C^\infty(Y,Z)$  the space of smooth maps from $Y$ to $Z$. 
\end{defin}

%

Two orbifolds $Y$ and $Z$ are called isomorphic if there are smooth 
maps $f:Y\to Z$ and $f':Z\to Y$ such that $ff'=\mathrm{id}$ and $f'f=\mathrm{id}$. Clearly, this is the case if $f:Y\to Z$ is a smooth homeomorphism such that each lifting $\widetilde{f}_U$ is a diffeomorphism. Moreover, in this case, by Proposition \ref{prop:lift},  there is an isomorphism of group $\rho_U:G_U\to G_V$ such that $\widetilde{f}_U$ is $\rho_U$-equivariant. Also, any possible lifting has the form $g\widetilde{f}_U$, $g\in G_U$. 

\begin{defin}\label{def:Gact}
An action of Lie group  $L$ on $Z$ is said to be  smooth, if  the action  $L\times Z\to Z$
is smooth.
\end{defin}
%
%
%
%
%
%
%

The following proposition is an extension  of \cite[Proposition 
13.2.1]{Thurston_geo_3_maniflod}.  We include a detailed proof  since some  constructions in the proof will be useful to show Theorem \ref{thm:1}.

\begin{prop}\label{prop:X/G}
Let $\Gamma$ be a discrete group acting smoothly and properly discontinuously  on the left on an orbifold $X$. 
Then  $\Gamma\backslash X$ has a canonical orbifold structure induced from $X$.
\end{prop}
\begin{proof}
Let $p:X\to \Gamma\backslash  X$ be the natural projection. We equip $\Gamma\backslash X$ with the quotient topology.  Since $X$ is Hausdorff and second countable, and since the $\Gamma$-action is properly discontinuous, then $\Gamma\backslash X$ is also Hausdorff and second countable. 

 Take $z\in \Gamma\backslash X$. We choose $x\in X$ such that $p(x)=z$. Set
\begin{align}
\Gamma_{x}=\{\gamma\in \Gamma: \gamma x=x\}.
\end{align}
As the $\Gamma$-action is properly discontinuous, $\Gamma_x$ is a finite group, and there exists an open connected $\Gamma_x$-invariant neighborhood $V_x\subset X$ of $x$ such that for $\gamma\in \Gamma- \Gamma_x$,  
\begin{align}\label{eq:ruv}
\gamma V_x\cap V_x=\varnothing.
\end{align}
Then, $p(V_x)\subset \Gamma\backslash X$  is an open connected neighborhood of $z$. Also, we have 
\begin{align}\label{eq:ruv2}
\Gamma_x\backslash V_x\simeq\Gamma\backslash \Gamma V_x=p(V_x).
\end{align}

By taking $V_x$ small enough, there is an orbifold chart  
$(\widetilde{V}_x, H_x,\pi_x)$  for $V_x$ centered at  
$\widetilde{x}\in \widetilde{V}_x$ (see Remark \ref{re:lch}).  As $\Gamma$ acts smoothly on $X$, we can assume that  homeomorphism of $V_x$ defined by $\gamma_x\in \Gamma_x$ lifts to a local diffeomorphism $
\widetilde{\gamma}_x$ defined near $\widetilde{x}$ such that $\pi_{x}\widetilde{\gamma}_x=\gamma_x\pi_x$ holds near $\widetilde{x}$. 

By Proposition \ref{prop:lift}, the lifting $\widetilde{\gamma}_x$ is not unique, and all possible liftings can be written as $h_x\widetilde{\gamma}_x$ for some $h_x\in H_x$. Let $G_x$ be the group of local diffeomorphism defined near $x$ generated by $\{\widetilde{\gamma}_x\}_{\gamma_x\in \Gamma_x}$ and $ H_x$. Then $G_x$ is a finite group. By choosing $V_x$ small enough and by \eqref{eq:ruv2}, $G_x$ acts on $\widetilde{V}_x$ such that 
%
%
%
\begin{align}\label{eq:Heff}
G_x\backslash \widetilde{V}_x\simeq p(V_x).
\end{align}
Since the $G_x$-action on $ \widetilde{V}_x$ is effective, $(\widetilde{V}_x, G_x,p \circ \pi_x)$ is an orbifold chart of $Z$ for $p(V_x)$. 

The family of open sets $\{p(V_x)\}$ covers $\Gamma\backslash X$. It remains to show that two such  orbifold charts $(\widetilde{V}_{x_1}, G_{x_1},p \circ \pi_{x_1})$ and 
$(\widetilde{V}_{x_2}, G_{x_2},p \circ \pi_{x_2})$
are compatible. Its proof consists of two steps.

In the fist step, we consider the case $x_2=\gamma x_1$ for some $\gamma\in \Gamma$. We can assume that $V_{x_2}=\gamma V_{x_1}$, and that $\gamma|_{V_{x_1}}$ lifts to $\widetilde{\gamma}_{x_1}: \widetilde{V}_{x_1}\to \widetilde{V}_{x_2}$. Then $\widetilde{\gamma}_{x_1}$ defines an isomorphism between  the orbifold charts $(\widetilde{V}_{x_1}, G_{x_1},p \circ \pi_{x_1})$ and 
$(\widetilde{V}_{x_2}, G_{x_2},p \circ \pi_{x_2})$.

In the second step, we consider general $x_1,x_2\in X$ such that $p(V_{x_1})\cap p(V_{x_2})\neq\varnothing$.
Because of the first step, we can assume that $V_{x_1}\cap V_{x_2}\neq \varnothing$. For $x_0\in V_{x_1}\cap V_{x_2}$, take an open connected neighborhood $V_{x_0}\subset V_{x_1}\cap V_{x_2}$ of $x_0$ and an orbifold chart $(\widetilde{V}_{x_0},H_{x_0},\pi_{x_0})$ of $X$   as before. We can  assume that there exist  two embeddings  $\phi_{V_{x_i}V_{x_0}}:(\widetilde{V}_{x_0},H_{x_0},\pi_{x_0})\hookrightarrow (\widetilde{V}_{x_i},H_{x_i},\pi_{x_i})$, for $i=1,2$, of orbifold  charts  of $X$. Then, $\phi_{V_{x_i}V_{x_0}}$ also define  two  embeddings of orbifold  charts of $Z$,
\begin{align}
(\widetilde{V}_{x_0},G_{x_0},p \circ \pi_{x_0})\hookrightarrow (\widetilde{V}_{x_i},G_{x_i},p \circ \pi_{x_i}).
\end{align}

The proof our proposition is completed.
\end{proof}

\begin{re}\label{re:GHR}
By the construction, $H_x$ is a normal subgroup of $G_x$, and $\gamma_x\to \widetilde{\gamma}_x$ induces a surjective morphism of groups
\begin{align}\label{eq:rHG}
\Gamma_x\to G_x/H_x.
\end{align}
 If the action of $\Gamma$ on $X$ is effective,  then \eqref{eq:rHG} 
 is an  isomorphism of groups. Thus, the following  sequence of groups
 \begin{align}\label{eq:GHR}
 1\to H_x\to G_x\to \Gamma_x\to1
 \end{align}
 is exact. 
 \end{re}

\subsection{Singular set of orbifolds}\label{sec:sing}
Let $Z$ be an orbifold with orbifold atlas $(\cU,\widetilde{\cU})$. Put
\begin{align}
&Z_{\rm reg}=\{z\in Z: G_z=\{1\}\},&Z_{\rm sing}=\{z\in Z: G_z\neq\{1\}\}.
\end{align}
Then $Z=Z_{\rm reg}\cup Z_{\rm sing}$. Clearly, $Z_{\rm reg}$ is a smooth manifold. However, $Z_{\rm sing}$ is not necessarily  an orbifold.  Following \cite[Section 1]{Kawasaki_Orb_sign},  we will introduce the orbifold resolution $\Sigma Z$ for 
$Z_{\rm sing}$.

 Let $[G_z]$ be the set of conjugacy classes of $G_z$. Set
\begin{align}
\Sigma Z=\{(z,[g]): z\in Z, [g]\in [G_z]- \{1\}\}.
\end{align}
By \cite[Section 1]{Kawasaki_Orb_sign}, $\Sigma Z$ possess a natural 
orbifold structure. Indeed, take $U\in \cU$ and $(\widetilde{U}, 
\pi_U,G_U)\in \widetilde{\cU}$. For $g\in G_U$, denote by 
$\widetilde{U}^g\subset \widetilde{U}$ the set of fixed points of $g$ 
in $ \widetilde{U}$, and by $Z_{G_U}(g)\subset G_U$ the centralizer 
of $g$ in $G_U$. Clearly, $Z_{G_U}(g)$ acts on $\widetilde{U}^g$, and 
the quotient $Z_{G_U}(g)\backslash \widetilde{U}^g$ depends only on 
the conjugacy class  $[g]\in [G_U]$.  The map $x\in \widetilde{U}^g\to (\pi_U(x),[g])\in \Sigma U$ induces an identification 
\begin{align}\label{eq:idsig}
 \coprod_{[g]\in [G_U]\backslash \{1\}} Z_{G_U}(g)\backslash \widetilde{U}^g\simeq \Sigma U. 
\end{align}
By \eqref{eq:idsig}, we  equip $\Sigma U$ with  the induced topology  and  orbifold structure. The topology and the orbifold structure on $\Sigma Z$ is obtained by gluing $\Sigma U$. We omit the detail. 

We decompose  $\Sigma Z=\coprod_{i=1}^{l_0}Z_i$ following its connected components. If $(z,[g])\in Z_i$, set
\begin{align}\label{eq:mi25}
m_i=\left|\ker\big(Z_{G_U}(g)\to {\rm Diffeo}(\widetilde{U}^g)\big) 
\right|\in \bN^{*}.
\end{align}
By definition, $m_i$ is locally constant, and is called the multiplicity of $Z_i$. 
In the sequel, we write
\begin{align}\label{eq:Z0}
&Z_0=Z,&m_0=1.
\end{align}

\subsection{Orbifold vector bundle}\label{sec:vec}
We recall the definition of  orbifold vector bundles. 


\begin{defin}\label{def:orbbundle}
A complex  orbifold vector bundle $E$ of rank $r$ on $Z$ consists of an orbifold $\cE$, called the total space, and a smooth map $\pi:\cE\to Z$, such that 
\begin{enumerate}
\item there is an orbifold atlas $(\cU,\widetilde{\cU})$ of $Z$ such that  for any $U\in \cU$ and $(\widetilde{U},G_U,\pi_U)\in \widetilde{\cU}$, there exist a finite group $G_U^E$ acting smoothly on $\widetilde{U}$ which induces a surjective morphism of groups $G_U^E\to G_U$, 
  a representation $\rho^E_U: G^E_U\to \GL_r(\bC)$, and a $G^E_U$-invariant continuous map
$\pi^{E}_{U}:\widetilde{U}\times\bC^r \to \pi^{-1}(U)$ 
which induces a homomorphism of topological spaces
\begin{align}\label{eq:loctri}
\widetilde{U}\tensor[_{G^E_{U}}]{\times}{}\bC^r \simeq \pi^{-1}(U);
\end{align}

\item the triple $(\widetilde{U}\times \bC^r, G^E_U,\pi^E_U)$ is a (compatible)  orbifold chart on $\cE$;
 \item if  $U_1,U_2\in \cU$ such that   $U_1\cap U_2\neq 
 \varnothing$, and for any $z\in U_1\cap U_2$, there exist a 
 connected open neighborhood  $U_0\subset U_1\cap U_2$ of $z$ with a 
 simply connected orbifold chart 
 $(\widetilde{U}_0,G_{U_0},\pi_{U_0})$ and the triple   $(G_{U_0}^E, 
 \rho^{E}_{U_0}, \pi^{E}_{U_0})$ such that  (1) and (2) hold, and that 
 the  embeddings of orbifold charts of $\cE$
\begin{align}
&\phi^E_{U_iU_0}: \(\widetilde{U}_0\times \bC^r, 
G^E_{U_0},\pi^E_{U_0}\)\hookrightarrow \(\widetilde{U}_i\times \bC^r, G^E_{U_i},\pi^E_{U_i}\), &\hbox{ for } i=1,2,
\end{align}
  have the following  form 
 \begin{align}
 \label{eq:phiE}
&\phi^E_{U_iU_0}(x,v)=\big(\phi_{U_iU_0}(x),g^E_{U_iU_0}(x) v\big), 
&\hbox{for } (x,v)\in \widetilde{U}_0\times \bC^r,
\end{align}
where $\phi_{U_iU_0}:(\widetilde{U}_0,G_{U_0},\pi_{U_0})\hookrightarrow (\widetilde{U}_i,G_{U_i},\pi_{U_i})$ is an embedding of orbifold charts of $Z$, and  $g^E_{U_iU_0}\in C^\infty\big(\widetilde{U}_0,\GL_r(\bC)\big)$. 

%
%
%
%
%
%
%
\end{enumerate}
The vector bundle $E$ is called proper if the surjective morphism  $G_U^E\to G_U$ is an isomorphism, and is called flat if  $g^E_{U_iU_0}$ can be chosen to be constant. 
%
%
\end{defin}

\begin{re}\label{re:ggE}
The embedding $\phi^E_{U_iU_0}$ exists since $\widetilde{U}_0\times \bC^r$ is simply connected. By Proposition \ref{prop:lift}, it is uniquely  determined by the first component $\phi_{U_iU_0}$ when $E$ is proper. 
\end{re}
%

%

\begin{re}
 We can define the real orbifold vector bundle in an obvious way. 
\end{re}

In the sequel, for $U\in \cU$, we will denote by $\widetilde{E}_U$ the trivial vector bundle of rank $r$ on $\widetilde{U}$, and by $E_U$ the restriction of $E$ to $U$. Their total spaces are given respectively  by 
\begin{align}\label{eq:triloc1}
&\widetilde{\cE}_U=\widetilde{U}\times \bC^r, & \cE_U=\widetilde{U}\tensor[_{G_U}]{\times}{} \bC^r.
\end{align}

Let us identify  the associated groupoid $\cG^E=(\cG^E_0,\cG^E_1)$ for the total space  $\cE$ of a proper orbifold vector bundle $E$. By \eqref{eq:Lie}, the object of $\cG^E$ is given by 
\begin{align}\label{eq:GE0}
\cG^E_0=\coprod_{U\in \cU}\widetilde{\cE}_U=\cG_0\times \bC^r.
\end{align}
 If $g\in \cG_1$ is represented by the germ of the transformation  $\phi_{U_2U_0}\phi_{U_1U_0}^{-1}$, denote by $g^E_*$  the germ of transformation $g^E_{U_2U_0}g^{E,-1}_{U_1U_0}$. 
 By Remark \ref{re:ggE},  $g^E_*$ is uniquely determined by $g$. 
 Thus, if $g$ is an arrow form $x$, and if $v\in \bC^r$, $(g,v)$ 
 defines an arrow from $(x,v)$ to $(gx,g^E_*v)$. This way gives  an identification 
\begin{align}\label{eq:GE1}
\cG^E_1=\cG_1\times\bC^r.
\end{align}

We give some examples of orbifold vector bundles. 

%

%
%

%

%
%

\begin{exa}
The tangent bundle $TZ$ of an orbifold $Z$ is a real proper orbifold vector bundle locally defined by $\{(T\widetilde{U},G_U)\}_{U\in \cU}$. 
\end{exa}

\begin{exa}Assume $Z$ is covered by linear charts $\{(\widetilde{U}, 
	G_U,\pi_U)\}_{U\in \cU}$ (see Remark \ref{re:lch}). The orientation line $o(TZ)$ is  a real proper orbifold  line bundle on $Z$, locally defined by $(\widetilde{U}\times \bR, G_U)$ where the action of $g\in G_U$ is given by 
\begin{align}
g:(x,v)\in \widetilde{U}\times \bR\to (gx,({\rm sign}\det(g))v)\in \widetilde{U}\times \bR.
\end{align}
Clearly, $o(TZ)$ is flat. If $o(TZ)$ is trivial, $Z$ is called orientable.
\end{exa}

\begin{exa}
If $E,F$ are orbifold vector bundles on $Z$, then  $E^*$, $\ol{E}$, 
$\Lambda^\cdot (E)$, $\mathscr{T}(E)=\oplus_{k\in \bN} E^{\otimes k}$ and $E\otimes F$ are  defined in an obvious way.
\end{exa}

\begin{exa}
Let $E$ be an orbifold vector bundle on $Z$. For $U\in \cU$, let 
$V_U\subset \bC^r$ be subspace of $\bC^{r}$ of the fixed points of $\ker(G_U^E\to 
G_U)$. Then $G_U$ acts on $V_U$, and $\{(\widetilde{U}\times V_U, G_U)\}_{U\in \cU}$ defines a proper orbifold vector bundle $E^{\rm pr}$ on $Z$. Clearly, if $E$ is flat, then $E^{\rm pr}$ is also flat. 
\end{exa}

A smooth section of $E$ is defined by a smooth  map $s: Z\to E$ in 
the sense of  
Definition \ref{def:fYZ} such that $\pi\circ s=\mathrm{id}$ and that 
each local lift $\widetilde{s}_{U}$ 
of $s|_{U}$ is $G^{E}_{U}$-invariant. 
	The space of smooth  sections of $E$ is denote  by
$C^\infty(Z,E)$.  
The space of differential forms with values in $E$ is defined by 
$\Omega^{\cdot}(Z,E)=C^{\infty}(Z,\Lambda^{\cdot}(T^{*}Z)\otimes_{\bR} E).
$ For $k\in \bN$, we define $C^{k}(Z,E)$ in a similar way. Also, the space of distributions $\cD'(Z,E)$ of $E$ is defined by the topological 
dual of  $C^\infty(Z,E^{*})$. By definition, we have 
\begin{align}\label{eq:Epr}
C^\infty(Z,E)=C^\infty(Z,E^{\rm pr}).
\end{align}
For this reason, most of results in this 
paper can be extended to non proper flat vector bundles. 

Assume now $E$ is proper. 
By \eqref{eq:Z=GG}, $s\in C^\infty(Z,E)$ can be represented by  $\{s_U\in C^\infty(\widetilde{U},\widetilde{E}_U)^{G_U}\}_{U\in \cU}$ a family of $G_U$-invariant sections such that for any $x_1\in U_1,x_2\in U_2$ and $g\in \cG_1$ from $x_1$ to $x_2$, near $x_1$ we have
\begin{align}\label{eq:gss}
g^*s_{U_2}=s_{U_1}. 
\end{align}
We have the 
similar description for elements of $ C^\infty(Z,E)$ and $D'(Z,E)$.

We call $g^E$  a Hermitian metric on $E$, if $g^{E}$ is a section in 
$C^\infty\big(Z,E^{*}\otimes \ol{E}^{*}\big)$ such that $g^E$ is represented by a family 
$\{g^{\widetilde{E}_{U}}\}_{U\in \widetilde{U}}$ 
of $G_U$-invariant metrics on $\widetilde{E}_U$ such that  
\eqref{eq:gss} holds. If $E$ is the real orbifold  vector bundle $TZ$, $g^{TZ}$ is called a Riemannian 
metric on $Z$.

Two orbifold vector bundles $E$ and $F$ are called isomorphic if 
there is $f\in C^\infty(Z,E^*\otimes F)$ and $g\in 
C^\infty(Z,F^*\otimes E)$ such that $fg=\mathrm{id}$ and 
$gf=\mathrm{id}$.

Let $\Gamma$ be a discrete group acting smoothly and properly discontinuously on an orbifold $X$. Let  $\rho:\Gamma\to \GL_r(\bC)$ be a representation of $\Gamma$. By Proposition \ref{prop:X/G},
$\Gamma\backslash X$ and 
\begin{align}
\cF=X\tensor[_\Gamma]{\times }{}\bC^r
\end{align}
have canonical  orbifold structures. 
The  projection $X\times \bC^r\to X$ descends to a smooth  map 
of orbifolds
\begin{align}\label{eq:F=XC/G}
\pi:\cF\to \Gamma\backslash X.
\end{align}

\begin{prop}\label{prop:flatg}Assume that the action of $\Gamma$ on $X$ is smooth, properly discontinuous and effective. Then \eqref{eq:F=XC/G} defines canonically  a proper flat vector bundle $F$ on $\Gamma\backslash X$. 
\end{prop}
\begin{proof}Recall that $p:X\to \Gamma\backslash X$ is the projection. 
For $x\in X$, we use the same notations $\Gamma_x$, $V_x$, $(\widetilde{V}_x, H_x)$ and $G_x$  as in the proof of Proposition \ref{prop:X/G}. Then, $\Gamma\backslash X$ is covered by 
\begin{align}
p(V_x)\simeq \Gamma_x\backslash V_x\simeq G_x\backslash \widetilde{V}_x.
\end{align}

The stabilizer subgroup  of $\Gamma$ at $(x,0)\in X\times \bC^r$ is $\Gamma_x$. By \eqref{eq:ruv},  if $\gamma\in \Gamma-\Gamma_x$,
\begin{align}
\gamma (V_x\times \bC^r)\cap (V_x\times \bC^r)= \varnothing.
\end{align}
 As in \eqref{eq:ruv2}, we have 
\begin{align}\label{eq:sruv}
\pi^{-1}\big(p(V_x)\big)=\Gamma\backslash \Gamma(V_x\times \bC^r)\simeq V_x \,{}_{\Gamma_x}\!\!\times \bC^r.
\end{align}
Since the action of $\Gamma$ on $X$ is effective, by Remark \ref{re:GHR}, we have a morphism of groups $G_x\to \Gamma_x$.  The group $G_x$ acts on $\bC^r$ via the composition of $G_x\to \Gamma_x$ and $\rho|_{\Gamma_x}:\Gamma_x\to \GL_{r}(\bC)$. Thus,  $G_x$ acts on $\widetilde{V}_x\times \bC^r$ effectively such that\begin{align}
V_x \tensor[_{\Gamma_x}]{\times}{} \bC^r\simeq \widetilde{V}_x \tensor[_{G_x}]{\times}{} \bC^r.
\end{align}
By Proposition \ref{prop:X/G}, $(\widetilde{V}_x\times \bC^r,G_x)$ is an orbifold chart of $\cF$ for $\pi^{-1}(p(V_x))$. 

Take two $(\widetilde{V}_{x_1}\times \bC^r, G_{x_1})$ and $(\widetilde{V}_{x_2}\times \bC^r, G_{x_2})$ orbifold charts of $\cF$. It remains to show the compatibility condition \eqref{eq:phiE}.  We proceed as in the proof of Proposition \ref{prop:X/G}.
If $x_2=\gamma x_1$, then 
\begin{align}\label{eq:xrx}
(x,v)\in \widetilde{V}_{x_1}\times \bC^r\to (\widetilde{\gamma}_{x_1} x,\rho(\gamma)v)\in \widetilde{V}_{x_2}\times \bC^r
\end{align}
defines an isomorphism of orbifold  charts on $\cF$. 

For general $x_1, x_2\in X$, we can assume  that $V_{x_1}\cap V_{x_2}\neq\varnothing$. For $x_0\in V_{x_1}\cap V_{x_2}$, take  $V_{x_0}$,  $\widetilde{V}_{x_0}$ and $\phi_{V_{x_i}V_{x_0}}$ as in the proof of Proposition \ref{prop:X/G}. 
Then
\begin{align}\label{eq:xrx2}
(x,v)\in \widetilde{V}_{x_0}\times \bC^r\to (\phi_{V_{x_i}V_{x_0}}(x),v)\in \widetilde{V}_{x_i}\times \bC^r
\end{align}
define  two embeddings of orbifold charts of $\cF$.  From 
\eqref{eq:xrx} and \eqref{eq:xrx2}, we deduce that \eqref{eq:F=XC/G} 
defines  a flat orbifold vector bundle on $\Gamma\backslash X$. The 
properness is clear from the construction.  The proof of  our proposition  is completed.
\end{proof}

\begin{re}\label{re:repiso}
 Take $A\in \GL_r(\bC)$. Let $\rho_A:\gamma\in \Gamma\to A\rho(\gamma) A^{-1}\in \GL_r(\bC)$ be another representation of $\Gamma$. Then 
 \begin{align}
(x,v)\in  X \times \bC^r\to (x,Av) \in X \times \bC^r  
 \end{align}
descends to an isomorphism between  flat orbifold vector bundles $X \tensor[_{\rho}]{\times}{} \bC^r$ and  $X \tensor[_{\rho_A}]{\times}{} \bC^r$.
\end{re}

\subsection{Orbifold fundamental groups and universal  covering orbifold}\label{sec:cov}
In this subsection, following \cite{Haefliger_orb},  \cite[Section III.$\cG$.3]{Bridson_Haefliger}, we recall the constructions of the orbifold fundamental group and the universal covering orbifold. We assume that the orbifold $Z$ is connected. Let $\cG$ be the groupoid associated with  some orbifold atlas $(\cU, \widetilde{\cU})$.

\begin{defin}\label{def:Gpath}
A continuous $\cG$-path ${c}=({b}_1, \ldots, {b}_k; 
g_0,\ldots,g_{k})$ starting at $x\in \cG_0$ and ending at $y\in \cG_0$ parametrized by $[0,1]$ is given by
\begin{enumerate}
  \item a partition $0=t_0<t_1<\cdots<t_k=1$ of $[0,1]$;
  \item continuous  paths ${b}_i: [t_{i-1},t_{i}] \to \cG_0$, for $1\le i\le k$;
  \item arrows $g_i\in \cG_1$ such that $g_0:x\to{b}_1(0)$, $g_i:{b}_i(t_{i})\to {b}_{i+1}(t_i)$, for $1\le i\le k-1$, and $g_k:{b}_k(1)\to y$.
\end{enumerate}
If $x=y$, we call that ${c}$ is a $\cG$-loop based at $x$.
\end{defin}



 Two  $\cG$-paths 
\begin{align}
&{c}=({b}_1, \cdots, {b}_k; g_0,\cdots,g_{k}),&{c}'=({b}'_1, \cdots, {b}'_{k'}; g'_0,\cdots,g'_{k'}),
\end{align}
such that 
${c}$ ending at $y$ and ${c}'$ starting at $y$ can be composed into a 
$\cG$-path (with a suitable reparametrization  \cite[Section III.$\cG$.3.4]{Bridson_Haefliger}),  
\begin{align}
{c}{c}'=({b}_1, \cdots, {b}_k,{b}'_1, \cdots, {b}'_{k'};g_0,\cdots,g'_0g_{k},\cdots,g'_{k'}).
\end{align}
Also, we can define the inverse of a $\cG$-path in an obvious way. 

\begin{defin}\label{defeqG}
We define an equivalence relation on $\cG$-paths generated by
\begin{enumerate}
  \item subdivision of the partition and adjunction by identity elements of $\cG_1$ on new partition points.
  \item for some $1\le i_0\le k$, replacement of the triple 
  $({b}_{i_0}, g_{i_0-1}, g_{i_0})$ by  the triple $(h{b}_{i_0}, h 
  g_{i_0-1}, g_{i_0} h^{-1})$,  where $h\in \cG_1$ is well-defined 
  near the path ${b}_{i_0}([t_{i_0-1},t_{i_0}])$.
  \end{enumerate}
  The equivalent class of $\cG$-paths is called the path on the 
  orbifold $Z$. 
\end{defin}

\begin{re}\label{rel0}
 If $Z$ is equipped with a Riemannian metric, then the length of a 
 path  on $Z$  represented by the 
	$\cG$-path $c=(b_{1},\ldots,b_{k};g_{0},\ldots,g_{k})$ is defined 
	by the sum 
	of the lengths of $b_{i}$. Clearly, this definition does not 
	depend 	on the choice of the representative $c$. The set of paths 
	on $Z$ 	with length $0$ is just the orbifold $Z\coprod \Sigma Z$. 
\end{re}

\begin{re}\label{rergeoo}
	Following \cite[Section 2.4.2]{GuruprasadHaefiger06}, 	if $Z$ is equipped with a Riemannian metric, a $\cG$-path 
	$c=(b_{1},\ldots,b_{k};g_{0},\ldots,g_{k})$ is called a $\cG$-geodesic, 
	if for all $1\le i\le k$,  $b_{i}$ is a geodesic and if for all $1\le i\le k-1$, 	$g_{i,*}\dot{b}_{i}(t_{i})=\dot{b}_{i+1}(t_{i})$. 
	The geodesic on $Z$ is defined by the equivalence class of the 
	$\cG$-geodesics.	 Similarly, we can define the 	closed 
	geodesic on $Z$ by the equivalent class of the closed geodesic 	
	$\cG$-paths, i.e., a $\cG$-geodesic	
	starting and ending at the same point such that  	
	$g_{0,*}g_{k,*}\dot{b}_{k}(1)=\dot{b}_{0}(0)$.
\end{re}

\begin{defin}
An elementary homotopy between two $\cG$-paths ${c}$ and ${c}'$ is  a 
family, parametrized by $s \in [0,1]$, of $\cG$-paths  ${c}^s = 
({b}^s_1,\cdots,  {b}^s_k; g_0^s,\cdots, g_k^s)$, over the 
subdivisions $0 = t_0^s \le  t_1^s \le  \cdots \le t_k^s = 1$, where 
$t_i^s , {b}^s_i$ and $g_i^s$ depend continuously on the parameter $s$, the elements $g_0^s$ and $g_k^s$ are independent of $s$ and ${c}^0 = {c}, {c}^1 = {c}'.$

\end{defin}

\begin{defin}
Two $\cG$-paths are said to be homotopic (with fixed extremities) if one can pass from the first to the second by equivalences of $\cG$-paths and elementary homotopies.
The homotopy class of a $\cG$-path ${c}$ will be denoted by $[{c}]$. 
\end{defin}

As ordinary paths in topological spaces, the composition and inverse operations of $\cG$-paths
are well-defined for their  homotopy classes.

 \begin{defin}
 Take $x_0\in \cG_0$.    With the operations of composition and inverse of $\cG$-paths, the homotopy classes of $\cG$-loops based at $x_0$ form a group $\pi_1^{\rm orb}(Z,x_0)$ called the orbifold fundamental group. 
 \end{defin}

As $Z$ is connected, any two points of $\cG_0$ can be connected by a $\cG$-path. Thus,
the isomorphic class of the group $\pi_1^{\rm orb}(Z,x_0)$  does not depend on the choice of $x_0$. Also, it depends only on the orbifold structure of $Z$. In the sequel, for simplicity, we denote by $\Gamma=\pi^{\rm orb}_1(Z,x_0)$.

\begin{re}
As a fundamental group of a manifold, $\Gamma$ is countable. 
\end{re}

%
%
%
%
%
%

In the rest of this subsection, following  \cite[Section 
III.$\cG$.3.20]{Bridson_Haefliger}, we will construct the universal 
covering orbifold $X$ of $Z$. 
Let us begin with introducing a groupoid $\widehat{\cG}$.  Assume that  $\cU=\{U_z\}$ and $\widetilde{\cU}=\{(\widetilde{U}_z,G_z,\pi_z)\}$ where all the $\widetilde{U}_z$ are simply connected and are centered at $x\in \widetilde{U}_z$ as in Remark \ref{re:lch}.  Fix $x_0\in \cG_0$ as before.

 Let $\widehat{\cG}_0$ be the space of homotopy classes of  $\cG$-paths starting at $x_0$.  The group $\Gamma$ acts naturally on $\widehat{\cG}_0$ by composition  at the starting point $x_0$.
We denote by 
\begin{align}\label{eq:wp}
\widehat{p}:\widehat{\cG}_0\to \cG_0
\end{align}
the projection sending $[c]\in \widehat{\cG}_0$ to its ending point. Clearly, $\widehat{p}$ is $\Gamma$-invariant. 

 Define a topology and manifold structure on $\widehat{\cG}_0$ as follows.
 For $x_1,x_2\in \widetilde{U}_z$, we denote by 
 $c_{x_1x_2}=(b_{x_1x_2};\mathrm{id},\mathrm{id})$ a $\cG$-path  starting at $x_1$ and ending at  $x_2$, where ${b}_{x_1x_2}$ is a path  in $\widetilde{U}_z$ connecting $x_1$ and $x_2$. Note that since $\widetilde{U}_z$ is simply connected, the homotopy class $[c_{x_1x_2}]$ does not depend on the choice of ${b}_{x_1x_2}$. For each $U_z\in \cU$, we fix a $\cG$-path ${c}_z$ starting at $x_0$ and  ending at $x\in \widetilde{U}_z$. For $a\in {\Gamma}$, set
\begin{align}\label{eq:tVza}
\widetilde{V}_{z,a}=\left\{[{c}]\in \widehat{p}^{-1}(\widetilde{U}_z):\[{c} c_{\widehat{p}([{c}])x}{c}_z^{-1}\]=a\right\}.
\end{align}
By \eqref{eq:wp} and \eqref{eq:tVza}, we have 
\begin{align}\label{eq:pngc}
& \widehat{p}^{-1}(\widetilde{U}_z)=\coprod_{a\in {\Gamma}}\widetilde{V}_{z,a}, &\widehat{\cG}_0=\coprod_{U_z\in \cU,a\in{\Gamma} }\widetilde{V}_{z,a}.
\end{align}
Also,
\begin{align}\label{eq:toma}
\widehat{p}:\widetilde{V}_{z,a}\to \widetilde{U}_z
\end{align}
is a bijection. We equip $\widetilde{V}_{z,a}$ with a topology and a manifold structure  via \eqref{eq:toma}. Clearly, the choice of $c_z$ is irrelevant. 
By \eqref{eq:pngc}, $\widehat{\cG}_0$ is a countable disjoint union of smooth manifolds such that \eqref{eq:wp} is a Galois covering with deck transformation group $\Gamma$. 

If $y\in \cG_0$ and if $g\in \cG_1$ is defined near  $y$, we denote by  
\begin{align}\label{eq:cyg}
c_{y,g}=(b_{y};\mathrm{id}, g)
\end{align}
 the $\cG$-path,  where $b_{y}$ is the constant path at $y$.  Set
\begin{align}
\widehat{\cG}_1=\big\{([c], g)\in \widehat{\cG}_0\times \cG_1: g \hbox{ is defined near } \widehat{p}([c])\in \cG_0 \big\}.
\end{align}
Then, $([c], g)\in \widehat{\cG}_1$ represents an arrow from $[c]$ to $[c][c_{\widehat{p}[c],g}]$.
This  defines a groupoid 
 $\widehat{\cG}=(\widehat{\cG}_0,\widehat{\cG}_1)$. 
 
 Let
 \begin{align}\label{eq:Xuc}
X=\widehat{\cG}_0/{{\widehat{\cG}_1}}
\end{align}
 be the orbit space of $\widehat{\cG}$ equipped with the quotient topology. 
The action of $\Gamma$ on $\widehat{\cG}_0$ descends to an effective and continuous action on $X$.  The projection $\widehat{p}$ descends to  a $\Gamma$-invariant continuous map
\begin{align}\label{eq:wp2}
p:X\to Z.
\end{align}

\begin{thm}\label{thm:uX}Assume that $Z$ is a connected orbifold. Then, 
	the topological space $X$ defined in \eqref{eq:Xuc} is connected and has a canonical orbifold structure such that $\Gamma$ acts smoothly, effectively and  properly discontinuously on $X$. Moreover, \eqref{eq:wp2} induces an isomorphism of orbifolds
\begin{align}\label{eq:Z=X/G}
\Gamma\backslash X\to  Z.
\end{align}
\end{thm}
\begin{proof}Let us begin with showing that the topological space $X$ 
	is connected. Take any $\cG$-path ${c}=({b}_1, \cdots, {b}_k; 
	g_0,\cdots,g_{k})$ starting at $x_{0}$. By our construction of 
	$\widehat{\cG_{1}}$, the images  in $X$ of the $\cG$-paths $c$ and  $({b}_1, \cdots, {b}_k; 
	g_0,\cdots,g_{k-1},\mathrm{id})$ are in the same connected 
	component of $X$. The same holds true for  the $\cG$-paths $({b}_1, 
	\cdots, {b}_{k-1}; 
	g_0,\cdots,g_{k-1})$ and  $({b}_1, 
	\cdots, {b}_{k}; 
	g_0,\cdots,g_{k-1}, \mathrm{id})$. By induction argument, the images in $X$ of $c$ and the constant 
	$\cG$-path at $x_{0}$ are in the same connected component of $X$. So $X$ is connected.

	Let us construct an orbifold atlas on $X$. For $a\in \Gamma$, let  $\pi_{z,a}$ be the composition of continuous maps $\widetilde{V}_{z,a}\hookrightarrow \widehat{\cG}_0\to X.$ Set
\begin{align}
V_{z,a}=\pi_{z,a}(\widetilde{V}_{z,a})\subset X.
\end{align}
 By \eqref{eq:pngc},  
\begin{align}\label{eq:toma2}
p^{-1}(U_z)=\bigcup_{a\in \Gamma} V_{z,a}.
\end{align}
Recall that $x\in \widetilde{U}_z$ and  $\pi_z(x)=z$. By \eqref{eq:cyg}, for $g\in G_z$, $c_{x,g}$ is a $\cG$-loop based at $x$. Set
\begin{align}\label{eq:rz}
r_z:g\in G_z\to [{c}_z c_{x,g} {c}_z^{-1}]\in \Gamma.
\end{align}
Then, $r_z$ is a  morphism of groups. By  \eqref{eq:Xuc} and \eqref{eq:toma2}, 
\begin{align}\label{eq:piUuc}
p^{-1}(U_z)=\coprod_{[a]\in  \Gamma/ \im(r_z)} V_{z,a}. 
\end{align}
Using the fact that $p^{-1}(U_z)$ is open in $X$, we can deduce  that $V_{z,a}$ is open in $X$. 

Put 
\begin{align}\label{eq:Huc}
H_{z}=\ker r_z.
\end{align}
By \eqref{eq:rz} and \eqref{eq:Huc}, $H_z$ acts on $\widetilde{V}_{z,a}$ by 
\begin{align}\label{eq:Hauc}
(g,[c])\in H_z\times \widetilde{V}_{z,a}\to [c][c_{\widehat{p}([c]),g^{-1}}]\in \widetilde{V}_{z,a}.
\end{align}
Then $\pi_{z,a}$ induces a homeomorphism  of topological spaces
\begin{align}\label{eq:49}
 H_{z}\backslash \widetilde{V}_{z,a}\simeq V_{z,a}.
\end{align}
As the $H_z$-action on $\widetilde{V}_{z,a}$ is effective, $(\widetilde{V}_{z,a},H_{z},\pi_{z,a})$ is an orbifold chart for $V_{z,a}$. Moreover, the compatibility of each charts is a consequence of \eqref{eq:toma} and the compatibility of charts in $\widetilde{\cU}$. Hence, $\{(\widetilde{V}_{z,a},H_{z},\pi_{z,a})\}_{U_z\in \cU,a\in \Gamma}$ forms an orbifold atlas on $X$.

As $\Gamma$ is countable, $X$ is second countable. We will show that $X$ is Hausdorff. Indeed, take $y_1,y_2\in X$ and $y_1\neq y_2$.  If $p(y_1)\neq p(y_2)$, as $Z$ is Hausdorff, take respectively open neighborhoods $U_1$ and $U_2$ of $p(y_1)$ and of $p(y_2)$ such that $U_1\cap U_2=\varnothing$. Then $p^{-1}(U_1)\cap p^{-1}(U_2)=\varnothing$.  
Assume $p(y_1)=p(y_2)$. By adding charts in $\cU$, we can assume that there is $U_z\in \cU$ such that 
$p(y_1)=p(y_2)=z$ with orbifold charts $(\widetilde{U}_{z},G_{z},\pi_{z})$ centered at $x$.
Assume $y_1$, $y_2$ are represented by $\cG$-paths $c_1$ and $c_2$ starting at $x_0$ and ending at $x$. For $i=1,2$, set
\begin{align}
a_i=[c_i][c^{-1}_{z_0}]\in \Gamma. 
\end{align}
As $y_1\neq  y_2$, then $[a_1]\neq [a_2]\in \Gamma/\im(r_z)$. Thus,
\begin{align}
&y_1\in V_{z,a_1},&y_2\in V_{z,a_2},&& V_{z,a_1}\cap V_{z,a_2}=\varnothing.
\end{align}

In summary, we have shown that $X$ is an orbifold.

Note that $\gamma V_{z,a}=V_{z,\gamma a}$. By \eqref{eq:piUuc}, the set
\begin{align}\label{eq:51}
\{\gamma\in \Gamma: \gamma V_{z,a}\cap V_{z,a}\neq\varnothing\}=a\im(r_z)a^{-1}\subset \Gamma
\end{align}
is finite. Then the $\Gamma$-action on $X$ is properly discontinuous. As $\Gamma$ acts on $\widehat{\cG}_0$,  the $\Gamma$-action is smooth.

We claim that \eqref{eq:Z=X/G} is homeomorphism of topological space. Indeed, by the construction, 
\eqref{eq:Z=X/G} is injective. It is surjective as  $Z$ is connected. The continuity of the inverse \eqref{eq:Z=X/G}  is a consequence of  \eqref{eq:piUuc}.

The isomorphism of orbifolds between $\Gamma\backslash X$ and $Z$ is a consequence of Proposition \ref{prop:X/G} and 
\eqref{eq:toma}, \eqref{eq:49} and \eqref{eq:51}. The proof our theorem is completed. 
\end{proof}
%
%

\begin{re}\label{re:cover}By \eqref{eq:piUuc}, \eqref{eq:49}, and 
by the covering orbifold  theory of Thurston \cite[Definition 
13.2.2]{Thurston_geo_3_maniflod}, $p:X\to Z$ is a covering orbifold 
of $Z$. Moreover, we can show that for any  connected covering orbifold $p':Y\to Z$, there exists a covering orbifold   $p'':X\to Y$ such that 
the diagram
\begin{align}
\begin{aligned}
        \xymatrix{
    X\ar[dd]_p \ar[dr]^{p''} & \\
    &Y\ar[dl]^{p'}\\
     Z& 
    }
  \end{aligned}\end{align}
commutes. 
For this reason, $X$ is called a universal covering orbifold of $Z$. 
As in the case of the classical covering theory of topological 
spaces, the universal covering orbifold  is unique up to covering isomorphism. Also, $\Gamma$ is isomorphic to  the orbifold deck transformation group of $X$. 
\end{re}

\begin{re}\label{re231}
	If a connected covering orbifold $X'$ of $Z$ has a trivial orbifold 
	fundamental group, then $X'$ is a universal covering orbifold of 
	$Z$. In particular, if $Z$ has a  covering orbifold $X'$, which is 
	a connected simply connected manifold, then $X'$ is a universal covering orbifold of 
	$Z$.
\end{re}

\begin{exa}
	The teardrop  $Z_{n}$ with 
	$n\g2$ (see Figure \ref{F1}) is an example 
	of an orbifold with a trivial orbifold 	fundamental group which 
	is not a manifold (c.f. \cite[p. 304]{Thurston_geo_3_maniflod}). Its underlying topological space is 	a 
	$2$-sphere	$\mathbb{S}^2$, and its singular set   consists of a single 
	point, whose neighbourhood is modelled on 
	$\mathbf{R}^2/(\mathbf{Z}/{n\mathbf{Z}})$, where the cyclic group $\mathbf{Z}/{n\mathbf{Z}}$ acts by rotations.
	\begin{figure}[htbp]
   \includegraphics[width=6cm]{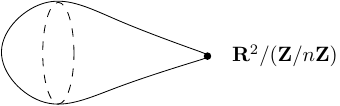}
      \caption{\label{F1}Teardrop $Z_{n}$.}
\end{figure}
\end{exa}

\subsection{Flat vector bundles and holonomy}\label{sec:hol}In this subsection,
we still assume that $Z$ is a connected orbifold. Let $F$ be a proper flat orbifold  vector bundle on $Z$. 
Let $(\cU,\widetilde{\cU})$ be an orbifold atlas as in Definition \ref{def:orbbundle}. Let  $\cG$ be  the associated groupoid. We fix $x_0\in \cG$.

For a $\cG$-path $c=(b_1,\cdots,b_k;g_0,\cdots,g_k)$, the parallel transport $\tau_{c}$ of $F$ along $c$ is defined by 
\begin{align}\label{eq:parat}
\tau_{c}= g_{k,*}^F\cdots g^F_{0,*}\in \GL_r(\bC).
\end{align}
It  depends only  on the homotopy class of $c$. In particular, it defines a representation, called holonomy representation of $F$,  
\begin{align}\label{eq:Ftor}
\rho: {\Gamma}\to \GL_r(\bC).
\end{align}
The isomorphic  class of the representation $\rho$ is independent of the choice of orbifold atlas on $Z$, of the local  trivialization of $F$,  and  of the choice of $x_0$. Moreover, it does not depend on the isomorphic class of $F$.


Let $\Hom(\Gamma,\GL_r(\bC))/_\sim$ be the set of equivalent classes of complex representations of $\Gamma$ of dimension  $r$, and 
let $\cM^{\rm pr}_r(Z)$ be the set of isomorphic classes of proper complex  flat orbifold vector bundles of rank $r$ on $Z$. 
By Proposition \ref{prop:flatg} and Remark \ref{re:repiso}, the map
\begin{align}\label{eq:GR}
\rho\in  \Hom(\Gamma,\GL_r(\bC))/_\sim \to X \,{}_{\rho}\times \bC^r\in \cM^{\rm pr}_r(Z)
\end{align}
is well-defined.

\begin{thm}\label{Thm:1}
The map \eqref{eq:GR} is one-one and onto, whose inverse is given by the holonomy representation \eqref{eq:Ftor}. 
\end{thm}
\begin{proof}
\ul{Step 1. The holonomy representation of 
$X\,{}_{\rho}\!\times\bC^r$ is isomorphic to $\rho$.}  

Assume that the orbifold $Z$ is covered by $\{U_z\}$ with simply connected orbifold charts $\{\widetilde{U}_z\}$ centered at $x\in \widetilde{U}_z$ and $X$ is covered by $\{V_{z,a}\}$ as \eqref{eq:piUuc} such that $p(V_{z,a})= U_z.$

Take $\gamma\in \Gamma$.
Let $c=(b_1,\cdots,b_k;g_0,\cdots,g_k)$ be a $\cG$-loop based at $x_0$ which represents $\gamma$. It is enough to show the parallel transport along $c$ is $\rho(\gamma)$. For $1\le i\le k$, take 
\begin{align}
x_i=b_i(t_{i-1}). 
\end{align}
Up to equivalence relation of $c$ and up to adding charts into the orbifold atlas of $Z$, we can assume that there are orbifold charts $\widetilde{U}_{z_i}$ of $Z$ centered at $x_i$ such that 
$b_i:[t_{i-1},t_{i}]\to \widetilde{U}_{z_i}.$
Also, we assume that $\widetilde{U}_{z_0}$ is an orbifold chart of $Z$ centered at $x_0$.

Let $c_{z_1}=c_{x_0,g_0}$ as in \eqref{eq:cyg}. For $2\le i\le k$, set 
\begin{align}
c_{z_i}=(b_1,\cdots,b_{i-1};g_0,\cdots,g_{i-1}).
\end{align}
By \eqref{eq:tVza}, for $1\le i\le k$, $c_{z_i}$ is a $\cG$-path starting at $x_0$ and ending at $x_i$ such that 
\begin{align}
[c_{z_i}]\in \widetilde{V}_{z_i,1}. 
\end{align}
We claim that for $1\le i\le k$, 
\begin{align}\label{eq:yu}
&V_{z_{i-1},1}\cap V_{z_{i},1}\neq \varnothing,& V_{z_{k},1}\cap V_{z_{0},\gamma}\neq\varnothing.
\end{align}
Indeed,  $c'_{z_i}=(b_1,\cdots,b_{i-1};g_0,\cdots,g_{i-2},\mathrm{id})$ projects to the same element of $X$ as $c_{z_i}$, and 
$[c'_{z_i}]\in \widetilde{V}_{z_{i-1},1}$. Also, $c'_{z_{k+1}}$ projects to the same element of $X$ as $c$. 

%

%
%
%
%
%
%
%
%
%
%
%

Recall that $\gamma V_{z_0,1}=V_{z_0,\gamma}$.
By \eqref{eq:xrx}, \eqref{eq:xrx2} and \eqref{eq:yu}, for $1\le i\le k-1$, we have
\begin{align}
& g_{i,*}=1,&g_{k,*}=\rho(\gamma).
\end{align} 
By \eqref{eq:parat},  the parallel transport along $c$ is $\rho(\gamma)$.

\ul{Step 2. If $F$ has holonomy $\rho$, then $F$ is isomorphic to $X\,{}_{\rho}\!\times\bC^r$.} 

We will construct the bundle isomorphism. By \eqref{eq:GE0} and \eqref{eq:GE1}, the groupoid of the total space $\cF$ is given by $\cG^F=\(\cG_0\times \bC^r,\cG_1\times \bC^r\).$
Let us  construct a universal covering orbifold of $\cF$ by determining  its groupoid $\widehat{\cG}^F$.

Take $(x_0,0),(x_1,u)\in \cG^F_0$.
Let $(c,v)$ be a $\cG^F$-path starting at $(x_0,0)$ and ending at $(x_1,u)$. 
Then there is a partition of $[0,1]$ given by $0=t_0<\cdots<t_k=1$ such that  $c=(b_1,\cdots,b_k;g_0,\cdots,g_k)$
as in Definition \ref{def:Gpath}. Also, $v=(v_1,\cdots,v_k)$, where $v_i:[t_{i-1},t_i]\to \bC^r$ is a continuous path such that $v_1(0)=0$,  $g^F_{k,*}v_{k}(1)=u$ and for $1\le i\le k-1$,
\begin{align}
g^F_{i,*}v_i(t_i)=v_{i+1}(t_i).
\end{align}

Put $w:[0,1]\to \bC^r$ a continuous path such that for $t\in [t_{i-1},t_i]$,
\begin{align}\label{eq:vtow}
w(t)=g^{F,-1}_{0,*}\cdots g^{F,-1}_{i-1,*} v_i(t).
\end{align}

We identify $(c,v)$ with $(c,w)$ via \eqref{eq:vtow}. Then, $(c,v)$ is homotopic to $(c',v')$ if and only if $c,c'$ are homotopic as $\cG$-path and $w,w'$ are  homotopic as ordinary continuous paths in $\bC^r$. 
Since any continuous path $w:[0,1]\to \bC^r$ such that $w(0)=0$ is homotopic to the path $t\in [0,1]\to tw(1)$, we have the identification
\begin{align}\label{eq:pathF}
[c,v]\in \widehat{\cG}_0^F\to ([c],w(1))=([c],\tau^{-1}_cu)\in \widehat{\cG}_0\times \bC^r. 
\end{align}
In particular,  we have an isomorphism of groups
\begin{align}
[c]\in \Gamma\to [(c,0)]\in \pi^{\rm orb}_1(\cF,(x_0,0)),
\end{align}
where $0$ is the constant loop at $0\in \bC^r$.  

In the same way, we identify 
\begin{align}
([c,v],g)\in \widehat{\cG}_1^F\to \big(([c],g), w(1)\big) \in \widehat{\cG}_1\times \bC^r.
\end{align}
We deduce  that $\big(([c],g), w(1)\big)$ represents an arrow from 
$\big([c],w(1)\big)\in \widehat{\cG}^F_0$ to 
$\big([cc_g],w(1)\big)\in \widehat{\cG}_0^F$. Therefore,  the orbit 
space of $\widehat{\cG}^F$, which is also the universal covering orbifold of $\cF$, is given by 
$X\times \bC^r.$

By the identification \eqref{eq:pathF}, the projection \eqref{eq:wp} is given by 
\begin{align}\label{eq:iso}
\widehat{p}_\rho:\big([c],w(1)\big)\in \widehat{\cG}_0\times \bC^r\to \big(\widehat{p}([c]),w(1)\big)\in \cG_0\times \bC^r.
\end{align}
The group $\Gamma$ acts on the left on $\widehat{\cG}_0$, and on the left on $\bC^r$ by $\rho$. As in \eqref{eq:wp}, 
the projection \eqref{eq:iso} is a Galois covering with deck transformation group $\Gamma$. And $\widehat{p}_\rho$ descends to a $\Gamma$-invariant continuous map
\begin{align}
p_\rho:X\times \bC^r\to \cF.
\end{align}

By Theorem \ref{thm:uX} , $p_{\rho}$ induces an isomorphism of orbifolds 
\begin{align}\label{eq:XCrR}
X \tensor[_\rho]{\times}{} \bC^r\simeq \cF.
\end{align}
Using the fact that \eqref{eq:iso} is linear on the $\bC^r$,  we can deduce that \eqref{eq:XCrR} is an isomorphism of  orbifold vector bundles. 
%
%
%
%
%
%
The proof of our theorem is completed. 
\end{proof}

\begin{re}
The properness condition is necessary. Indeed, Theorem \ref{Thm:1} implies that the proper flat vector bundle is trivial on the universal cover. Consider a non trivial
 finite group $H$ acting effectively on $\bC^r$. Then  
$H\backslash \bC^r$ is a non proper orbifold vector bundle over a point. Clearly, it is not trivial. 
\end{re}

\begin{re}
By \eqref{eq:Epr} and Theorem \ref{Thm:1}, we get Corollary \ref{cor:1}. 
\end{re}

\section{Differential calculus on orbifolds}\label{Sec:diff}
The purpose of this section is to explain briefly  how to extend the usual differential 
calculus to orbifolds. To simplify our presentation, we assume that 
the underlying  orbifold is compact. We assume also that  the 
orbifold vector bundles are proper. By \eqref{eq:Epr}, all the 
constructions in this section extend trivially to non proper orbifold vector bundles.

This section is organized as follows.
In subsections \ref{sec:diff}-\ref{sec:dis}, we introduce 
differential operators, integration of differential forms, integral operators and 
Sobolev space on orbifolds.


In subsection \ref{sec:chara},  we explain   Chern-Weil theory for 
the
orbifold vector bundles. The Euler form, odd Chern character form, their Chern-Simons classes, and their canonical extensions to $Z\coprod \Sigma Z$ are constructed in detail.


\subsection{Differential operators on orbifolds}\label{sec:diff}
Let $Z$ be a compact orbifold with atlas $(\cU,\widetilde{\cU})$.
Let  $E$ be a proper orbifold vector bundle on $Z$ such that \eqref{eq:loctri} holds.

A differential operator $D$ of order $p$ is a family 
$\{\widetilde{D}_U:C^\infty(\widetilde{U},\widetilde{E}_{U})\to 
C^\infty(\widetilde{U},\widetilde{E}_{U})\}_{U\in \cU}$ of 
$G_U$-invariant differential operators  of order $p$ such that  if 
$g\in \cG_1$ is an arrow  from $x_1 \in \widetilde{U}_{1}$ to $x_2\in \widetilde{U}_{2}$, then near $x_1$, we have
\begin{align}\label{eq:gdd}
g^* \widetilde{D}_{U_{2}}=\widetilde{D}_{U_{1}}.
\end{align}
If each $\widetilde{D}_U$ is elliptic, then $D$ is called elliptic.

If $s\in C^\infty(Z,E)$ is represented by the family $\{s_U\in C^\infty(\widetilde{U},\widetilde{E}_{U})^{G_U}\}$ such that \eqref{eq:gss} holds. By \eqref{eq:gss} and \eqref{eq:gdd}, $\{\widetilde{D}_U\widetilde{s}_U\in C^\infty(\widetilde{U},\widetilde{E}_{U})^{G_U}\}_{U\in \cU}$ defines a section of $E$, which is denoted by $Ds$. Clearly, $D:C^\infty(Z,E)\to C^\infty(Z,E)$ is a linear operator such that 
\begin{align}
\Supp(Ds)\subset \Supp(s). 
\end{align}
As in the manifold case, 
the differential operator acts naturally on distributions. 




%
%
%

\begin{exa}
	  A connection $\nabla^E$ on $E$ is a first  order differential operator from $C^\infty(Z,E)$ to $\Omega^1(Z,E)$ such that
 $\nabla^E$ is represented by  a family $\{\nabla^{\widetilde{E}_U}\}_{U\in \widetilde{U}}$ of $G_U$-invariant connections on $\widetilde{E}_U$ such that \eqref{eq:gdd} holds. 
The curvature $R^E=(\nabla^E)^2$ is defined as usual. It is a section 
of $\Lambda^2(T^{*}Z)\otimes_{\bR} \End(E)$. As usual, $\nabla^{E}$ 
is called metric with respect to a Hermitian metric $g^{E}$ if
$\nabla^{E}g^{E}=0$. 
\end{exa}

\begin{exa}
	Let $(Z,g^{TZ})$ be a Riemannian orbifold. 
If $g^{TZ}$ is  defined by the family  $\{g^{T\widetilde{U}}\}_{U\in \cU}$ of Riemannian metrics, then the family of  Levi-civita connections on $(\widetilde{U},g^{T\widetilde{U}})$ defines the Levi-civita connection $\nabla^{TZ}$ on  $(Z,g^{TZ})$. 
\end{exa}

\begin{exa}
	Let $F$ be a  flat orbifold  vector bundle on $Z$. The de Rham operator $d^Z: \Omega^\cdot(Z,F)\to \Omega^{\cdot+1}(Z,F)$ is 
a first order differential operator  represented   by the family of de Rham 
operators
\begin{align}
\left\{d^{\widetilde{U}}:\Omega^{\cdot}(\widetilde{U},\bC^r)\to 
\Omega^{\cdot+1}(\widetilde{U},\bC^r)\right\}_{U\in \cU}.
\end{align}
Clearly,  $(d^Z)^2=0$. The  complex  $\(\Omega^{\cdot}(Z,F),d^Z\)$ is 
called the orbifold de Rham complex with values in $F$.  Denote by 
$H^\cdot(Z,F)$ the corresponding cohomology. When $F=\bR$ is the trivial bundle, we denote simply by $\Omega^{\cdot}(Z)$ and $H^\cdot(Z)$.\footnote{
By Satake \cite{Satake_gene_mfd}, $H^\cdot(Z)$ coincides with the 
singular cohomology of the underlying topological space $Z$. In 
general, $H^\cdot(Z,F)$ coincides with the cohomology of the sheaf of 
locally constant sections of $F$.} Clearly, $\nabla^{F}=d^{Z}|_{C^{\infty}(Z,F)}$ defines a 
connection on $F$ with vanishing curvature. As in manifold case, such 
a
connection will be called flat.   We say $F$ is unitarily flat, if 
there exists a Hermitian metric $g^F$ on $F$ such that $\nabla^F 
g^F=0$. Clearly, this is equivalent to say the holonomy representation $\rho$ is unitary. 
\end{exa}

\subsection{Integral operators  on orbifolds}\label{sec:dR}
Since $Z$ is Hausdorff and compact, there exists a (finite) partition of unity subordinate to $\cU$. That
means 
there is a (finite) family  of smooth functions $\{\phi_i\in 
C^\infty_c(Z,[0,1])\}_{i\in I}$  on $Z$  such that  the support 
$\Supp \phi_i$ is contained in some $U_i\in \cU$, 
and that  
\begin{align}\label{eq:sum=1}
\sum_{i\in I} \phi_i=1.
\end{align}
Denote by 
\begin{align}
\widetilde{\phi}_i=\pi^*_{U_i} \big(\phi_i\big)\in C^\infty_c(\widetilde{U}_i)^{G_{U_i}}.
\end{align}

Following \cite[p. 474]{SatakeGaussB}, for $\alpha\in 
\Omega^\cdot\big(Z,o(TZ)\big)$ which is represented  by the invariant 
forms $\{\alpha_{U_{i}}\in 
\Omega^{\cdot}(\widetilde{U}_{i},o(T\widetilde{U}_{i}))^{G_{U_{i}}}\}_{i\in I}$, define
\begin{align}\label{eq:intZ}
\int_{Z}\alpha=\sum_{i\in I}\frac{1}{|G_{U_i}|}\int_{\widetilde{U_i}}\widetilde{\phi}_i\alpha_{U_i}.
\end{align}
By \eqref{eq:sum=1} and \eqref{eq:intZ}, we get:
\begin{prop}
If $\alpha\in \Omega^\cdot\big(Z,o(TZ)\big)$, then $\alpha$ is integrable on $Z_{\rm reg}$ such that
\begin{align}\label{eq:azzs}
\int_{Z}\alpha=\int_{Z_{\rm reg}}\alpha.
\end{align}
\end{prop}

%
From \eqref{eq:azzs}, we see that the definition \eqref{eq:intZ} does 
not depend on the choice of orbifold atlas and the partition of 
unity. Also, we have the orbifold Stokes formula. 

\begin{thm}\label{thm:stokes}The following identity holds: for $\alpha\in \Omega^{\cdot}\big(Z,o(TZ)\big)$,
\begin{align}\label{eq:stokes}
\int_{Z}d^Z\alpha=0. 
\end{align}
\end{thm}

Let us introduce  integral  operators. Let $(E,g^{E})$ be a Euclidean orbifold vector 
bundle on $Z$. Fix a  volume 
form $dv_Z\in \Omega^\cdot(Z,o(TZ))$ of $Z$.  Then, we can define the 
space $L^{2}(Z,E)$ of $L^{2}$-sections in an 
obviously way. By \eqref{eq:azzs}, we have
\begin{align}\label{eq:L2reg}
L^2(Z,E)=L^2(Z_{\rm reg},E_{\rm reg}).
\end{align}
As in manifold case, with the help of $dv_{Z}$, we have the natural embedding 
$C^{\infty}(Z,E)\to \mathcal{D}'(Z,E)$.

By our local description of smooth sections and 
distributions \eqref{eq:gss}, the 
Schwartz kernel theorem still holds for orbifolds. 
That means  for any continuous linear map
$A:C^\infty(Z,E)\to \cD'(Z,E)$,  there exists a unique $p\in \cD'(Z\times Z,E\boxtimes E^*)$ such that for $s_1\in C^\infty(Z,E)$ and $s_2\in C^\infty(Z,E^*)$, we have
\begin{align}\label{eq:schker}
\<As_1,s_2\>=\<p,s_2\otimes s_1\>. 
\end{align}

Assume that $p$ is of class $C^k$ for some $k\in \bN$. Then $A$ is called integral operator.  The restriction of $p$ to the regular part defines a bounded section $p_{\rm reg}\in C^k(Z_{\rm reg}\times Z_{\rm reg},E_{\rm reg}\boxtimes E^*_{\rm reg})$ such that  for $s\in C^\infty(Z,E)$ and $z\in Z_{\rm reg}$, 
\begin{align}
As(z)= \int_{z'\in Z_{\rm reg}} p_{{\rm reg}}(z,z')s(z') dv_{Z_{\rm reg}}.
\end{align}
Using \eqref{eq:L2reg}, $A$ extends uniquely to a bounded operator on 
$L^2(Z,E)$. Moreover, since $p_{{\rm reg}}(z,z')$ is bounded, 
\begin{align}
\int_{(z,z')\in Z_{\rm reg}\times Z_{\rm reg}} |p_{{\rm reg}}(z,z')|^2dv_{Z_{\rm reg}\times Z_{\rm reg}}<\infty.
\end{align}
Then $A$ is in the Hilbert-Schmidt class.  If $A$ is in the trace 
class, then 
\begin{align}\label{eq:TrA}
\Tr\[A\]=\int_{z\in Z_{\rm reg}} \Tr^{E}\[p_{\rm 
reg}(z,z)\]dv_{Z_{\rm reg}}=\int_{z\in Z} \Tr^{E}\[p(z,z)\]dv_{Z}.
\end{align}

Now we give another description of  integral operators.   For any local chart $\widetilde{U}$, there 
is a $G_{U}$-invariant integral operator 
$\widetilde{A}_{U}:C^{\infty}(\widetilde{U},\widetilde{E}_{U})\to 
C^{\infty}(\widetilde{U},\widetilde{E}_{U})$ with integral kernel 
$\widetilde{p}_{U}\in C^{k}(\widetilde{U}\times \widetilde{U}, 
\widetilde{E}_{U}\boxtimes \widetilde{E}_{U}^{*})$ such that if $s\in 
C^{\infty}_{c}(U,E|_{U})$, then $As|_{U}$ is defined by the invariant 
section 
\begin{align}\label{eqloclift19}
	\widetilde{A}_{U}s_{U}(x)=\int_{x'\in \widetilde{U}} 
	\widetilde{p}_{U}(x,x')s_{U}(x')dv_{\widetilde{U}}. 
\end{align}
Then the restriction of the integral kernel $p$ on $U\times U$ is 
represented by the invariant section (see \cite[(2.2)]{Ma_Orbifold_immersion})
\begin{align}\label{eq:res}
\sum_{g\in G_U} g\widetilde{p}_{U}(g^{-1}x,x')\in 
C^k\big(\widetilde{U}\times \widetilde{U},\widetilde{E}_{U}\boxtimes 
\widetilde{E}_{U}^{*}\big)^{G_{U}\times G_{U}}.
\end{align}
If $A$ is in trace class, we have 
\begin{align}\label{eq:TrA19}
	\Tr[A]=\sum_{i\in 
	I}\frac{1}{|G_{U_{i}}|}\sum_{g\in 
	G_{U_{i}}}\int_{x\in \widetilde{U}_{i}} \widetilde{\phi}_{i}(x)
	\Tr\[g\widetilde{p}_{U_{i}}(g^{-1}x,x)\]dv_{\widetilde{U}_{i}}. 
\end{align}

\subsection{Sobolev space on orbifolds}\label{sec:dis}
%

Let $(Z,g^{TZ})$ be a compact Riemannian  orbifold of dimension $m$. Let $\nabla^{TZ}$ be the Levi-civita connection on $TZ$, and let 
$R^{TZ}$ be the corresponding curvature.  Let $(E,g^E)$ be a proper  Hermitian  orbifold vector bundle with connection $\nabla^E$. 
When necessary, we identity $E$ with $\ol{E}^*$ via $g^E$.

%

Denote still by $\nabla^{\mathscr{T}(T^*Z)\otimes_{\bR} E}$ the connection $\mathscr{T}(T^*Z)\otimes_{\bR} E$
induced by $\nabla^{TZ}$ and $\nabla^E$. For $q\in \bN$, take $\cH^q(Z,E)$ to be the Hilbert completion of $C^\infty(Z,E)$ under the norm defined by 
\begin{align}
\|s\|^2_q=\sum_{j=0}^q\int_{Z}\left|\(\nabla^{\mathscr{T}(T^*Z)\otimes_{\bR} E} \)^js(z)\right|^2dv_Z.
\end{align}
 Let $\cH^{-q}(Z,E)$ be the dual of $\cH^q(Z,E)$.  
 If $q\in \bR$, we can define $\cH^q(Z,E)$ by interpolation. As in the case of smooth sections, $s\in\cH^q(Z,E)$ can be represented by the family $\{s_U \in\cH^q(\widetilde{U},\widetilde{E}_U)^{G_U}\}_{U\in \cU}$
 of $G_U$-invariant sections  such that \eqref{eq:gss} holds. 

Using these local descriptions, we have 
\begin{align}\label{eq:HaCD}
&  \bigcap_{q\in \bR} \cH^q(Z,E)=C^\infty(Z,E),& \bigcup_{q\in \bR} \cH^q(Z,E)=\cD'(Z,E).
\end{align}
Moreover, if $q>q'$, we have the compact embedding 
\begin{align}\label{eq:ijj}
\cH^q(Z,E)\hookrightarrow \cH^{q'}(Z,E),
\end{align}
 and if $q\in \bN$ and $q>m/2$, we have the continuous embedding
\begin{align}\label{eq:ijjc}
\cH^q(Z,E)\hookrightarrow C^{q-[m/2]}(Z,E).
\end{align}


\subsection{Characteristic forms on orbifolds}\label{sec:chara}
Assume now $(E,g^{E})$ is a  real Euclidean proper orbifold vector bundle of rank $r$  with a metric connection $\nabla^E$. The Euler form $e(E,\nabla^E)\in \Omega^r\big(Z,o(E)\big)$ is defined by the family of closed forms
$\Big\{e\big(\widetilde{E}_{U},\nabla^{\widetilde{E}_{U}}\big)
\Big\}_{U\in 
\cU}.$
Following \cite[Section 3.3]{SatakeGaussB}, the 
orbifold Euler characteristic number is defined by   
\begin{align}\label{eq:chiorb}
\chi_{\rm orb}(Z)=\int_Z e\(TZ,\nabla^{TZ}\).
\end{align}
If $\nabla^{E'}$ is another metric connection,  the  class of Chern-Simons form 
\begin{align}
\widetilde{e}(E,\nabla^{E}, \nabla^{\prime E})\in \Omega^{r-1}\big(Z,o(E)\big)/d\Omega^{r-2}\big(Z,o(E)\big)
\end{align}
is defined by the family 
$\Big\{\widetilde{e}\big(\widetilde{E}_{U},\nabla^{\widetilde{E}_{U}},\nabla^{\prime\widetilde{E}_{U}}\big)
\Big\}_{U\in \cU}.
$ Clearly, \eqref{eq:wie} still holds true.

Let $(F,\nabla^F)$ be a proper orbifold  flat vector bundle on $Z$ 
with a Hermitian metric $g^F$. 
The odd Chern character $h\(\nabla^F,g^F\)\in \Omega^{\rm odd}(Z)$ of $(F,\nabla^F)$ is defined by the family of closed odd forms 
$
\left\{h\big(\nabla^{\widetilde{F}_U},g^{\widetilde{F}_U}\big)
\right\}_{U\in \cU}.
$
If  $g^{\prime F}$ is another Hermitian metric on $F$, the class of 
 Chern-Simons form  
 \begin{align}
\widetilde{h}(\nabla^F,g^F,g^{\prime F})\in\Omega^{\rm even }(Z)/d \Omega^{\rm odd}(Z)
\end{align}
is defined by the family 
$\left\{h\big(\nabla^{\widetilde{F}_U},g^{\widetilde{F}_U},g^{\prime 
\widetilde{F}_U}\big)
\right\}_{U\in \cU}.
$ As before, \eqref{eq:dehgg} still holds true. 
%
%

The degree $1$-part of $h(\nabla^{F},g^{ F})$ and the degree $0$-part of $\widetilde{h}(\nabla^{F},g^{ F},g^{\prime F})$  will be especially important in the formulation of Theorem \ref{thm:2}. We denote by 
\begin{align}\label{eq:hd=0}
	\begin{aligned}
&\theta\(\nabla^F,g^F\)=2h\(\nabla^{F},g^{ F}\)^{[1]},\\& \widetilde{\theta}\(\nabla^F,g^F,g^{\prime F}\)=2\widetilde{h}\(\nabla^{F},g^{ F},g^{\prime F}\)^{[0]}.
\end{aligned}
\end{align}
By \eqref{eq:omeF}-\eqref{eq:dehgg} and \eqref{eq:hd=0}, we have
\begin{align}\label{eqtran}
	\begin{aligned}
	&\theta\(\nabla^F,g^F\)=\Tr\[\(g^{F}\)^{-1}\nabla^Fg^F\],\\
&d^Z\widetilde{\theta}\(\nabla^F,g^F,g^{\prime F}\)=\theta\(\nabla^F,g^{\prime F}\)-\theta\(\nabla^F,g^{ F}\).
\end{aligned}
\end{align}
Let $ \|\cdot\|_{\det 
F}$ and $ \|\cdot\|^{\prime}_{\det 
F}$ be the metrics on the line bundle $\det F$ induced by the metrics 
$g^{F}$ and $g^{\prime F}$. By \cite[(4.12)]{BZ92}, we have 
\begin{align}
\widetilde{\theta}\(\nabla^F,g^F,g^{\prime 
F}\)=\log\(\frac{ \|\cdot\|^{\prime}_{\det F}}{ \|\cdot\|_{\det 
F}}\)^{2}. 
\end{align}

The odd Chern character  form $h(\nabla^{F},g^{ F})$ and the Chern-Simons class $\widetilde{h}(\nabla^{F},g^{ F},g^{\prime F})$ can be extended   to $Z\coprod\Sigma Z$.  Recall that for $U\in \cU$ and $g\in G_U$, 
$\widetilde{U}^g$ is an  orbifold chart of $Z\coprod\Sigma Z$. The 
restriction of  $(\widetilde{F}_{U},\nabla^{\widetilde{F}_{U}})$ to 
$\widetilde{U}^g$ is a flat vector bundle. The element  $g$ acts 
fiberwisely on $\widetilde{F}_{U}|_{\widetilde{U}^{g}}$ and preserves $\nabla^{\widetilde{F}_{U}}$ and $g^{\widetilde{F}_U}$. The family  
\begin{align}\label{eq:defhi}
\left\{h_g\(\nabla^{\widetilde{F}_U},g^{\widetilde{F}_U}\)\in \Omega^{\rm odd}(\widetilde{U}^g)\right\}_{U\in \cU, g\in G_U}
\end{align}
defines a closed differential form $h_\Sigma(\nabla^F,g^F)\in \Omega^{\rm odd}\(Z\coprod \Sigma Z\)$. 
Denote by $h_i(\nabla^F,g^F)$  the restriction of $h_\Sigma(\nabla^F,g^F)$ to $Z_i\subset Z\coprod\Sigma Z$.
Similarly, we can define 
\begin{align}
	\begin{aligned}
\widetilde{h}_i(\nabla^F,g^F,g^{\prime F})\in \Omega^{\rm even}(Z_i)/ 
\Omega^{\rm odd}(Z_i), \\ \theta_i(\nabla^F,g^F)\in \Omega^{1}(Z_i), 
\\\widetilde{\theta}_i(\nabla^F,g^F,g^{\prime F})\in C^\infty(Z_i).
\end{aligned}
\end{align}

The rank of $F$ can be extended to a locally constant function  on 
$Z\coprod \Sigma Z$ in a similar way. Indeed, the family  
$\{\Tr[\rho^{F}_U(g)]\in C^\infty(\widetilde{U}^g)\}_{U\in \cU,g\in G_U}$ of constant functions defines a locally constant function $\rho$ on $Z\coprod \Sigma Z$. Denote by  $\rho_i$ its value at $Z_i$. Clearly, 
\begin{align}\label{eq:rhoi}
\rho_0=\rk[F].
\end{align}

\section{Ray-Singer metric of orbifolds}\label{Sec:Tor}
In this section,  given metrics $g^{TZ}$ and $g^{F}$ on $TZ$ and $F$, we introduce  the Ray-Singer metric on the determinant of the de Rham cohomology $H^\cdot(Z,F)$. We establish  the anomaly formula for the Ray-Singer metric. In particular, 
when $ Z$ is of odd dimension and orientable, the Ray-Singer metric is a topological invariant. 


In subsection \ref{sec:hodge}, we introduce the Hodge Laplacian 
associated to the metrics $(g^{TZ},g^F)$. We  state Gauss-Bonnet-Chern Theorem for compact orbifolds.

In subsection \ref{sec:torsion}, we construct the analytic torsion 
and  the Ray-Singer metric. We restate the anomaly formula. 

In subsection \ref{sec:torsiontran}, following \cite{BLott}, we 
interpret the analytic torsion as a transgression  of  odd Chern 
forms. We state Theorem 
\ref{prop:tran}, which extends the main result of Bismut-Lott, and 
from which  the anomaly formula follows.

In subsection \ref{Sec:estheat}, we prove Gauss-Bonnet-Chern Theorem 
and Theorem 
\ref{prop:tran} in a unified way. Using an argument due to \cite[p. 
2230]{Ma_Orbifold_immersion}, which is based 
on the finite propagation speeds for the solutions of hyperbolic 
equations \cite[Theorem D.2.1]{MaMa}, we can turn  our problem into a 
local one. Since  the orbifold locally is a quotient of a manifold by some finite group, we can then rely on the results of Bismut-Goette \cite{BG01}, where the authors there consider some  similar problems in the equivariant setting.

\subsection{Hodge Laplacian}\label{sec:hodge}
Let $Z$ be a compact orbifold of dimension $m$, and  let 
$F$ be a 
proper flat orbifold vector bundle of rank $r$ with flat connection 
$\nabla^F$.  Put
\begin{align}
	\begin{aligned}
  \chi_{\rm top}(Z,F)=\sum_{i=0}^m (-1)^i\dim_\bC 
  H^i(Z,F),\\\chi'_{\rm top}(Z,F)=\sum_{i=1}^m (-1)^ii\dim_\bC H^i(Z,F).
  \end{aligned}
\end{align}

Take a Riemannian metric $g^{TZ}$ and a Hermitian metric $g^F$ on 
$F$. We apply the construction of subsection \ref{sec:dis} to the  
Hermitian  orbifold vector bundle 
$E={\Lambda^{\cdot}(T^*Z)\otimes_{\bR} F}$ with the Hermitian  metric 
induced by $g^{TZ}$ and $g^F$, and with the connection 
$\nabla^{{\Lambda^{\cdot}(T^*Z)\otimes_{\bR} F}}$ induced by the 
Levi-Civita connection $\nabla^{TZ}$ and the flat connection $\nabla^F$. 
Let $d^{Z,*}$ be the formal adjoint  of $d^Z$. Put
\begin{align}\label{eq:DiracHodge}
  &D^Z=d^Z+d^{Z,*},&\Box^Z=D^{Z,2}=\[d^Z,d^{Z,*}\].
\end{align}
Then $d^{Z,*}$ is a first order differential operator,  represented  by the family of the formal adjoint $d^{\widetilde{U},*}$  of $d^{\widetilde{U}}$ with respect  to the $L^2$-metric defined by $g^{T\widetilde{U}}$ and $g^{\widetilde{F}_U}$. Also, $\Box^Z$
is  a formally self-adjoint second order elliptic operator acting on  
$\Omega^\cdot(Z,F)$, which is  represented  by the family of Hodge 
Laplacian $\Box^{\widetilde{U}}$ acting on 
$\Omega^\cdot(\widetilde{U},\widetilde{F}_U)$ associated with 
$g^{T\widetilde{U}}$ and $g^{\widetilde{F}_U}$.  Also, the operator 
$\(\Box^Z,\Omega^\cdot(Z,F)\)$ is essentially self-adjoint. And the 
domain of the self-adjoint extension is $\cH^2(Z,{\Lambda^{\cdot}(T^*Z)\otimes_{\bR} F})$.
The following theorem is well-known (e.g., \cite[Proposition 
2.2]{Ma_Orbifold_immersion}, \cite[Proposition 2.1]{Daiyu}). 
\begin{thm}\label{thm:orbhodge}
   The following orthogonal decomposition holds: 
   \begin{align}\label{eq:Hodgeorth}
   \Omega^\cdot(Z,F)=\ker \Box^Z \oplus \im \(d^Z|_{\Omega^\cdot(Z,F)}\) \oplus \im\(d^{Z,*}|_{\Omega^\cdot(Z,F)}\).
   \end{align}
   In particular, we have the canonical identification of the vector spaces
   \begin{align}\label{eq:Hodge}
  \ker \Box^Z\simeq H^\cdot(Z,F).
\end{align}
\end{thm}
In the sequel, we still denote by $\Box^Z$ the  self-adjoint 
extension of the operator $\(\Box^Z,\Omega^\cdot(Z,F)\)$. 
 By \eqref{eq:ijjc}, for $k\gg1$, the operator 
$(1+\Box^Z)^{-k}$ has a continuous kernel. In particular, 
$(1+\Box^Z)^{-k}$ is in the Hilbert-Schmidt class, and $(1+\Box^Z)^{-2k}$ is in the trace class.  By the above argument,  if $f$ lies in the Schwartz space $\cS(\bR)$, then $f(\Box^Z)$ has a smooth kernel, and is in the trace class. For $t>0$, the same statement holds true for the heat operator $\exp(-t\Box^Z)$ of $\Box^Z$.  In this way, most of results on compact manifolds, which have been obtained by the functional calculus of the Hodge Laplacian, still hold true for compact orbifolds.

%

Let $N^{\Lambda^{\cdot}(T^*Z)}$ be the number operator on $\Lambda^\cdot(T^*Z)$.
We write $\Trs[\cdot]=\Tr\[(-1)^{N^{\Lambda^{\cdot}(T^*Z)}}\cdot\]$ for the supertrace. 
By the classical argument of Mckean-Singer formula \cite{MckeanSinger}, we get:
\begin{prop}
For $t>0$,  the following identity holds:
  \begin{align}\label{eq:McSinger}
    \chi_{\rm top}(Z,F)=\Trs\[\exp\(-t\Box^Z\)\].
  \end{align}
\end{prop}

Recall that  $\chi_{\rm orb}(Z)$ and $\rho_i$  are defined in \eqref{eq:chiorb} and \eqref{eq:rhoi}. 

\begin{thm}\label{thm:GBC}
 When $t\to0$, we have 
  \begin{align}\label{eq:GBC}
     \Trs\[\exp\(-t\Box^Z\)\]\to \sum^{l_0}_{i=0}  
	 \rho_i\frac{\chi_{\rm orb}(Z_{i})}{m_i}.  \end{align}
In particular, 
\begin{align}\label{eq:35}
\chi_{\rm top}(Z,F)= \sum^{l_0}_{i=0}\rho_i \frac{\chi_{\rm orb}(Z_i)}{m_i}.  
\end{align}
\end{thm}
\begin{proof}
Equation \eqref{eq:35} is a consequence of \eqref{eq:McSinger} and 
\eqref{eq:GBC}.  The proof of \eqref{eq:GBC} will be given in  
subsection \ref{sec:GBC}. 
\end{proof}

\subsection{Analytic torsion and its anomaly}\label{sec:torsion}
By \eqref{eq:Hodgeorth},  let $P^Z$ be the orthogonal projection onto  $\ker \Box^Z$.   
By the  short time asymptotic expansions of the heat trace \cite[Proposition 2.1]{Ma_Orbifold_immersion}, proceeding as  in \cite[Proposition 9.35]{BGV}, the function
\begin{align}\label{eq:deftheta}
  \theta(s)=-\frac{1}{\Gamma(s)}\int_0^\infty \Trs\[N^{\Lambda^\cdot(T^*Z)}\exp\(-t\, \Box^Z\)\(1-P^Z\)\]t^{s-1}dt
\end{align}
defined on the region $\{s\in\bC: \Re(s)>m/2\}$ is holomorphic, and  has a meromorphic extension to $\bC$ which is holomorphic at $s=0$.

\begin{defin}
  The analytic torsion of $F$ is defined by
  \begin{align}
    T\(F,g^{TZ},g^F\)=\exp\big(\theta'(0)/2\big)>0.
  \end{align}
\end{defin}

\begin{re}\label{re:redet}
The formalism of Voros \cite{Voros} on the regularized determinant of the resolvent of Laplacian  extends to orbifolds trivially, as the  proof  relies only on the  short time asymptotic expansions of the heat trace and on the functional calculus. Thus the weighted product of  zeta regularized determinants 
 \begin{align}
\sigma\to \prod_{i=1}^m \det\(\sigma+\Box^Z|_{\Omega^i(Z,F)}\)^{(-1)^ii}
 \end{align}
 is a  meromorphic function on $\bC$ such that  when $\sigma\to 0$, we have 
\begin{multline}\label{eq:TTT2}
 \prod_{i=1}^m \det\(\sigma+\Box^Z|_{\Omega^i(Z,F)}\)^{(-1)^ii}\\
 =T\(Z,g^{TZ},g^F\)^2\sigma^{\chi'_{\rm top}(Z,F)}+\cO(\sigma^{\chi'_{\rm top}(Z,F)+1}).
\end{multline}
\end{re}

We have a generalization of \cite[Theorem 2.3]{RSTorsion}.
\begin{prop}\label{prop:48}
 If $Z$ is an orientable  even dimensional compact  orbifold and if $F$ is a  unitarily  flat orbifold vector bundle, then for any Riemannian metric  $g^{TZ}$ and any flat Hermitian metric $g^F$,  
\begin{align}\label{eq:nn2=0}
\Trs\[\(N^{\Lambda^\cdot(T^*Z)}-\frac{m}{2}\)\exp\(-t\Box^Z\)\]=0.
\end{align}
In particular, 
 \begin{align}\label{eq:evenT=1}
 T\(F,g^{TZ},g^F\)=1.
 \end{align}
\end{prop}
\begin{proof}
Let $*^Z:\Lambda^\cdot(T^*Z)\to  \Lambda^{m-\cdot}(T^*Z)\otimes o(TZ)$ be the 
Hodge star operator associated to $g^{TZ}$, which is  locally defined 
by the family 
$\{*^{\widetilde{U}}:\Lambda^\cdot(T^*\widetilde{U})\to 
\Lambda^{m-\cdot}(T^*\widetilde{U})\otimes_\bR 
o(T\widetilde{U})\}_{U\in \cU}$. 
Note that $Z$ is orientable, so $o(TZ)$ is trivial. Write   
$\star^Z=*^{Z}\otimes_{\bR} \mathrm{id}_{F}:\Lambda^\cdot(T^*Z)\otimes_{\bR} F\to  \Lambda^{m-\cdot}(T^*Z)\otimes_{\bR} F$. 
Clearly,  we have 
\begin{align}\label{eq:Hodgest}
\star^Z\(N^{\Lambda^\cdot(T^*Z)}-\frac{m}{2}\)\star^{Z,-1}=-\(N^{\Lambda^\cdot(T^*Z)}-\frac{m}{2}\).\end{align}
Since $g^F$ is flat, we have 
\begin{align}\label{eq:Hodgest2}
\star^Z \Box^Z\star^{Z,-1}=\Box^Z.
\end{align}
Note that when $m$ is even, $\star^Z$ is an even isomorphism of 
$\Omega^\cdot(Z,F)$. Hence, by \eqref{eq:Hodgest} and  \eqref{eq:Hodgest2}, we get \eqref{eq:nn2=0}. Equation \eqref{eq:evenT=1} is a consequence of \eqref{eq:nn2=0}. 
\end{proof}

Set
\begin{align}\label{eq:r=detH}
\lambda=\bigotimes_{i=0}^m  \(\det H^i(Z,F)\)^{(-1)^i}.
\end{align}
Then $\lambda$ is a complex line. Let $|\cdot|^{\rm RS,2}_{\lambda}$ be the $L^2$-metric on $\lambda$ induced via \eqref{eq:Hodge}.

\begin{defin}
 The Ray-Singer metric on $\lambda$ is defined by 
 \begin{align}\label{eq:RS}
 \|\cdot\|^{\rm RS}_{\lambda}=T\(F,g^{TZ},g^F\)|\cdot|^{\rm RS}_{\lambda}.
 \end{align}
\end{defin}

Let $(g^{TX},g^F)$ and $(g^{\prime TX},g^{\prime F})$ be two pairs of metrics on $TX$ and $F$. Let $\|\cdot\|^{\rm RS, 2}_\lambda$ and $\|\cdot\|^{\prime {\rm RS}, 2}_\lambda$ be the corresponding Ray-Singer metrics on $\lambda$. We restate Theorem \ref{thm:2}.

\begin{thm}\label{Thm:2}
 The following identity holds:
\begin{multline}\label{eq:anomaly}
\log\(\frac{\|\cdot \|^{\prime \rm RS,2}_{\lambda}}{\|\cdot \|^{\rm RS,2}_{\lambda}}\)
=
\sum_{i=0}^{l_0}
\frac{1}{m_i}
\int_{Z_i}\Big( \widetilde{\theta}_i\(\nabla^F,g^F,g^{\prime F}\)e\(TZ_i,\nabla^{TZ_i}\)\\
-\theta_i\(\nabla^F,g^F\)\widetilde{e}\(TZ_i,\nabla^{TZ_i},\nabla^{\prime TZ_i}\)
\Big).
\end{multline}
\end{thm}
\begin{proof}
	The proof of our theorem will be given in Remark \ref{reproof}. 
\end{proof}

\begin{cor}\label{cor:Ziodd}
If all  the $Z_i$'s are  of  odd dimension, then $\|\cdot\|^{ \mathrm{RS},2}_{\lambda}$ does not depend on $g^{TZ}$ or $g^F$. In particular, this is the case if $Z$ is an orientable odd dimensional orbifold. 
\end{cor}
\begin{proof}
When  $\dim Z_i$ is  odd, $e\(TZ_i,\nabla^{TZ_i}\)=0$ and $\widetilde{e}\(TZ_i,\nabla^{TZ_i},\nabla^{\prime TZ_i}\)=0$. By Theorem \ref{Thm:2}, we get Corollary \ref{cor:Ziodd}.
\end{proof}
%
%


\begin{cor}\label{cor:ana}If   for $0\le i\le l_{0}$, 
\begin{align}\label{eq:asum} 
\chi_{\rm orb}(Z_i)=0,
\end{align}
and if $F$ is  unitarily flat, then $\|\cdot\|^{ RS,2}_{\lambda}$ does not depend on $g^{TZ}$ or on  the flat Hermitian metric $g^F$. 
\end{cor}
\begin{proof}
Take $(g^{TX},g^{F})$ and $(g^{\prime TX}, g^{\prime F})$ two pairs 
of  metrics on $TX$ and $F$ such that $\nabla^Fg^F=0$ and 
$\nabla^Fg^{\prime F}=0$. By \eqref{eq:defhi}, for $0\le i\le l_{0}$, 
\begin{align}\label{eq:ana111}
\theta_i\(\nabla^F,g^F\)=\theta_i\(\nabla^F,g^{\prime F}\)=0.
\end{align}By  \eqref{eqtran} and \eqref{eq:ana111}, $\widetilde{\theta}_i(\nabla^F, g^F, g^{\prime F})$ is closed. It becomes a constant $c_i\in \bC$ as $Z_i$ is connected. Using \eqref{eq:asum}, we get 
\begin{align}\label{eq:ana2}
\int_{Z_i} \widetilde{\theta}_i\(F,g^F,g^{\prime F}\)e\(TZ_i,\nabla^{TZ_i}\)=c_i\int_{Z_i}e\(TZ_i,\nabla^{TZ_i}\)=0.
\end{align}
By \eqref{eq:anomaly},  \eqref{eq:ana111},  and \eqref{eq:ana2}, we get $\|\cdot\|^{\rm{ RS},2}_{\lambda}=\|\cdot\|^{\prime \rm{RS},2}_{\lambda}.$
\end{proof}


\begin{re}
 If $F$ is not proper, we can define the analytic torsion and 
 Ray-Singer metric in the same way. Indeed, we have 
 \begin{align}
	 &  H^{\cdot}(X,F)=H^{\cdot}(X,F^{\rm pr}),
&T\(F,g^{TZ},g^{F}\)=T\(F^{\rm pr},g^{TZ},g^{F^{\rm pr}}\).
 \end{align}
 Also, the Ray-Singer metrics of $F$ and $F^{\rm pr}$ coincides. For 
 this reason, all  the result in this section holds true for non proper flat orbifold  vector bundle. 
 \end{re}

%
%
%
%

\subsection{Analytic torsion as a transgression}\label{sec:torsiontran}
Let $ (g_s^{TZ},g_s^F)_{s\in \bR}$ be a smooth family of metrics on $TZ$ and $F$ such that 
\begin{align}
&(g^{TZ}_s,g^F_s)|_{s=0}=(g^{TZ},g^F),&(g^{TZ}_s,g^F_s)|_{s=1}=(g^{\prime TZ},g^{\prime F}).
\end{align}
Then, $s\in \bR\to \log T(F,g^{TZ}_s,g^{F}_s)$ is a smooth function. 
Following \cite[Section III]{BLott}, we will interpret the analytic torsion 
function $\log 
T(F,g^{TZ}_\cdot,g^{F}_\cdot)$ as a transgression for some odd Chern 
forms associated to certain flat superconnections on $\bR$. 

Recall that $\pi$ is defined in \eqref{eq:pi}, and $g^{\pi^* (TZ)}$, 
$g^{\pi^*F}$ are defined in \eqref{eq:pugE}  with $E=TZ$ 
or $F$. 
Consider now a trivial infinite dimensional  vector bundle $W$ on $\bR $ defined by 
\begin{align}
\bR  \times \Omega^\cdot(Z,F)\to \bR.
\end{align}
Let $g^W$ be a Hermitian 
metric on $W$ such that  $g^W_s $ is the $L^2$-metric  on 
$\Omega^{\cdot}(Z,F)$ induced by  $(g^{TZ}_s,g^{F}_s)$.  Put
\begin{align}\label{eq:Ap}
A'=d^{\bR}+d^Z.
\end{align}
Then $A'$ is a flat superconnection on $W$.  Let $A''$ be the adjoint 
of $A'$ with respect to $g^W$.  For $s\in \bR$, denote by  
$d^{Z,*}_s,\Box^{Z}_{s}$  the corresponding objects for 
$(g^{TZ}_{s},g^{F}_{s})$, and  by  $*^{Z}_s$  the Hodge star operator with respect to $g^{TZ}_s$. Thus,
\begin{align}\label{eq:A''}
A''=d^{\bR }+d^{Z,*}_s+ds\wedge \(g^{F,-1}_s\frac{\p g^F_s}{\p 
s}+*^{Z,-1}_s\frac{\p *^{Z}_s}{\p s}\).
\end{align}
 Set
\begin{align}\label{eq:AB}
&A=\frac{1}{2}(A''+A'),&B=\frac{1}{2}(A''-A').
\end{align}
Then $A$ is a superconnection on $W$, and $B$ is a fibrewise first order elliptic differential operator. The curvature of $A$ is given by 
\begin{align}
A^2=-B^2.
\end{align}
It is a fibrewise second order elliptic differential operator.  

Following \cite[Definition 2.7]{BLott}, we  introduce a deformation of $g^W$. For $t>0$, set 
\begin{align}
g_t^W=t^{N^{\Lambda^\cdot(T^*Z)}}g^W.
\end{align}
Let $A''_t$ be the adjoint of $A'$ with respect to $g_t^W$. Clearly,
\begin{align}\label{eq:546}
A''_t=t^{-N^{\Lambda^\cdot(T^*Z)}} A'' t^{N^{\Lambda^\cdot(T^*Z)}}.
\end{align}
We define $A_t$ and $B_t$ as in \eqref{eq:AB}, i.e.,
\begin{align}\label{eq:547}
&A_t=\frac{1}{2}(A''_t+A'),&B_t=\frac{1}{2}(A''_t-A').
\end{align}

\begin{thm}\label{thm:GBCf}
For $t>0$, we have
\begin{align}\label{eq:GBCf}
\Trs\[\exp\(B_t^2\)\]=\chi_{\rm top}(Z,F).
\end{align}
\end{thm}
\begin{proof}
 Theorem \ref{thm:GBCf} can be proved using the technique of 
 the local  family  index theory as in \cite[Theorem 3.15]{BLott}. 
 Since 
 here the parameter space  $\bR$ is of dimension $1$, we give a short 
 proof.  By  construction, $\Trs\[\exp\(B_t^2\)\]$ is an even form 
 on $\bR$, thus it is a function. By \eqref{eq:Ap}, \eqref{eq:A''}, \eqref{eq:546} and \eqref{eq:547}, we have
 \begin{align}\label{eq:GBCf2}
\Trs\[\exp\(B_t^2\)\]=\Trs\[\exp\(-t\Box^Z_s/4\)\].
 \end{align}
 By  \eqref{eq:McSinger} and \eqref{eq:GBCf2}, we get 
 \eqref{eq:GBCf}.
\end{proof}


Recall that $h$ is defined in \eqref{eq:h}.
Following \cite[(2.22) and (2.23)]{BLott}, for $t>0$, set
\begin{align}\label{eq:utvt}
&u_t=\Trs\[h(B_t)\]\in \Omega^1(\bR),
&v_t=\Trs\[\frac{N^{\Lambda^\cdot(T^*Z)}}{2}h'(B_t)\]\in C^\infty(\bR).
\end{align}
The subspace  $\ker(\Box^Z_s)\subset \Omega^{\cdot}(Z,F)$ defines a 
finite dimensional  subbundle $W_0\subset W$ on $\bR$. By Theorem 
\ref{thm:orbhodge}, the  fiber of $W_0$ is $H^\cdot(Z,F)$. As 
in \cite[Section III.f]{BLott}, we equip $W_0$ with the restricted metric $g^{W_0}$ and the induced connection $\nabla^{W_0}$. Put
\begin{align}
	\begin{aligned}
	&u_{0}=\sum_{i=0}^{l_0}\frac{1}{m_i}\int_{Z_i}e\(\pi^*(TZ_i),\nabla^{\pi^*(TZ_i)}\) h_i\(\nabla^{\pi^*F},g^{\pi^* F}\),\\
	&u_{\infty}=h\(\nabla^{W_0},g^{W_0}\),
	\end{aligned}
\end{align}
and
\begin{align}
	&v_{0}=\frac{m}{4}\chi_{\rm top}(Z,F),
	&v_{\infty}=\frac{1}{2}\chi'_{\rm top}(Z,F).
\end{align}

For a smooth family $\{\alpha_t\}_{t>0}$ of differential forms on 
$\bR$,  we say $\alpha_t=\cO(t)$ if for all the compact $K\subset 
\bR$, and for all  $k\in \bN$, there is $C>0$ such that $\|\alpha_{t}\|_{C^k(K)}\le Ct$.

\begin{thm}\label{prop:tran}
For $t>0$,  $u_t$ is a closed $1$-form on $\bR$ such that  its cohomology class does not 
depend on $t>0$ and that 
the following identity of $1$-forms holds:
\begin{align}\label{eq:tran}
\frac{\p }{\p t} u_{t}=d^{\bR}\(\frac{v_{t}}{t}\).
\end{align}
As $t\to 0$, we have 
\begin{align}\label{eq:at1}
&u_t=u_{0}+\cO\(\sqrt{t}\),
&v_{t}=v_{0}+\cO\(\sqrt{t}\),
\end{align}
as $t\to \infty$,
\begin{align}\label{eq:at2}
&u_t=u_{\infty}+\cO(1/\sqrt{t}), &v_{t}=v_{\infty}+\cO(1/\sqrt{t}).
\end{align}
\end{thm}
\begin{proof}
	By \cite[Theorems 1.8 and 3.20]{BLott}, $u_t$ is  closed and \eqref{eq:tran} holds.  The first 
	equation of \eqref{eq:at1} will be proved in subsection \ref{sec:thmat}. 
The first equation of \eqref{eq:at2} can be proved as   \cite[Theorem 
3.16]{BLott}, whose proof is based on functional calculus. 
Proceeding as \cite[Theorem 3.21]{BLott}, we can show 
the second equations of \eqref{eq:at1} and 
\eqref{eq:at2} as consequence of the first 
equations\footnote{More precisely, we need the corresponding results 
on a larger parametrized space $\bR\times (0,\infty)$. We leave the 
details  to readers. } of \eqref{eq:at1} and 
\eqref{eq:at2}. 
	\end{proof}

\begin{cor}\label{cordT}
 The following identities in $C^{\infty}(\bR)$ and $\Omega^{1}(\bR)$ 
 hold:
\begin{align}\label{eq:tautran}
\log 
T\(F,g^{TZ}_\cdot,g^{F}_\cdot\)=-\int_0^\infty\bigg\{v_{t}-v_{\infty}h'(0)
-\(v_{0}-v_{\infty}\)h'\(\frac{i\sqrt{t}}{2}\)\bigg\}\frac{dt}{t},
\end{align}
and 
\begin{align}\label{eq:dtau}
 u_{0}-u_{\infty}=d^{\bR} \log T\(F,g^{TZ}_\cdot,g^{F}_\cdot\).
\end{align}
\end{cor}
\begin{proof}By Theorem \ref{prop:tran}, proceeding as  \cite[Theorems 3.23 and 3.29]{BLott}, we 
	get	Corollary \ref{cordT}.\end{proof}

\begin{re}\label{reproof}Theorem \ref{Thm:2} is just \eqref{eq:dtau}. Indeed, if 
	$\|\cdot \|^{\rm RS,2}_{\lambda,s}$  denotes the Ray-Singer metric 
 on $\lambda$ associated to  $(g^{TZ}_s,g^F_s)$. By \eqref{eq:RS}, 
 \eqref{eq:dtau}, 
 and using the fact that our base manifold $\bR$ is of dimension $1$, we have
 \begin{multline}\label{eq:dtau2}
ds\wedge \frac{\p}{\p s}\left\{\log\(\|\cdot \|^{\rm 
RS,2}_{\lambda,s}\)\right\}\\
=
 \sum_{i=0}^{l_0}\frac{1}{m_i}\int_{Z_i} e\(\pi^*(TZ_i),\nabla^{\pi^*(TZ_i)}\)\theta_i\(\nabla^{\pi^*F},g^{\pi^*F}\).
\end{multline}
By \eqref{eq:aass} and \eqref{eq:bbss} with $\alpha_{i,s}\in 
\Omega^{\dim Z_i-1}(Z_i,o(TZ_i))$, $\beta_{i,s}^{[0]}\in C^\infty(Z_i)$ defined in an obvious way, we have
\begin{align}\label{eq:74}
\begin{aligned}
e\(\pi^*(TZ_i),\nabla^{\pi^*(TZ_i)}\)&=e\(TZ_i,\nabla^{TZ_i}\)+d^Z\int_0^s\alpha_{i,s} ds+ ds\wedge \alpha_{i,s},\\
\theta_i\(\nabla^{\pi^*F},g^{\pi^*F}\)&=\theta_i\(\nabla^F,g^{F}\)+2d^Z\int_0^s\beta^{[0]}_{i,s} ds+ 2ds \wedge \beta^{[0]}_{i,s}.
\end{aligned}
\end{align}
By \eqref{eq:74},  we have
\begin{multline}\label{eq:744}
\int_{Z_i} e\(\pi^*(TZ_i),\nabla^{\pi^*(TZ_i)}\)\theta_i\(\nabla^{\pi^*F},g^{\pi^*F}\)\\
=ds \int_{Z_i}\( 2 \beta^{[0]}_{i,s}\ e\(TZ_i,\nabla^{\pi^*(TZ)}\)- \theta_i\(\nabla^F,g^{F}\)\wedge\alpha_{i,s}\).
\end{multline}
By integrating \eqref{eq:744} with respect to the variable $s$ form $0$ to $1$, and by \eqref{eq:dtau2}, we get \eqref{eq:anomaly}. 
\end{re}

\subsection{Estimates on heat kernels}\label{Sec:estheat}

\subsubsection{Proof of Theorem \ref{thm:GBC}}\label{sec:GBC}
 We follow \cite[Section 13.2]{BG01}. 
Take $\alpha_0>0$. Let $f,g:\bR\to  [0,1]$ be  smooth even functions such that 
\begin{align}\label{eq:fg}
&f(s)=\left\{
\begin{array}{cr}
1, &|s|\le \alpha_0/2;\\
 0, &|s|\ge \alpha_0,
\end{array}
\right.&g(s)=1-f(s).
\end{align}

\begin{defin}
For $t>0$ and $a\in \bC$, set 
\begin{align}\label{eq:FG}
 &F_t(a)=\int_{\bR}e^{2isa}e^{-s^2}f\(\sqrt{t}s\)\frac{ds}{\sqrt{\pi}},&
 G_t(a)=\int_{\bR}e^{2isa}e^{-s^2}g\(\sqrt{t}s\)\frac{ds}{\sqrt{\pi}}.
\end{align}
\end{defin}
By \eqref{eq:fg} and \eqref{eq:FG},  we get
\begin{align}\label{eq:F+G}
\exp\(-a^2\)=F_t(a)+G_t(a).
\end{align}
Moreover, $F_t$ and $G_t$ are even holomorphic functions, whose restriction to $\bR$ lies in $\cS(\bR)$. By \eqref{eq:FG}, we find that given $m,m'\in \bN$, $c>0$, there exist  $C>0, C'>0$ such that if $t\in (0,1]$, $a\in \bC$, $|\im(a)|\le c$, 
\begin{align}\label{eq:GG}
|a|^m \left|G_t^{(m')}(a)\right|\le C\exp(-C'/t).
\end{align}
There exist uniquely well-defined holomorphic functions $\mathscr{F}_t(a)$ and $\mathscr{G}_t(a)$ such that 
\begin{align}\label{eq:FGT}
&F_t(a)=\mathscr{F}_t(a^2),&G_t(a)=\mathscr{G}_t(a^2). 
\end{align}
By \eqref{eq:F+G} and \eqref{eq:FGT}, we have
\begin{align}\label{eq:13.20}
\exp(-a)=\sF_t(a)+\sG_t(a).
\end{align}

By \eqref{eq:13.20}, we get
\begin{align}\label{eq:cet-1}
\exp\(-t\Box^Z\)=\sF_t\(t\Box^Z\)+\sG_t\(t\Box^Z\).
\end{align}
If $A$ is a bounded operator, let $\|A\|$ be its operator norm. If 
$A$ is in the trace class, let $\|A\|_1=\Tr\[\sqrt{A^{*}A}\]$ be its trace norm. 

\begin{prop}\label{prop:cet}
There exist $c>0$ and $C>0$ such that for $t\in (0,1]$, 
\begin{align}\label{eq:cet}
\left\|\sG_t\(t\Box^Z\)\right\|_1\le Ce^{-c/t}.
\end{align}
In particular, as $t\to 0$, we have
\begin{align}\label{eq:617}
\Trs\[\exp\(-t\Box^Z\)\]=\Trs\[\sF_t\(t\Box^Z\)\]+\cO(e^{-c/t}).
\end{align}
\end{prop}
\begin{proof}
By \eqref{eq:GG} and \eqref{eq:FGT}, for any $k\in \bN$, the operator 
$\(1+\Box^Z\)^k \sG_t\(t\Box^Z\)$ is a bounded operator   such that there exist  $C>0$ and $C'>0$,
\begin{align}\label{eq:cet1}
\left\|\(1+\Box^Z\)^k \sG_t\(t\Box^Z\)\right\|\le C\exp\({-C'/t}\).
\end{align}
Take  $k\in \bN$ big enough such that $(1+\Box^Z)^{-k}$ is of trace class. Then 
\begin{align}\label{eq:cet2}
\left\|\sG_t\(t\Box^Z\)\right\|_{1}\le \left\|(1+\Box^Z)^{-k}\right\|_1\left\|(1+\Box^Z)^k \sG_t\(t\Box^Z\)\right\|.
\end{align}
By \eqref{eq:cet1} and \eqref{eq:cet2}, we get \eqref{eq:cet}.
By \eqref{eq:cet-1} and \eqref{eq:cet}, we get \eqref{eq:617}. 
\end{proof}

Assume that $Z$ is covered by a finite family $\cU=\{U_i\}_{i\in I}$ of  connected open sets  with orbifold atlas $\widetilde{\cU}=\{(\widetilde{U}_i,G_{U_i}, \pi_{U_i})\}_{i\in I}$. Let  $\{\phi_i\}_{i\in I}$ be a partition of unity subordinate to $\{U_i\}_{i\in I}$. 
Let $\widetilde{\sF_t}\(t\Box^Z\)_{U_{i}}(z,z')$ be 
the smooth kernel in the sense \eqref{eqloclift19}. 
By \eqref{eq:TrA19}, we have
\begin{multline}\label{eq:618}
\Trs\[\sF_t\(t\Box^Z\)\]\\
=\sum_{i\in I} \frac{1}{|G_{U_i}|}\sum_{g\in G_{U_i}}\int_{\widetilde{U}_i} \widetilde{\phi}_i(x)\Trs\[g\widetilde{\sF}_t\(t\Box^Z\)_{U_i}(g^{-1}x,x)\]dv_{\widetilde{U}_i}.
\end{multline}

For $i\in I$, and $g\in G_{U_i}$,  let $N_{\widetilde{U}_i^g/\widetilde{U}_i}$ be the normal bundle of $\widetilde{U}_i^g$ in $\widetilde{U}_i$. We identity 
$N_{\widetilde{U}_i^g/\widetilde{U}_i}$ with the orthogonal bundle of $T\widetilde{U}_i^g$ in $T\widetilde{U}_i|_{\widetilde{U}_i^g}$. For $\e_0>0$, set 
\begin{align}
N_{\widetilde{U}_i^g/\widetilde{U}_i,\e_0}=\left\{(y,Y)\in N_{\widetilde{U}_i^g/\widetilde{U}_i}: {\rm dist}\(y, \Supp(\widetilde{\phi}_i)\)<\e_0, |Y|<\e_0\right\}.
\end{align}
Take $\e_0>0$ small enough such that for all  $i\in I$,
\begin{align}
\ol{\{x\in \widetilde{U}_i: {\rm 
dist}(x,\Supp(\widetilde{\phi}_i))< \e_0\} }\subset \widetilde{U}_i,
\end{align}
and such that all  $i\in I$, $g\in G_{U_i}$, the exponential map $(y,Y)\in N_{\widetilde{U}_i^g/\widetilde{U}_i,\e_0}\to \exp_y(Y)\in \widetilde{U}_i$ defines a diffeomorphism from $N_{\widetilde{U}_i^g/\widetilde{U}_i,\e_0}$ onto its image $\widetilde{U}_{i,\e_0,g}\subset \widetilde{U}_i$. Also, there exists $\delta_0>0$ such that for all $i\in I$, $g\in G_{U_i}$ if $x\in \Supp(\widetilde{\phi}_i)$ and ${\rm dist}(g^{-1} x,x)< \delta_0$, then 
\begin{align}\label{eq:64}
x\in \widetilde{U}_{i,\e_0,g}.
\end{align}
Let $dv_{\widetilde{U}^g}$ be the induced Riemannian volume of $\widetilde{U}^g$, and let $dY$ be the induced Lebesgue volume on the fiber of $N_{\widetilde{U}_i^g/\widetilde{U}_i}$. Let $k_i:N_{\widetilde{U}_i^g/\widetilde{U}_i,\e_0}\to \bR^*_+$ be a smooth function such that on $ \widetilde{U}_{i,\e_0,g}$ we have 
\begin{align}\label{eq:dvk}
dv_{\widetilde{U}_i}=k_i(y,Y)dv_{\widetilde{U}^g} dY.
\end{align}
Clearly,  $
k_i(y,0)=1.$

For $x\in \Supp (\widetilde{\phi}_i)$ and $r\in (0,\e_0)$, let $B_x^{\widetilde{U}_i}(r)$ be the geodesic ball of center $x$
and radius $r$. 
Using the result of the finite propagation 
speeds for the solutions of hyperbolic equations on orbifolds 
\cite[Section 2.3]{Ma_Orbifold_immersion} (see also \cite[Theorem D.2.1]{MaMa}),  by taking $\alpha_0<\frac{1}{4}\min \{\delta_0, \e_0\}$, for  $x\in \Supp (\widetilde{\phi}_i)$, we find the support of 
$\widetilde{\sF}_t\(t\Box^Z\)_{U_i}(x,\cdot)$ 
 in $B^{\widetilde{U}_i}_x(4\alpha_0)$. Moreover, 
$\widetilde{\sF}_t\(t\Box^Z\)_{U_i}(x,\cdot)$ depends only on  the 
Hodge Laplacian $\Box^{\widetilde{U}_i}$ acting on 
$\Omega^\cdot(\widetilde{U}_{i},\widetilde{F}_{U_{i}})$.  Using \eqref{eq:64} and \eqref{eq:dvk}, we get
\begin{multline}\label{eq:619}
\int_{\widetilde{U}_i} 
\widetilde{\phi}_i(x)\Trs\[g\widetilde{\sF}_t\(t\Box^Z\)_{U_i}(g^{-1}x,x)\]dv_Z\\=\int_{y\in \widetilde{U}^g_i} dv_{\widetilde{U}^g_i}\int_{Y\in N_{\widetilde{U}^g_i/\widetilde{U}_i,y}, |Y|<\e_0}\widetilde{\phi}_i(y,Y)\\
\Trs\[g\widetilde{\sF}_t\(t\Box^Z\)_{U_i}\big(g^{-1}(y,Y),(y,Y)\big)\]k_i(y,Y)dY.
\end{multline}

Consider an isometric embedding of 
$(\widetilde{U}_i,g^{T\widetilde{U}_i})$  into a compact manifold 
$(X_i,g^{TX_i})$. We extend  the trivial  Hermitian vector bundle 
$(\widetilde{F}_{U_{i}}, g^{\widetilde{F}_{U_i}})$ to a trivial  Hermitian vector bundle $(F_i,g^{F_i})$ on $X_i$. Thus, when restricted on $\widetilde{U}_i$, 
\begin{align}\label{eq:or=man}
\Box^{\widetilde{U}_i}=\Box^{X_i}.
\end{align}
Using the results of the finite propagation speeds for the solutions 
of hyperbolic equations on orbifolds \cite[Section 
2.3]{Ma_Orbifold_immersion} and on manifolds \cite[Theorem D.2.1]{MaMa}, for $x,y\in \widetilde{U}_i$, we have 
\begin{align}\label{eq:wisf}
\widetilde{\phi}_i(x)\widetilde{\sF}_t\(t\Box^Z\)_{U_i}(x,y)=\widetilde{\phi}_i(x)\sF_t\(t\Box^{X_i}\)(x,y).
\end{align}
Recall that for $x\in \widetilde{U}_i$ and $g\in G_{U_i}$, $g:F_{g^{-1}x}\to F_{x}$ is a linear map. In particular,
$g\sF_t\(t\Box^{X_i}\)(g^{-1}x,x)$ is well-defined on $\widetilde{U}_i$.  By \eqref{eq:wisf},  for $y\in \widetilde{U}^g_i$, we have
\begin{multline}\label{eq:623}
\int_{Y\in N_{\widetilde{U}^g_i/\widetilde{U}_i,y}, 
|Y|<\e_0}\!\!\!\!\!\widetilde{\phi}_i(y,Y)\Trs\!\!\[g\widetilde{\sF}_t\(t\Box^Z\)_{U_i}\big(\!g^{-1}(y,Y),(y,Y)\big)\]\!k_i(y,Y)dY
\\
=\int_{Y\in N_{\widetilde{U}^g_i/\widetilde{U}_i,y}, |Y|<\e_0}\!\!\!\!\!\widetilde{\phi}_i(y,Y)\Trs\[g\sF_t\(t\Box^{X_i}\)\big(g^{-1}(y,Y),(y,Y)\big)\]k_i(y,Y)dY\\
=t^{\frac{1}{2}\dim{N_{\widetilde{U}^g_i/\widetilde{U}_i}}}\int_{Y\in N_{\widetilde{U}^g_i/\widetilde{U}_i,y}, \sqrt{t}|Y|<\e_0}\phi_i\(y,\sqrt{t}Y\)\\
\Trs\[g\sF_t\(t\Box^{X_i}\)\big(g^{-1}(y,\sqrt{t}Y),(y,\sqrt{t}Y)\big)\]k_i\(y,\sqrt{t}Y\)dY.
\end{multline}

Let $\[e\(T\widetilde{U}^g,\nabla^{T\widetilde{U}^g}\)\]^{\max}$ be the function defined  on $\widetilde{U}^g$ such that 
\begin{align}\label{eq:emax}
e\(T\widetilde{U}^g,\nabla^{T\widetilde{U}^g}\)=\[e\(T\widetilde{U}^g,\nabla^{T\widetilde{U}^g}\)\]^{\max}dv_{\widetilde{U}^g}.
\end{align}

\begin{thm}\label{thm:BG}
There exist  $c>0$ and $C>0$ such that for any $i\in I$, $g\in G_{U_i}$ and $(y,\sqrt{t}Y)\in N_{\widetilde{U}^g_i/\widetilde{U}_i,\e_0}$, we have
\begin{multline}
t^{\frac{1}{2}\dim{N_{\widetilde{U}^g_i/\widetilde{U}_i}}}\bigg|\widetilde{\phi}_i\(y,\sqrt{t}Y\)k_i\(y,\sqrt{t}Y\)\\
\Trs\[g\sF_t\(t\Box^{X_i}\)\big(g^{-1}(y,\sqrt{t}Y),(y,\sqrt{t}Y)\big)\]
\bigg|
\le C\exp(-c|Y|^2).
\end{multline}
As $t\to 0$, we have
\begin{multline}\label{eq:625}
t^{\frac{1}{2}\dim{N_{\widetilde{U}^g_i/\widetilde{U}_i}}}\int_{Y\in 
N_{\widetilde{U}^g_i/\widetilde{U}_i,y},\sqrt{t}|Y|< 
\e_{0}}\bigg\{\widetilde{\phi}_i(y,\sqrt{t}Y)k_i\(y,\sqrt{t}Y\)\\
\Trs\[g\sF_t\(t\Box^{X_i}\)\big(g^{-1}(y,\sqrt{t}Y),(y,\sqrt{t}Y)\big)\]
dY\bigg\}\\
\to \widetilde{\phi}_i(y,0)\Tr[\rho^{F}_{U_i}(g)]\[e\(T\widetilde{U}^g,\nabla^{T\widetilde{U}^g}\)\]^{\max}. 
\end{multline}
\end{thm}
\begin{proof}
Theorem \ref{thm:BG} is a consequence of \cite[Theorems 13.13 and 13.15]{BG01}.
\end{proof}

\begin{proof}[The end of the proof of Theorem \ref{thm:GBC}]
By \eqref{eq:617}, \eqref{eq:618}, \eqref{eq:619}, \eqref{eq:623}-\eqref{eq:625}, and the  dominated convergence Theorem, we get 
\begin{multline}\label{eq:egbc}
\lim_{t\to 0}\Trs\[\exp\(-t\Box^Z\)\]\\
=\sum_{i\in I} 
\frac{1}{|G_{U_i}|}\sum_{g\in G_{U_i}}\Tr[\rho^{F}_{U_i}(g)]\int_{\widetilde{U}^g_i}\widetilde{\phi}_i(y,0)e\(T\widetilde{U}_i^g,\nabla^{T\widetilde{U}_i^g}\).
\end{multline}
As $\widetilde{\phi}_{i}$ is $G_{U_i}$-invariant,  the integral on the right-hand side of \eqref{eq:egbc}  depends only on the conjugation class of $G_{U_i}$. Thus,
 \begin{multline}\label{eq:egbc1}
 \sum_{i\in I} \frac{1}{|G_{U_i}|}\sum_{g\in 
 G_{U_i}}\Tr[\rho^{F}_{U_i}(g)]\int_{\widetilde{U}^g_i}\widetilde{\phi}_i(y,0)e\(T\widetilde{U}_{i}^g,\nabla^{T\widetilde{U}^g_{i}}\)\\
 =\sum_{i\in I} \sum_{[g]\in 
 [G_{U_i}]}\frac{\Tr[\rho^{F}_{U_i}(g)]}{|Z_{G_{U_i}}(g)|}\int_{\widetilde{U}^g_i}\widetilde{\phi}_i(y,0)e\(T\widetilde{U}_{i}^g,\nabla^{T\widetilde{U}^g_{i}}\)
 =\sum_{i=0}^{l_{0}}\rho_i\frac{\chi_{\rm orb}(Z_i)}{m_i}.
 \end{multline}
By \eqref{eq:egbc} and \eqref{eq:egbc1}, we get \eqref{eq:GBC}. 
\end{proof}

\subsubsection{The end of the proof of Theorem 
\ref{prop:tran}}\label{sec:thmat}It remains to show the first 
identity of \eqref{eq:at1}. 
Following \cite[p. 68]{BG01},  we introduce a new Grassmann variable 
$\bf z$ which is anticommute  with $ds$. For  two operators $P,Q$ in 
the  trace class, set 
\begin{align}\label{eq:trz}
\Tr^{\bf z}[P+{\bf z}Q]=\Tr[Q].
\end{align}
By \eqref{eq:h}, \eqref{eq:utvt}, \eqref{eq:13.20} and \eqref{eq:trz}, we have
\begin{multline}\label{eq:ufgab}
u_{t}=\Tr^{\bf z}_{\rm s}\[\exp\(-A^2_t+{\bf z} B_t\)\]\\
=\Tr_{\rm s}^{\bf 
z}\[\sF_t\(A^2_t-{\bf z} B_t\)\]+\Tr_{\rm s}^{\bf z}\[\sG_t\(A^2_t-{\bf z} B_t\)\].
\end{multline}

We follow the same strategy used in the proof of Theorem \ref{thm:GBC}.
As in \eqref{eq:cet}, proceeding as in \cite[Theorem 13.6]{BG01}, when $t\to 0$, we have 
\begin{align}\label{eq:ufgab1}
\Tr_{\rm s}^{\bf z}\[\sG_t\(A^2_t-{\bf z} B_t\)\]=\cO(e^{-c/t}).
\end{align}
Moreover, the principal  symbol of  the lifting of $A^2_t-{\bf z} 
B_t$ on $\widetilde{U}_i$  is scalar, and is equal to $t|\xi|^2/4$ 
for $\xi\in T^*\widetilde{U}_i$.  Take $\alpha_0<\min\{\delta_0, 
\e_0\}$. As in the case of $\widetilde{\sF_t}\(t\Box^Z\)$,  for $x\in 
\widetilde{U}_{i}$ and $x\in \Supp(\widetilde{\phi}_i)$, the support of  $\widetilde{\sF_t}\(A^2_t-{\bf z} B_t\)_{U_{i}}(x,\cdot)$
is $B^{\widetilde{U}_i}_x(\alpha_0)$ and its value depends only on the restriction of $A^2_t-{\bf z} B_t$ on $U_i$. Also,
\begin{multline}
\Tr_{\rm s}^{\bf z}\[\sF_t\(A^2_t-{\bf z} B_t\)\]
=\sum_{i\in I} \frac{1}{|G_{U_i}|}\sum_{g\in G_{U_i}}\int_{y\in \widetilde{U}^g_i} \bigg\{\int_{Y\in N_{\widetilde{U}^g/\widetilde{U}_i,y}, |Y|<\e_0}
\widetilde{\phi}_i(y,Y)
\\\Tr_{\rm s}^{\bf z}\[g\widetilde{\sF}_t\(A^2_t-{\bf z} B_t\)_{U_i}\(g^{-1}(y,Y),(y,Y)\)\]k_i(y,Y)dY\bigg\}dv_{\widetilde{U}^g_i}
\end{multline}
As in \eqref{eq:or=man}, we can replace $A^2_t-{\bf z} B_t$ by the corresponding operator on manifolds. 
Recall that $h_g\(\nabla^{\pi^* \widetilde{F}_{U_i}}, g^{\pi^* \widetilde{F}_{U_i}}\)\in \Omega^{\rm odd}(\widetilde{U}_i^g)$ is defined in \eqref{eq:defhi}. Proceeding  as \cite[Theorem 3.24]{BG01}, as $t\to 0$, we have
\begin{multline}\label{eq:haha}
\int_{y\in \widetilde{U}^g_i} \bigg\{\int_{Y\in N_{\widetilde{U}^g/\widetilde{U}_i,y}, |Y|<\e_0}
\widetilde{\phi}_i(y,Y)k_i(y,Y)\\
\Tr_{\rm s}^{\bf z}\[g\widetilde{\sF}_t\(A^2_t-{\bf z} B_t\)_{U_i}\(g^{-1}(y,Y),(y,Y)\)\]dY\bigg\}dv_{\widetilde{U}^g_i}
\\= \int_{\widetilde{U}^g_i} \widetilde{\phi}_i(y,0) e\(\pi^*(T\widetilde{U}^g_i), \nabla^{\pi^*(T\widetilde{U}^g_i)}\) h_g\(\nabla^{\pi^* \widetilde{F}_{U_i}}, g^{\pi^* \widetilde{F}_{U_i}}\)+\cO(\sqrt{t}).
\end{multline}
Proceeding now as in \eqref{eq:egbc1}, by 
\eqref{eq:ufgab1}-\eqref{eq:haha}, we get the first identity of \eqref{eq:at1}.   \qed


%
%
%
%
%
%
%
%
%

\section{Analytic torsion on compact locally symmetric space}\label{Sec:Fried}
Let $G$ be a linear connected real  reductive group with maximal compact  
subgroup $K\subset G$, and let $\Gamma\subset G$ be a discrete 
cocompact subgroup  of $G$. The corresponding  locally symmetric 
space $Z=\Gamma\backslash G/K$ is a compact orientable orbifold. The 
purpose of this section is to show Theorem \ref{thm:3} which claims 
an equality between the analytic torsion of an acyclic unitarily flat 
orbifold vector bundle $F$ on $Z$ and the zero value of  the dynamical zeta function associated to the holonomy of $F$.  

This section is organized as follows.
In subsections \ref{sec:redu} and \ref{sec:sym}, we  recall  some facts on  reductive groups and the associated  symmetric spaces.

In subsections \ref{sec:semi} and \ref{sec:orb}, we recall the 
definition of  semisimple elements and the semisimple orbital 
integrals. We recall the Bismut formula for semisimple orbital 
integrals \cite[Theorem 6.1.1]{B09}. 

In subsection \ref{sec:loc}, we introduce the discrete cocompact subgroup $\Gamma$ and the associated  locally symmetric spaces. We recall the Selberg trace formula.

In subsection \ref{sec:dyn}, we introduce a Ruelle  dynamical 
zeta function associated to the holonomy of  a unitarily flat 
orbifold vector bundle on $Z$. We restate Theorem \ref{thm:3}. When 
the fundamental rank $\delta(G)\in \bN$ of $G$ does not equal to $1$ 
or when $G$ has noncompact center, we show Theorem \ref{thm:3}.

Subsections \ref{sec:d=1}-\ref{sec:proof} are devoted to  the case 
where $G$ has compact center and  $\delta(G)=1$. In subsection \ref{sec:d=1}, we recall some 
notation and results proved in  \cite[Sections 6A and 6B]{Shfried}.
In subsection \ref{sec:repK}, we introduce a class of representations 
of $K$. In subsection \ref{sec:suan}, using the Bismut formula,  we evaluate  the orbital integrals for the heat operators of the Casimir    associated to the $K$-representations constructed in subsection \ref{sec:repK}. 
In subsection \ref{sec:selzeta},  we introduce the  Selberg  zeta functions, which are shown to be meromorphic on $\bC$ and satisfy certain functional equations. 
In subsection \ref{sec:proof},  we show that  the dynamical zeta function equals   an alternating product of certain Selberg  zeta functions. 
We show Theorem \ref{thm:3}.

\subsection{Reductive  groups}\label{sec:redu}
Let $G$ be a linear connected real reductive group \cite[p. 3]{Knappsemi}, let $\theta\in \mathrm{Aut}(G)$ be the Cartan involution.  
That means $G$ is a closed connected group of real matrices that is stable under transpose, and $\theta$ is the composition of transpose and inverse of matrices. 
Let $K\subset G$ 
be the subgroup of $G$ fixed by $\theta$, so that $K$ is a maximal compact subgroup of $G$.

Let $\fg$ and $\fk$ be  the Lie algebras of $G$ and $K$. 
The Cartan involution $\theta$ acts naturally as Lie algebra automorphism of $\fg$. Then $\fk$ is the eigenspace of $\theta$ associated with the eigenvalue $1$. Let $\fp\subset \fg$ be the eigenspace with the eigenvalue $-1$, so that
\begin{align}\label{eq:cartan1}
  \fg=\fp\oplus\fk.
\end{align}
By \cite[Proposition 1.2]{Knappsemi}, we have the diffeomorphism
\begin{align}\label{eq:cartan2}
 (Y,k)\in \fp\times K\to  e^Y k\in G.
\end{align}
Set
\begin{align}
  &m=\dim\fp,&n=\dim\fk.
\end{align}

Let $B$ be a real-valued  non degenerate bilinear symmetric form on $\fg$ which is invariant under the adjoint action of $G$, and also under $\theta$. Then \eqref{eq:cartan1} is an orthogonal splitting  with respect to $B$. We assume $B$ to be positive on $\fp$, and negative on $\fk$. The form $\<\cdot,,\cdot\>=-B(\cdot,\theta\cdot)$ defines an $\Ad(K)$-invariant scalar product on $\fg$ such that the splitting \eqref{eq:cartan1} is still orthogonal. We denote by $|\cdot|$ the corresponding norm.

%

Let $\fg_\bC=\fg\otimes_\bR\bC$ be the complexification of $\fg$ and let $\fu=\sqrt{-1}\fp\oplus \fk$
be the compact form of $\fg$. Let $G_\bC$ and $U$ be the connected group of complex matrices associated to the Lie algebras $\fg_\bC$ and $\fu$.
By \cite[Propositions 5.3 and 5.6]{Knappsemi}, if $G$ has a compact center, $G_\bC$ is a linear connected complex reductive group with maximal compact subgroup $U$.


Let $\mathscr{U}(\fg)$ be the enveloping algebra of $\fg$. We  identify $\mathscr{U}(\fg)$ with the algebra of left-invariant differential operators on $G$.
Let $C^\fg\in \mathscr{U}(\fg)$ be the Casimir element. If 
$e_1,\cdots,e_m$ is an orthonormal  basis of $\fp$, if 
$e_{m+1},\cdots,e_{m+n}$ is an orthonormal  basis of $\fk$, then 
\begin{align}\label{eq:Cg}
  C^\fg=-\sum_{i=1}^{m}e^2_i+\sum_{i=m+1}^{n+m}e_i^{2}.
\end{align}
Classically, $C^\fg$ is in the center of $\mathscr{U}(\fg)$.

We define $C^{\fk}$ similarly. Let $\tau$ be a finite dimensional representation of $K$ on $V$. We denote by $C^{\fk,V}$ or $C^{\fk,\tau}\in \End(V)$ the corresponding Casimir operator acting on $V$, so that 
\begin{align}\label{eq:Ckvkt}
C^{\fk,V} = C^{k,\tau} = \sum_{i=m+1}^{m+n}\tau(e_i)^2.
\end{align}

%

Let $\delta(G)\in \bN$ be the fundamental rank of $G$, that is 
defined by  the difference between the complex ranks of $G$ and $K$. If  $T\subset K$ is a maximal torus of $K$ with Lie algebra of $\ft\subset \fk$, set
\begin{align}
\fb=\{Y\in \fp: [Y,\ft]=0\}.
\end{align}
Put
\begin{align}
&\fh=\fb\oplus \ft,&H=\exp(\fb)T.
\end{align}
By \cite[Theorem 5.22]{Knappsemi}, $\fh\subset \fg$ (resp.~$H\subset G$) is a $\theta$-invariant Cartan subalgebra (resp. subgroup). Therefore, 
\begin{align}
\delta(G)=\dim \fb.
\end{align}
Moreover, up to conjugation, $\fh\subset \fg$ (resp.~$H\subset G$) is the unique  Cartan subalgebra (resp. subgroup) with minimal noncompact dimension.

\subsection{Symmetric space}\label{sec:sym}
Let $\omega^\fg$ be the canonical left-invariant $1$-form on $G$ with values in $\fg$, and let $\omega^\fp, \omega^\fk$ be
its components in $\fp, \fk$, so that
\begin{align}\label{eq:wgkp}
  \omega^\fg=\omega^\fp+\omega^\fk.
\end{align}
Set $X=G/K$. Then $p:G\to X=G/K$
is a $K$-principle bundle equipped with the connection form $\omega^\fk$. 

Let $\tau$ be a finite dimensional orthogonal representation of $K$ on the real Euclidean space $E_\tau$.
Let $\cE_\tau$ be the associated Euclidean vector bundle with total space 
$G\times_{K} E_\tau$. It is equipped a Euclidean connection $\nabla^{\cE_\tau}$ induced by $\omega^{\fk}$. 
We identify  $C^\infty({X},\cE_\tau)$ with the $K$-invariant subspace $C^\infty(G,E_\tau)^K$ of smooth $E_\tau$-valued functions on $G$. Let $C^{\fg,X,\tau}$ be the Casimir element  of $G$ acting on $C^\infty(X,\cE_\tau)$. 

Observe that $K$ acts isometrically on $\fp$ by adjoint action. Using the above construction, the total space of the tangent bundle $TX$ is given by 
\begin{align}\label{eq:TX}
G\times_K\fp.
\end{align}
 It is equipped with a Euclidean metric $g^{TX}$ and a Euclidean connection $\nabla^{TX}$, which coincides with the Levi-Civita connection of the Riemannian manifold  $(X,g^{TX})$. 
Classically, $(X,g^{TX})$ has  non positive sectional curvature.

If $E_\tau=\Lambda^\cdot(\fp^*)$ is  equipped with the $K$-action induced by the  adjoint action, then $C^\infty(X,\cE_\tau)=\Omega^\cdot(X)$. In this case, we write $C^{\fg,X}=C^{\fg,X,\tau}$. By \cite[Proposition 7.8.1]{B09}, $C^{\fg,X}$ coincides with the Hodge Laplacian acting on
$\Omega^\cdot(X)$. 

Let $dv_X$ be the Riemannian volume of $(X,g^{TX})$. Define 
$[e(TX,\nabla^{TX})]^{\max}$ as in \eqref{eq:emax}. Since both $dv_X$ 
and $e(TX,\nabla^{TX})$ are $G$-invariant, we see that $[e(TX,\nabla^{TX})]^{\max}\in\bR$ is a constant.
Note that $\delta(G)$ and $\dim X$ have the same parity. 
 By \cite[Proposition 4.1]{Shfried}, if $\delta(G)\neq0$, then
\begin{align}\label{eq:ee0}
\[e\(TX,\nabla^{TX}\)\]^{\max}=0.
\end{align}
If $\delta(G)=0$, $G$ has a compact center. Then $U$ is a compact 
group with maximal torus $T$. Denote by $W(T,U)$ (resp.~$W(T,K)$) the Weyl group of $U$ (resp.~$K$) with respect to $T$,  and by $\vol(U/K)$ the volume of $U/K$ induced by $-B$. Then, \cite[Proposition 4.1]{Shfried} asserts 
\begin{align}\label{eq:en0}
\[e\(TX,\nabla^{TX}\)\]^{\max}=(-1)^{m/2}\frac{|W(T,U)|/|W(T,K)|}{\vol(U/K)}.
\end{align}

\subsection{Semisimple elements}\label{sec:semi}
If $\gamma\in G$, we denote by $Z(\gamma)\subset G$ the centralizer of $\gamma$ in $G$, and by $\fz(\gamma)\subset \fg$ its Lie algebra. If $a\in \fg$, let $Z(a)\subset G$ be the stabilizer of $a$ in $G$, and let $\fz(a)\subset \fg$ be its Lie algebra.

Following \cite[Section 3.1]{B09}, $\gamma\in G$ is said to be semisimple if and only if there is $g_\gamma\in G$, such that $\gamma=g_\gamma e^ak^{-1}g_\gamma^{-1}$ with
 \begin{align}\label{eq:asr}
&a\in \fp, &k\in K,&&\Ad(k)a=a.
 \end{align}
%

%

 Set
 \begin{align}\label{eq:akr}
 &a_\gamma=\Ad(g_\gamma)a,& k_\gamma=g_\gamma k g_\gamma^{-1}.
 \end{align}
 Therefore, $ \gamma=e^{a_\gamma}k_\gamma^{-1}$. Moreover, this decomposition does not depend on the choice of $g_\gamma$. By \cite[(3.3.3)]{B09}, we have 
 \begin{align}\label{eq:Zr}
&Z(\gamma)=Z(a_\gamma)\cap Z(k_\gamma),&\fz(\gamma)=\fz(a_\gamma)\cap \fz(k_\gamma).
\end{align}
 
By \cite[Proposition 7.25]{KnappLie}, $Z(\gamma)$ is reductive. The corresponding Cartan evolution and 
bilinear form are given by 
\begin{align}
& \theta_{g_\gamma}=g_\gamma \theta g^{-1}_\gamma, &B_{g_\gamma}(\cdot,\cdot)=B\(\Ad(g^{-1}_\gamma)\cdot, \Ad(g^{-1}_\gamma)\cdot\).
\end{align}
Let $K(\gamma)\subset Z(\gamma)$ be the fixed point of $\theta_{g_\gamma}$, so  $K(\gamma)$ is a maximal compact subgroup $Z(\gamma)$. 
Let $\fk(\gamma)\subset \fz(\gamma)$ be the Lie algebra of $K(\gamma)$. Let 
\begin{align}\label{eqZrPrKr}
\fz(\gamma)=\fp(\gamma)\oplus \fk(\gamma)
\end{align}
be the Cartan decomposition of $\fz(\gamma)$.  Let 
%
\begin{align}\label{eq:Xr1}
  X(\gamma)= Z(\gamma)/K(\gamma)
\end{align}
be the associated symmetric space. 

%
%

Let $Z^0(\gamma)$ be the connected component of the identity in $Z(\gamma)$. Similarly,  $Z^0(\gamma)$ is reductive  with maximal compact subgroup $Z^0(\gamma)\cap K(\gamma)$. Also, $ Z^0(\gamma)\cap K(\gamma)$ coincides with $K^0(\gamma)$,  the connected component of the identity in $K(\gamma)$. Clearly, we have 
\begin{align}\label{eq:Xr}
  X(\gamma)=Z^0(\gamma)/K^0(\gamma).
\end{align}

The semisimple element $\gamma$ is called elliptic if $a_\gamma=0$.  Assume now $\gamma$ is semisimple and nonelliptic. Then $a_\gamma\neq0$. Let $\fz^{a,\bot}(\gamma) $ (resp.~$\fp^{a,\bot}(\gamma)$) be the orthogonal spaces to $a_\gamma$ in $\fz(\gamma)$ (resp.~$\fp(\gamma)$) with respect to $B_{g_\gamma}$. Thus,
\begin{align}\label{eq:zabot}
  \fz^{a,\bot}(\gamma)=\fp^{a,\bot}(\gamma)\oplus \fk(\gamma).
\end{align}
Moreover, $ \fz^{a,\bot}(\gamma)$ is a Lie algebra. Let $Z^{a,\bot,0}(\gamma)$ be the connected subgroup of $Z^0(\gamma)$ that is associated with the Lie algebra $\fz^{a,\bot}(\gamma)$.
By \cite[ (3.3.11)]{B09}, $Z^{a,\bot,0}(\gamma)$ is reductive with maximal compact subgroup $K^0(\gamma)$ with Cartan decomposition \eqref{eq:zabot}, and
\begin{align}\label{eq:Z0a}
  Z^0(\gamma)=\mathbf{R}\times Z^{a,\bot,0}(\gamma),
\end{align}
so that $e^{ta_\gamma}\in Z^0(\gamma)$ maps into $t|a|\in \bR$.
Set
\begin{align}\label{eq:Xa0}
X^{a,\bot}(\gamma)=  Z^{a,\bot,0}(\gamma)/K^0(\gamma).
\end{align}
By \eqref{eq:Xr}, \eqref{eq:Z0a}, and \eqref{eq:Xa0}, we have
\begin{align}\label{eq:ZXM}
  X(\gamma)=\mathbf{R}\times X^{a,\bot}(\gamma),
\end{align}
so that the action $e^{ta_\gamma}$ on $X(\gamma)$ is just the translation by $t|a|$ on $\bR$.

\subsection{Semisimple orbital integral}\label{sec:orb}
Recall that $\tau$ is a finite dimensional orthogonal representation of $K$ on the real Euclidean space $E_\tau$. 
Let $p_t^{X,\tau}(x,x')$ be the smooth kernel of the heat operator $\exp(-tC^{\fg,X,\tau}/2)$ with respect to the Riemannian volume $dv_X$. As in \cite[(4.1.6)]{B09}, let $p_t^{X,\tau}(g)$ be the equivariant representation of the 
section $p_t^{X,\tau}(p1,\cdot)$. Then $p_t^{X,\tau}(g)$ is a 
$K\times K$-invariant  function in $ C^\infty(G, \End (E_\tau))$.

Let $dv_G$ be the left-invariant Riemannian volume on $G$ induced by 
the metric $-B(\cdot, \theta \cdot)$.
For a semisimple element $\gamma\in G$, denote by $dv_{Z^0(\gamma)}$ the  left-invariant Riemannian volume on $Z^0(\gamma)$ induced by $-B_{g_\gamma}(\cdot,\theta_{g_\gamma}\cdot)$.  Clearly, the choice of  $g_\gamma$ is irrelevant. Let 
$dv_{Z^0(\gamma)\backslash G}$ be the Riemannian volume on $Z^0(\gamma)\backslash G$ such that $dv_G=dv_{Z^0(\gamma)}dv_{Z^0(\gamma)\backslash G}$. By \cite[Definition 4.2.2, Proposition 4.4.2]{B09},  the  orbital integral 
\begin{multline}\label{eq:TRr}
\Tr^{[\gamma]}\[\exp\(-tC^{\fg,X,\tau}/2\)\]\\
=\frac{1}{\vol(K^0(\gamma)\backslash K)}\int_{Z^0(\gamma)\backslash G}\Tr^{E_\tau}\[p^{X,\tau}_t(g)\]dv_{Z^0(\gamma)\backslash G}
\end{multline}
 is well-defined.

%
%

\begin{re}\label{rer1}
In \cite[Definition 4.2.2]{B09},  the volume
$\vol(K_0(\gamma)\backslash K) $ are normalized to be $1$. By 
\cite[(3.3.18)]{B09}, in the definition of the orbital integral \eqref{eq:TRr}, we can 
replace $K^{0}(\gamma),Z^{0}(\gamma)$ by  $K(\gamma),Z(\gamma)$. 
\end{re}

\begin{re}
As the notation $\Tr^{[\gamma]}$ indicates, the orbital integral only depends on the conjugacy class of $\gamma$ in $G$.
However, the notation $[\gamma]$ will be used later for the conjugacy class of a discrete group $\Gamma$. Here, we consider $\Tr^{[\gamma]}$ as an abstract symbol.
\end{re}

We will also consider the case where $E_\tau=E_\tau^+\oplus E_\tau^-$ is a $\bZ_2$-graded representation of $K$. In this case, We will use the  notation $\Trs^{[\gamma]}\[\exp(-tC^{\fg,X,\tau}/2)\]$ when the trace on the right-hand side of \eqref{eq:TRr} is replaced by the supertrace on $E_\tau.$

In \cite[Theorem 6.1.1]{B09}, for any semisimple element $\gamma\in G$, Bismut gave an explicit formula for $\Tr^{[\gamma]}\[\exp\(-tC^{\fg,X,\tau}/2\)\]$. 
For the later use, let us recall the formula when $\gamma$ is elliptic. 

Assume now $\gamma\in K$. By \eqref{eq:akr}, we can take $g_\gamma=1$. Then $\fp(\gamma)\subset \fp$, $\fk(\gamma)\subset \fk$.
Let $\fp^\bot(\gamma)\subset \fp$, $\fk^\bot(\gamma)\subset \fk$ be the orthogonal space of $\fp(\gamma)$, $\fk(\gamma)$. Take $\fz^\bot(\gamma)=\fp^\bot(\gamma)\oplus\fk^\bot(\gamma)$.
Recall that $\widehat{A}$ is defined in \eqref{eq:defAhat}. Following \cite[Theorem 5.5.1]{B09}, for $Y\in \fk(\gamma)$, put
\begin{multline}\label{eq:J}
  J_{\gamma}({Y})=\frac{\widehat{A}\big(i\ad({Y})|_{\fp(\gamma)}\big)}{\widehat{A}\big(i\ad(Y)|_{\fk(\gamma)}\big)}\\
  \[\frac{1}{\det\big(1-\Ad(\gamma)\big)|_{\fz^\bot(\gamma)}}\frac{\det\big(1-\exp(-i\ad(Y))\Ad(\gamma)\big)|_{\fk^\bot(\gamma)}}{\det\big(1-\exp(-i\ad(Y))\Ad(\gamma)\big)|_{\fp^\bot(\gamma)}}\]^{1/2}\!\!\!\!\!.
\end{multline}
Note that by \cite[Section 5.5]{B09}, the square root in \eqref{eq:J} is well-defined, and its sign is chosen such that 
\begin{align}
J_{\gamma}(0)=\Big({\rm\det}\big(1-\Ad(\gamma)\big)|_{\fp^\bot(\gamma)}\Big)^{-1}.
\end{align}
Moreover, $J_{\gamma}$ is an  $\Ad\big(K^0(\gamma)\big)$-invariant analytic function on $\fk(\gamma)$ such that  there exist $c_\gamma>0, C_\gamma>0$,  for $Y\in \fk(\gamma)$, 
\begin{align}\label{eq:Jrexp}
   |J_{\gamma}(Y)|\le C_\gamma \exp{(c_\gamma|Y|)}.
\end{align}

Denote by $dY$ be the Lebesgue measure on $\fk(\gamma)$ induced by $-B$. Recall that $C^{\fk,\fp}, C^{\fk,\fp}$ are defined in \eqref{eq:Ckvkt}. 
By \cite[Theorem 6.1.1]{B09}, for $t>0$, we have
\begin{multline}\label{eq:trrJr}
\Tr^{[\gamma]}\[\exp\(-tC^{\fg,X,\tau}/2\)\]\\
=\frac{1}{(2\pi t)^{\dim \fz(\gamma)/2}}\exp\(\frac{t}{16}\Tr^{\fp}\[C^{\fk,\fp}\]+\frac{t}{48}\Tr^{\fk}\[C^{\fk,\fk}\]\)\\
  \int_{Y\in \fk(\gamma)}J_\gamma(Y)
  \Tr^{E_\tau}\[\tau\(\gamma\)\exp(-i{\tau(Y)})\]\exp\(-|{Y}|^2/2t\)dY.
\end{multline}

\subsection{Locally symmetric spaces}\label{sec:loc}
Let $\Gamma\subset G$ be a discrete cocompact subgroup of $G$. By 
\cite[Lemma 1]{Selberg60},  the elements of $\Gamma$ are semisimple. 
Let $\Gamma_e\subset \Gamma$ be the subset of elliptic elements in 
$\Gamma$. Set $\Gamma_+=\Gamma-\Gamma_e$. Let $[\Gamma]$ be the set 
of conjugacy classes of $\Gamma$, and 
let $[\Gamma_e]\subset [\Gamma]$ and $[\Gamma_+]\subset [\Gamma]$ be 
respectively the subsets of $[\Gamma]$ formed by the conjugacy classes of elements in $\Gamma_e$ and $\Gamma_+$. Clearly, $[\Gamma_e]$ is a finite set.

The group
$\Gamma$ acts properly  discontinuously and isometrically  on the 
left on $X$. Take $Z=\Gamma\backslash X$ to be the corresponding 
locally symmetric space. By Proposition \ref{prop:X/G} and Theorem \ref{thm:uX}, $Z$ is a compact 
orbifold. Note that by \eqref{eq:cartan2}, $X$ is a contractible manifold. By 
 Remark \ref{re231}, $X$ is the universal covering orbifold of $Z$. The Riemannian metric $g^{TX}$ on $X$ induces a Riemannian metric $g^{TZ}$ on $Z$. Clearly, $(Z,g^{TZ})$ has nonpositive curvature.

Let $\Delta_{\Gamma}\subset \Gamma$ be the subgroup of the elements 
in $\Gamma$ that act like the identity on $X$. Clearly, 
$\Delta_{\Gamma}$ is a finite group given by  
 \begin{align}\label{eqDg}
 	\Delta_{\Gamma}=\Gamma\cap K\cap Z(\fp),
 \end{align}
 where $Z(\fp)\subset G$ is the stabiliser of $\fp$ in $G$. 
  Thus, the orbifold fundamental group of $Z$ is $\Gamma/ \Delta_{\Gamma}.$

Let $F$ be a (possibly non proper) flat vector bundle on $Z$ with 
holonomy $\rho':\Gamma/\Delta_{\Gamma}\to \GL_r(\bC)$ such that 
\begin{align}
C^\infty(Z,F)=C^\infty(X,\bC^r)^{\Gamma/\Delta_{\Gamma}}.
\end{align}
Take $\rho$ to be the composition of the projection $\Gamma\to 
\Gamma/\Delta_{\Gamma}$ and $\rho'$. Then
\begin{align}\label{eq:733}
C^\infty(Z,F)=C^\infty(X,\bC^r)^{\Gamma}.
\end{align}
By abuse of notation, we still call  $\rho:\Gamma\to \GL_r(\bC)$ the 
holonomy of $F$. In the rest of this section, we assume $F$ is 
unitarily flat, or equivalently $\rho$ is unitary. Let $g^F$ be the 
associate flat Hermitian metric on $F$. 
Since $g^{TZ}$ and $g^{F}$ are fixed in the whole section, we write
\begin{align}
T(F)=T\(F,g^{TZ},g^F\).
\end{align}


The group $\Gamma$ acts on the Euclidean  vector bundles like $\cE_\tau$, and preserves the corresponding connections $\nabla^{\cE_\tau}$. 
The vector bundle  $\cE_\tau$ descends to a (possibly  non proper) orbifold  vector bundle  $\cF_\tau$  on $Z$. The total space of $\cF_\tau$ is given by $\Gamma\backslash G \times_{K} E_{\tau}$, and we have the identification of vector spaces
\begin{align}\label{eq:734}
C^\infty(Z,\cF_\tau)\simeq C^\infty(\Gamma\backslash G, E_{\tau})^{K}.
\end{align}
%
%
By \eqref{eq:733} and \eqref{eq:734}, we identify $C^\infty(Z, \cF_\tau \otimes_\bR F)$  with the $\Gamma$-invariant subspace of  $C^\infty(X, \cE_\tau \otimes_\bR \bC^r)$.
 Let $C^{\fg,Z,\tau,\rho}$ be the Casimir operator of $G$ acting on $C^\infty(Z, \cF_\tau \otimes_\bR F)$.  
 As we see in subsection \ref{sec:sym}, when $E_\tau=\Lambda^\cdot(\fp^*)$, 
 \begin{align}
 \Omega^\cdot(Z,F)\simeq C^\infty(Z, \cF_\tau \otimes_\bR F),
 \end{align}
 and the Hodge Laplacian acting on $\Omega^\cdot(Z,F)$ is given by 
\begin{align}\label{eq:BoxZ}
\Box^{Z}=C^{\fg,Z,\tau,\rho}.
\end{align}

%
%
%

 For $\gamma\in \Gamma$, set $\Gamma(\gamma)=Z(\gamma)\cap \Gamma$. 
 By \cite[Lemma 2]{Selberg60} (see also \cite[Proposition 4.9]{Shfried}), $\Gamma(\gamma)$ is cocompact in $Z(\gamma)$. Then $\Gamma(\gamma)\backslash X(\gamma)$ is a compact locally symmetric orbifold. Clearly, it depends only on the conjugacy class of $\gamma$ in $\Gamma$. Denote by $\vol(\Gamma(\gamma)\backslash X(\gamma))$ the Riemannian volume of $\Gamma(\gamma)\backslash X(\gamma)$ induced by $B_{g_\gamma}$.
 
 The group $K(\gamma)$ acts on the right on $\Gamma(\gamma)\backslash 
 Z(\gamma)$. For $h\in \Gamma(\gamma)\backslash 
 Z(\gamma)$, let $K(\gamma)_{h}$ be the  stabilizer of $h$ in 
 $K(\gamma)$.  Since $\Gamma(\gamma)\backslash X(\gamma)$ is 
 connected, the cardinal of a generic stabilizer is well defined 
 and  depends only on 
 the conjugacy class of $\gamma$ in $\Gamma$. We denote it by $n_{[\gamma]}$.  
%
%
%
 Then, we have 
 \begin{align}\label{eqgzR1}
	 \frac{\vol(\Gamma(\gamma)\backslash 
	 Z(\gamma))}{\vol(K(\gamma))}=\frac{\vol(\Gamma(\gamma)\backslash 
	 X(\gamma))}{n_{[\gamma]}}. 
\end{align}
Let us note that even if $K(\gamma)$ acts effectively on 
$\Gamma(\gamma)\backslash Z(\gamma)$, $n_{[\gamma]}$ is not necessarily equal to $1$. 

\begin{prop}
	For $\gamma\in \Gamma$, we have
	\begin{align}\label{eqng119}
		n_{[\gamma]}=\big|K\cap \Gamma(\gamma)\cap Z(\fp(\gamma))\big|. 
	\end{align}
In particular, if $\gamma=e$, we have
\begin{align}\label{eqRne}
n_{[e]}=	|\Delta_{\Gamma}|,
\end{align}
and if $\gamma\in \Delta_{\Gamma}$, we have
\begin{align}\label{eqRne19}
	n_{[\gamma]}=	|\Gamma(\gamma)\cap \Delta_{\Gamma}|.
\end{align}
\end{prop}
\begin{proof}
	For a generic element  $g=e^{f}h$ in $Z(\gamma)$ with $f\in 	
	\fp(\gamma)$ and $h\in K(\gamma)$, the stabliser of 
	$\Gamma(\gamma)g\in \Gamma(\gamma)\backslash Z(\gamma)$ in $K(\gamma)$ is given by 
	\begin{align}
		K(\gamma)_{\Gamma(\gamma)g}= K(\gamma)\cap  g ^{-1}		
		\Gamma(\gamma)g= g ^{-1}		\big(e^{f}K(\gamma)e^{-f}
\cap 		\Gamma(\gamma)\big)g. 
	\end{align}
	Then,
	\begin{align}
		n_{[\gamma]}=\left|e^{f}K(\gamma)e^{-f}
\cap 		\Gamma(\gamma)\right|. 
	\end{align}
Since   $e^{f}K(\gamma)e^{-f}$ is compact, since 
$ 		\Gamma(\gamma)$ is  discrete, and  since $f$ can vary 
in an open dense set, we can deduce that 
\begin{align}
	n_{[\gamma]}=|K(\gamma)\cap Z(\fp(\gamma)) \cap 		
	\Gamma(\gamma)|,
\end{align}
	from which we get \eqref{eqng119}. By \eqref{eqDg} and 
	\eqref{eqng119}, we get \eqref{eqRne}. If $\gamma\in 
	\Delta_{\Gamma}$, by \eqref{eqDg}, we get $\fp(\gamma)=\fp$. 
	Combining the result with \eqref{eqng119}, we get \eqref{eqRne19}. 
\end{proof}




\begin{thm}\label{thm:sel}
There exist
 $c>0$ and $C>0$ such that, for $t>0$, we have
\begin{multline}\label{eq:sel1}
\sum_{[\gamma]\in [\Gamma_+]}\frac{\vol(\Gamma(\gamma)\backslash 
X(\gamma))}{n_{[\gamma]}}\left|\Tr^{[\gamma]}\[\exp\(-tC^{\fg,X,\tau}/2\)\]\right|\\
\le C\exp\(-\frac{c}{t}+Ct\).
\end{multline}
For $t>0$, the following identity holds:
\begin{multline}\label{eq:sel2}
\Tr\[\exp\(-tC^{\fg,Z,\tau,\rho}/2\)\]\\
=\sum_{[\gamma]\in 
[\Gamma]}\Tr[\rho(\gamma)]\frac{\vol\big(\Gamma(\gamma)\backslash 
X(\gamma)\big)}{n_{[\gamma]}}\Tr^{[\gamma]}\[\exp\(-tC^{\fg,X,\tau}/2\)\].
\end{multline}
\end{thm}

\begin{proof}
 The proof  is identical to the one given in \cite[Theorem 
 4.10]{Shfried}. One difference is that we need to show an estimate  like  \cite[(4-26)]{Shfried}. This can be deduced from \cite[Lemma 
 8]{Selberg60}, which states that $\Gamma$ possesses a  normal 
 torsion free subgroup of finite index.
 
 Let us explain the reason for which the coefficients $n_{[\gamma]}$ appear in \eqref{eq:sel2}. Indeed, the 
 restriction on the diagonal of the trace of the integral kernel of 
 $\exp(-t C^{\fg,Z,\tau,\rho}/2)$ is given by 
 \begin{align}\label{eqRsel}
	\frac{1}{|\Delta_\Gamma |} \sum_{\gamma\in 
	 \Gamma }\Tr[\rho(\gamma)]\Tr^{\cE_{\tau}}\[\gamma_{x}p^{X,\tau}_{t}(x,\gamma 
	 x)\],
 \end{align}
 where $\gamma_{x}$ denotes the obvious element in $\Hom 
 (\cE_{\tau,x},  \cE_{\tau,\gamma x})$ induced by $\gamma$ (see \cite[p. 
 79]{B09}).  
 By \eqref{eqgzR1}, \eqref{eqRne}, and \eqref{eqRsel},  we have 
 \begin{multline}\label{eqR33}
	 \Tr\[\exp\(-t C^{\fg,Z,\tau,\rho}/2\)\]\\
	 =\frac{1}{\vol(K)}\int_{\Gamma\backslash G}\sum_{\gamma\in 
	 \Gamma }\Tr[\rho(\gamma)]\Tr^{E_{\tau}}\[p_{t}^{X,\tau}(g^{-1}\gamma 
	 g)\]dv_{\Gamma\backslash G},
\end{multline}
where $dv_{\Gamma\backslash G}$ is the volume form on $\Gamma\backslash G$ 
induced by $dv_{G}$. Proceeding as in \cite[(4.8.8)-(4.8.12)]{B09}, by Remark \ref{rer1}, 
 \eqref{eq:TRr}, and 
 \eqref{eqR33}, we get
 \begin{multline}\label{eqgzR2}
	 \Tr\[\exp\(-t 
	 C^{\fg,Z,\tau,\rho}/2\)\]\\
	 =\sum_{[\gamma]\in 
	 [\Gamma] } 
	 \Tr[\rho(\gamma)]\frac{\vol(\Gamma(\gamma)\backslash 
	 Z(\gamma))}{\vol(K(\gamma))}\Tr^{[\gamma]}\[\exp(-tC^{\fg,X,\tau}/2)\].
\end{multline}
 By \eqref{eqgzR1} and \eqref{eqgzR2}, we get \eqref{eq:sel2}.
 \end{proof}

\begin{re}\label{reg0}	
	By \cite[Theorem 
	3.14, Remark 3.15]{Ruan_Orbifold} and \cite[(4.8.22)]{B09}, when 
	counting with multiplicites, we 
	have the identification 
	of the orbifolds, 
	\begin{align}\label{eqmmm}
	Z\coprod \Sigma Z\simeq \coprod_{[\gamma]\in 
	[\Gamma_{e}]}\Gamma(\gamma)\backslash X(\gamma),
    \end{align}
	where the multiplicity of each component $Z_{i}$ of $Z\coprod 
	\Sigma Z$ is  $m_{i}$ (see \ref{eq:mi25}) and the multiplicity of 
	$\Gamma(\gamma)\backslash X(\gamma)$ is $n_{[\gamma]}$. 
	Note also that, by Remark \ref{rel0}, we can consider $\coprod_{[\gamma]\in 
[\Gamma_{e}]}\Gamma(\gamma)\backslash X(\gamma)$ as the space of paths on $Z$ of length $0$. 
\end{re}
 
\begin{re}\label{reGeo}
Assume that $\Gamma$ acts effectively on $X$. Then $Z$ can be 
represented by the action groupoid $\cG $ (see \cite[Example 
1.32]{Ruan_Orbifold}) 
whose object is $X$ and whose arrow is  $\Gamma\times X$. An arrow 
$(\gamma,x)\in \Gamma\times X$ maps $x$ to $\gamma x$. By Remark 
\ref{rergeoo}, any closed geodesic on $Z$ can be represented by the 
closed $\cG$-geodesic $c=(b_{0}; \mathrm{id},\gamma^{-1})$ where $b_{0}:[0,1]\to X$ is a 
geodesic on $X$ such that $\gamma b_{0}=b_{1}$ and $\gamma_{*} 
\dot{b}_{0}(0)=\dot{b}_{0}(1)$. By 
\cite[Theorem 3.1.2]{B09}, the length of $b_{0}$ is given by 
$|a_{\gamma}|$ (see \eqref{eq:akr}). Moreover,  the space of the closed 
$\cG$-geodesics is given 
by $\cup_{\gamma\in \Gamma}X(\gamma)$. Thus, the space of  the closed 
geodesics on $Z$ is given by $\coprod_{[\gamma]\in [\Gamma]} 
\Gamma(\gamma)\backslash X(\gamma)$, whose  component has the  multiplicity 
$n_{[\gamma]}$. 
For general $\Gamma$, as in Remark \ref{reg0}, the same result still 
holds ture. 
\end{re}
 
We extend \cite[Corollary 2.2]{MStorsion} and \cite[Theorem 
7.9.3]{B09} to orbifolds (see also \cite[Corollary 4.3]{Shfried}).

\begin{cor}\label{cor:ms}Let $F$ be a unitarily flat orbifold vector 
	bundle on $Z$. If   $\dim Z$ is odd and $\delta(G)\neq 1$, then for any $t>0$, we have
\begin{align}\label{eq:BMS}
    \Trs\[N^{\Lambda^\cdot(T^*Z)}\exp\(-t\Box^Z/2\)\]=0.
  \end{align}
In particular,
  \begin{align}\label{eq:tau=0}
    T(F)=1.
  \end{align}
\end{cor}
\begin{proof}Since $\dim Z$ is odd, $\delta(G)$ is odd. Since 
	$\delta(G)\neq 1$,  $\delta(G)\g3$.  By \cite[Theorem 
	4.12]{Shfried}, for any $\gamma\in G$ semisimple, we have 
\begin{align}\label{eq:vanishT2}
    \Trs^{[\gamma]}\[N^{\Lambda^\cdot(T^*X)}\exp\(-tC^{\fg,X}/2\)\]=0.
  \end{align}
  By  \eqref{eq:BoxZ}, \eqref{eq:sel2}, and \eqref{eq:vanishT2}, we 
  get  \eqref{eq:BMS}.
\end{proof}

Suppose that $\delta(G)=1$. Let us recall some notation in 
\cite[(4-49)-(4-52), (6-15), (6-16)]{Shfried}. Up to sign, we fix an element  $a_{1}\in \fb$ such 
that $B(a_{1},a_{1})=1$. As in subsection \ref{sec:semi}, set  
\begin{align}\label{eq:MPMKM1}
&M=Z^{a_{1},\bot,0}(e^{a_{1}}),& K_{M}=K^{0}(e^{a_{1}}),
\end{align}
and \index{M@$\fm$} \index{P@$\fp_{\fm}$} \index{K@$\fk_{\fm}$}
\begin{align}\label{eq:mpkd1}
&\fm=\fz^{a_{1},\bot}(e^{a_{1}}), 
&\fp_{\fm}=\fp^{a_{1},\bot}(e^{a_{1}}), &&\fk_{\fm}=\fk(e^{a_{1}}).
\end{align}
As in subsection \ref{sec:semi}, $M$  
is a connected reductive group such that  $\delta(M)=0$ with Lie algebra $\fm$, with maximal 
compact subgroup $K_{M}$, and  with Cartan decomposition 
$\fm=\fp_{\fm}\oplus\fk_{\fm}.$  Let 
\begin{align}\label{eq:XM}
X_{M}=M/K_{M}
\end{align}
be the corresponding symmetric space. 
For $k\in T$, we have $\delta(Z^{0}(k))=1$. 
Denote by $M^{0}(k)$, $\fm(k)$,  $\fp_\fm(k)$, $\fk_\fm(k)$,  
$X_{M}(k)$ the analogies of $M$, $\fm$, $\fp_{\fm}$, $\fk_{\fm}$, $X_{\fm}$ when $G$ 
is replaced by $Z^{0}(k)$.  

%

Assume  that $\delta(G)=1$ and that $G$ has  noncompact center.  By 
\cite[4-51]{Shfried},
we have
\begin{align}\label{eq:G=RM}
	&G=\bR\times M, &K=K_{M}, &&X=\bR\times X_{M}. 
\end{align}
Recall that $H=\exp(\fb) T$. Note that if $\gamma=e^{a}k^{-1}\in H$ with 
$a\neq0$, then 
\begin{align}\label{lla}
	X^{a,\bot}(\gamma)=X_{M}(k).
\end{align}

We have an extension  of 
\cite[Proposition 4.14]{Shfried}. 

\begin{prop}
	Let $\gamma\in G$ be a semisimple element. 
	If $\gamma$ can not be conjugated into $H$ by elements of $G$, 
	then for $t>0,$ we have 
\begin{align}\label{eq:dgn00}
 		\Trs^{[\gamma]}\[N^{\Lambda^{\cdot}(T^{*}X)}\exp\(-tC^{\fg, 
		X}/2\)\]=0.
 	\end{align}
	If $\gamma=e^{a}k^{-1}\in H$ with $a\in \fb$ and $k\in T$, then for $t>0,$ we have 
	\begin{multline}\label{eq:dgn1}
		\Tr^{[\gamma]}\[N^{\Lambda^{\cdot}(T^{*}X)}\exp\(-tC^{\fg,X}/2\)\]\\
		=-\frac{1}{\sqrt{2\pi t}}e^{-|a|^{2}/2t}\[e\(TX_{M}(k),\nabla^{TX_{M}(k)}\)\]^{\max}.
	\end{multline}
%
\end{prop}
\begin{proof}
%
%
	Equations \eqref{eq:dgn00}, \eqref{eq:dgn1} with $\gamma=1$ or 
	$\gamma=e^{a}k^{-1}$ with $a\neq0$ are just \cite[(4-45), (4-53), (4-54)]{Shfried}. 
	Equation \eqref{eq:dgn1} for general $\gamma\in H$ is a 
	consequence of \cite[(4-55), (4-56), (4-58)]{Shfried} and 
	\eqref{lla}.
\end{proof}

\subsection{Ruelle dynamical zeta functions}\label{sec:dyn}
By Remark \ref{reGeo} (c.f. \cite[Proposition 
5.15]{DuistermaatKolkVaradarajan} for manifold case), the space of the 
closed geodesics on $Z$ of  positive lengths  consists of a disjoint union of smooth connected compact orbifolds
  \begin{align}
 \coprod_{[\gamma]\in [\Gamma_+]} B_{[\gamma]}.
 \end{align} Moreover,   
 $B_{[\gamma]}$ is diffeomorphic to  $ \Gamma(\gamma)\backslash 
 X(\gamma)$ with multiplicity 
 $n_{[\gamma]}$. Also, all the elements in  $B_{[\gamma]}$ have the same length $l_{[\gamma]}=|a_{\gamma}|>0$.

The group  $\bbS^1$ acts locally freely   on $B_{[\gamma]}$ by 
rotation. 
Then 
$\bbS^1\backslash B_{[\gamma]}$ is still an orbifold.  Set  
\begin{align}\label{eq:mulr}
m_{[\gamma]}=n_{[\gamma]}\left|\ker \Big(\bbS^1\to 
\mathrm{Diffeo}\big(B_{[\gamma]}\big)\Big)\right|\in \mathbf{N}^{*}.
\end{align}
We define $m_{[\gamma]}$ to be the multiplicity of $\bbS^1\backslash 
B_{[\gamma]}$.

\begin{prop}\label{prop:orbv0}
For $\gamma\in \Gamma_+$ such that $\gamma=a_\gamma k_\gamma^{-1}$ as in \eqref{eq:akr}, we have 
\begin{align}\label{eq:V=orbe}
\frac{\chi_{\rm orb}\(\bbS^1\backslash 
B_{[\gamma]}\)}{m_{[\gamma]}}=\frac{\vol\big(\Gamma(\gamma)\backslash 
X(\gamma)\big)}{|a_\gamma|n_{[\gamma]}}\[e\(TX^{a,\bot}(\gamma),\nabla^{TX^{a,\bot}(\gamma)}\)\]^{\max}.
\end{align}
In particular, if $\delta(G)\g2$, then for all $[\gamma]\in [\Gamma_+] $, we have 
\begin{align}\label{eq:orb0}
\chi_{\rm orb}\(\bbS^1\backslash B_{[\gamma]}\)=0.
\end{align}
Also, if $\delta(G)=1$ and if $\gamma$ can not be conjugated into $H$, then \eqref{eq:orb0} still  holds. 
\end{prop}
\begin{proof}
The proof of our proposition  is identical to the one given in 
\cite[Proposition 5.1,  Corollary 5.2]{Shfried}.
\end{proof}

%
%
%
%

Recall that  $\rho:\Gamma\to U(r)$ is a unitary   representation of 
$\Gamma$.

\begin{defin}\label{def:ruelle}
 The Ruelle dynamical zeta function $R_\rho$ is said to be well-defined if 
 \begin{itemize}
\item for $\Re(\sigma)\gg1$, the sum 
\begin{align}
\Xi_{\rho}(\sigma)=\sum_{[\gamma]\in [\Gamma_+]}\Tr[\rho(\gamma)]\frac{\chi_{\rm orb}\(\bbS^1\backslash B_{[\gamma]}\)}{m_{[\gamma]}}e^{-\sigma l_{[\gamma]}}
\end{align}
converges absolutely to a holomorphic function; 
\item the function $R_\rho(\sigma)=\exp(\Xi_{\rho}(\sigma))$ has a meromorphic extension to $\sigma\in \bC$.
\end{itemize}
\end{defin}
%

By \eqref{eq:orb0}, if $\delta(G)\g2$, the dynamical zeta function $R_\rho$ is well-defined and 
\begin{align}\label{eq:R=1}
R_\rho\equiv1.
\end{align}
We restate Theorem \ref{thm:3}, which is the main result of this section.
\begin{thm}\label{Thm:3}
If $\dim Z$ is  odd, then the dynamical zeta function $R_\rho(\sigma)$ is  well-defined. 
There exist explicit constants $C_\rho\in \bR$ with $C_\rho\neq0$ and 
$r_\rho\in \bZ$ (see \eqref{eq:Cr}) such that as $\sigma\to0$, 
\begin{align}\label{eq:masi1}
R_\rho(\sigma)=C_\rho T(F)^2\sigma^{r_\rho}+\cO(\sigma^{r_\rho+1}).
\end{align}
Moreover, if $H^\cdot(Z,F)=0$, we have 
\begin{align}\label{eq:masi2}
&C_\rho=1,&r_\rho=0,
\end{align}
so that 
\begin{align}\label{eq:FFs}
R_\rho(0)=T(F)^2.
\end{align}
\end{thm}

\begin{proof}If  $\delta(G)\neq 1$, Theorem \ref{Thm:3} is a 
	consequence of \eqref{eq:tau=0} and \eqref{eq:R=1}. Assume now	
	$\delta(G)=1$ and $G$ has  noncompact center. 
	Proceeding as \cite[Theorem 5.6]{Shfried}, up to evident 
	modification, we see that the dynamical zeta function $R_{\rho}(\sigma)$ extends 
meromorphically to $\sigma\in \bC$ such that the following identity of 
meromorphic function holds,
\begin{multline}
	\label{eq:R=Tdg=1}
	R_{\rho}(\sigma)=\prod_{i=1}^m 
	\det\(\sigma^{2}+\Box^Z|_{\Omega^i(Z,F)}\)^{(-1)^ii}\\\exp\Bigg(\sigma\!\!\!\sum_{\footnotesize\substack{[\gamma]\in 
	[\Gamma_{e}]\\ \gamma=g_{\gamma}k^{-1}g_{\gamma}^{-1}}}\!\!\!\Tr\[\rho(\gamma)\]
	\frac{\vol\big(\Gamma(\gamma)\backslash 
	X(\gamma)\big)}{n_{[\gamma]}}\[e\(TX_{M}(k),\nabla^{TX_{M}(k)} 
	\)\]^{\max}\Bigg),
\end{multline}
from which we get \eqref{eq:masi1}-\eqref{eq:FFs}. The proof for 
	the case where $\delta(G)=1$ and where $G$ has compact center will be given in subsections \ref{sec:repK}-\ref{sec:proof}.
\end{proof}

\begin{re}
By \eqref{eq:FFs}, we have the formal  identity 
\begin{align}\label{eq:Friedformel}
2\log T(F)=
\sum_{[\gamma]\in [\Gamma_+]}\Tr[\rho(\gamma)]\frac{\chi_{\rm orb}\big(\bbS^1\backslash B_{[\gamma]}\big)}{m_{[\gamma]}}.
\end{align}
We note the similarity between \eqref{eq:35} and \eqref{eq:Friedformel}.
\end{re}
\begin{re}
The formal identity \eqref{eq:Friedformel} can be deduced formally  
using the path integral argument and Bismut-Goette's $V$-invariant 
\cite{BGdeRham} as 
in \cite[Section 1E]{Shfried}. We leave the details to  readers.
\end{re}

%

\subsection{Reductive group with $\delta(G)=1$ and with compact center}\label{sec:d=1}
From now on, we assume that $\delta(G)=1$ and that $G$ has compact 
center. Let us introduce some 
notation following \cite[Sections 6A and 6B]{Shfried}. 
We use the notation in \eqref{eq:MPMKM1}-\eqref{eq:XM}. Let 
$Z(\fb)\subset G$ 
be the stabilizer of $\fb$ in $G$, and 
let $\fz(\fb)\subset \fg$ be its Lie algebra. 
We define $\fp(\fb)$, $\fk(\fb)$, $\fp^{\bot}(\fb)$, 
$\fk^{\bot}(\fb)$, $\fz^{\bot}(\fb)$ in an obvious way as in 
subsection \ref{sec:semi}, so that 
\begin{align}\label{eq:tao1}
&\fp(\fb)=\fb\oplus\fp_{\fm},&\fk(\fb)=\fk_{\fm},
\end{align}
and
\begin{align}\label{eq:mpk1}
 & \fp=\fb\oplus\fp_\fm\oplus\fp^\bot(\fb),&\fk=\fk_\fm\oplus\fk^\bot(\fb).
\end{align}
Let $Z^{0}(\fb)$  be the connected component of 
the identity in $Z(\fb)$. By \eqref{eq:Z0a}, we have
\begin{align}\label{eq:Z0UB}
  Z^0(\fb)=\bR\times M.
\end{align}





Set
\begin{align}
\fz^\bot(\fb)=\fp^\bot(\fb)\oplus \fk^\bot(\fb).
\end{align}
Recall that we have fixed $a_{1}\in \fb$ such that 
$B(a_{1},a_{1})=1$. The choice of $a_{1}$ fixes an orientation of 
$\fb$.  By  \cite[Proposition 6.2]{Shfried}, there exists unique $\alpha\in 
\fb^*$ such that $\<\alpha,a_{1}\>>0$, and that for any  $a\in \fb$, the action of $\ad(a)$  on 
$\fz^\bot(\fb)$ has only two eigenvalues $\pm\<\alpha,a\>\in \bR$.  
Take $a_0=a_{1}/\<\alpha,a_{1}\>\in \fb$. We have  
\begin{align}\label{eq:655}
\<\alpha,a_0\>=1.
\end{align}
Let $\fn\subset \fz^\bot(\fb)$ (resp. $\ol{\fn}$) be the $+1$ (resp. $-1$) eigenspace of $\ad(a_0)$, so that 
\begin{align}
\fz^\bot(\fb)=\fn\oplus \ol{\fn}.
\end{align}
Clearly, $\ol{\fn}=\theta \fn$, and $M$ acts on $\fn$ and $\ol{\fn}$. 
As explained in \cite[Section 5.1]{Shfried}, $\dim \fn$ is even. Set 
\begin{align}
l=\frac{1}{2}\dim \fn.
\end{align}

Let $\fu(\fb)\subset \fu$ and $\fu_\fm\subset \fu$ be respectively the compact forms of $\fz(\fb)$ and of $\fm$. Then,
\begin{align}
&\fu(\fb)=\sqrt{-1}\fb\oplus \sqrt{-1} \fp_\fm \oplus \fk_\fm,&\fu_\fm= \sqrt{-1} \fp_\fm \oplus \fk_\fm.
\end{align}
Let $\fu^\bot(\fb)\subset \fu$ be the orthogonal space of $\fu(\fb)$, 
so that
\begin{align}\label{eq:ubuu}
\fu=\sqrt{-1}\fb\oplus \fu_\fm\oplus \fu^\bot(\fb).
\end{align}

Let $U(\fb)\subset U$ and $U_M\subset U$ be respectively the 
corresponding connected subgroups of complex 
matrices of groups associated to the Lie algebras $\fu(\fb)$ and 
$\fu_\fm$. By \cite[Section 6B]{Shfried}, $U(\fb)$ and $U_M$ are  
compact such that 
\begin{align}\label{UbbUm}U(\fb)=\exp(\sqrt{-1}\fb)U_M.
\end{align}
Clearly, $U(\fb)$ acts on $\fb, \fu_\fm, \fu^\bot(\fb)$ and preserves 
the splitting \eqref{eq:ubuu}.

Put
\begin{align}
Y_{\fb}=U/U(\fb).
\end{align}
By \cite[Propositions 6.7]{Shfried}, $Y_\fb$ is a Hermitian symmetric 
space of the compact type. Let $\omega^{Y_\fb}\in \Omega^2(Y_\fb)$ be 
the canonical K\"alher form on $Y_\fb$ induced by $B$.  As in 
subsection \ref{sec:sym}, $U\to Y_\fb$ is a $U(\fb)$-principle bundle 
on $Y_\fb$ with canonical connection. Let $(TY_\fb,\nabla^{TY_\fb})$ 
and $(N_\fb,\nabla^{N_\fb})$ the Hermitian vector bundle with 
Hermitian connection induced by the representation of $U(\fb)$ on 
$\fu_\fm$ and $\fu^\bot_\fm$. For a vector space $E$, we still denote 
by $E$ the corresponding trivial bundle  on $Y_\fb$. By \eqref{eq:ubuu}, we have an analogy of \cite[(2.2.1)]{B09},
\begin{align}\label{eq:udada}
\fu=\sqrt{-1}\fb \oplus N_\fb \oplus TY_\fb.
\end{align}

Take  $k\in T$.  Denote by  
$\fn(k)$, $U^0(k)$, $Y_\fb(k)$ and $\omega^{Y_\fb(k)}$ the analogies 
of   $\fn$, $U$, $Y_\fb$ and $\omega^{Y_\fb}$ when $G$ is replaced by $Z^0(k)$. The embedding $U^0(k)\to U$ induces an embedding $Y_\fb(k)\to Y_\fb$. 
 Clearly, $k$ acts on the left on $Y_\fb$, and  $Y_\fb(k)$ is fixed by the action of $k$. 
 Recall that the equivariant $\widehat{A}$-forms 
 $\widehat{A}_{k^{-1}}\(N_\fb|_{Y_\fb(k)}, 
 \nabla^{N_\fb|_{Y_\fb(k)}}\)$ and 
 $\widehat{A}_{k^{-1}}\(TY_\fb|_{Y_\fb(k)}, 
 \nabla^{TY_\fb|_{Y_\fb(k)}}\)$ are defined in \eqref{eq:Agpm1}. Let 
 $\widehat{A}^{\fu_{\fm}}_{k^{-1}}(0)$ and 
 $\widehat{A}^{\fu^{\bot}(\fb)}_{k^{-1}}(0)$ be the components of 
 degree $0$ of the form $\widehat{A}_{k^{-1}}\(N_\fb|_{Y_\fb(k)}, \nabla^{N_\fb|_{Y_\fb(k)}}\)$ and $\widehat{A}_{k^{-1}}\(TY_\fb|_{Y_\fb(k)}, \nabla^{TY_\fb|_{Y_\fb(k)}}\)$. Following \cite[(7.7.3)]{B09},  set 
 \begin{align}
 \widehat{A}_{k^{-1}}(0)= \widehat{A}^{\fu_{\fm}}_{k^{-1}}(0)\widehat{A}^{\fu^{\bot}(\fb)}_{k^{-1}}(0). 
 \end{align}
 By \eqref{eq:udada}, as in  \cite[(7.7.5)]{B09}, the following identity of closed forms on $Y_\fb(k)$ holds:
 \begin{align}\label{eq:AAA}
 \widehat{A}_{k^{-1}}(0)=\widehat{A}_{k^{-1}}\(N_\fb|_{Y_\fb(k)}, 
 \nabla^{N_\fb|_{Y_\fb(k)}}\)\widehat{A}_{k^{-1}}\(TY_\fb|_{Y_\fb(k)}, \nabla^{TY_\fb|_{Y_\fb(k)}}\),
  \end{align}  
  which generalizes \cite[Proposition 6.8]{Shfried}.

%

\subsection{Auxiliary virtual representations of $K$}\label{sec:repK}
We follow \cite[Sections 6C and 7A]{Shfried}. Denote by $RO(K_M)$ and
$RO(K)$ the real representation rings of $K_M$ and $K$. Since $K_M$ and 
$K$ have the same maximal torus $T$, the restriction $RO(K)\to RO(K_M)$ is injective. 

By \cite[Proposition 6.10]{Shfried},  we have the identity in $RO(K_M)$,
\begin{align}\label{eqkkk}
	\(\sum_{i=1}^{m} (-1)^{i-1}i 
	\Lambda^i(\fp^*)\)|_{K_{M}}=\sum_{i=0}^{\dim 
	\fp_{\fm}}\sum_{j=0}^{2l}(-1)^{i+j}\Lambda^{i}(\fp^{*}_{\fm})\otimes 
	\Lambda^{j}(\fn^{*}).
\end{align}
By  \cite[Corollary 6.12]{Shfried}, each term on the right hand side 
of \eqref{eqkkk} has a lift to $RO(K)$. More precisely, let us recall 
\cite[Assumption 7.1]{Shfried}.

\begin{as}\label{as:1}
Let $\eta$ be a real finite dimensional representation of $M$ on the 
vector space $E_\eta$ such that 
\begin{enumerate}
   \item the restriction $\eta|_{K_M}$ to $K_M$ can be lifted into $RO(K)$;
   \item the action of  the Lie algebra  $\fu_\fm\subset \fm\otimes_\bR\bC$ on
$E_\eta\otimes_\bR\bC$, induced by complexification, can be lifted to an action of Lie group $U_M$;
   \item the Casimir element $C^{\fu_\fm}$ of $\fu_\fm$ acts on $E_\eta\otimes_\bR\bC$ as a scalar $C^{\fu_\fm,\eta}\in \bR$.
 \end{enumerate}
\end{as}

By \cite[Corollary 6.12]{Shfried}, let 
$\widehat{\eta}=\widehat{\eta}^+-\widehat{\eta}^-\in RO(K)$ be a real 
virtual finite dimensional representation of $K$ on
$E_{\widehat{\eta}}=E^+_{\widehat{\eta}}- E^-_{\widehat{\eta}}$
such that the following identity in $RO(K_M)$ holds:
\begin{align}\label{eq:hatvar}
E_{\widehat{\eta}}|_{K_M}=  \sum_{i=0}^{\dim \fp_\fm}(-1)^i\Lambda^i(\fp^*_\fm)\otimes E_{\eta}|_{K_M}.
\end{align}

Note that  $M$ acts on $\fn$ by  adjoint action. By 
\cite[Corollary 6.12 and Proposition 6.13]{Shfried}, for $0\le j\le 2l$, the induced representation $\eta_{j}$ 
of $K_M$ on $\Lambda^j(\fn^*)$ satisfies  Assumption 
\ref{as:1}, such that the following identity in $RO(K)$ holds,
\begin{align}\label{eq:sumsum}
 \sum_{i=1}^{m} (-1)^{i-1}i \Lambda^i(\fp^*)=\sum_{j=0}^{2l}(-1)^jE_{\widehat{\eta}_j}.
\end{align}

\subsection{Evaluation of $\Trs^{[\gamma]}[\exp(-t C^{\fg,X,{\widehat{\eta}}}/2)]$}\label{sec:suan}
In \cite[Theorem 7.3]{Shfried}, we evaluate  the orbital integral 
$\Trs^{[\gamma]}[\exp(-t C^{\fg,X,{\widehat{\eta}}}/2)]$ when 
$\gamma=1$ or when $\gamma$ is a non elliptic semisimple element. 
In this subsection, we evaluate  $\Trs^{[\gamma]}[\exp(-t 
C^{\fg,X,{\widehat{\eta}}}/2)]$ when $\gamma$ is elliptic. To state 
the result, let us introduce some notation \cite[(7-3)-(7-7)]{Shfried}. 

 Recall that  
$T$ is a maximal torus of $K_M$, $K$ and $U_M$. Denote by $W(T,U_M)$ 
and $W(T,K)$ the corresponding Weyl groups, and denote by $\vol(K/K_M)$ and 
 $\vol(U_M/K_M)$ the Riemannian volumes induced by $-B$.
Set
\begin{align}\label{eq:cG}
c_G=(-1)^{(\dim \fp-1)/2}\frac{|W(T,U_M)|}{|W(T,K)|}\frac{\vol(K/K_M)}{\vol(U_M/K_M)}\in \bR.
\end{align}
The constant $c_{Z^{0}(k)}$ is defined in a similar way.

As in \cite[(7-3)]{Shfried}, by (2) of  Assumption \ref{as:1}, 
$U_M$ acts on $E_\eta\otimes_\bR \bC$.  We extend this action  to $U(\fb)$ such that $\exp(\sqrt{-1}\fb)$ acts trivially. 
Denote by $F_{\fb,\eta}$  the Hermitian vector bundle  on $Y_\fb$ with total space $U\times_{U(\fb)} (E_{\eta}\otimes_{\bR}\bC)$ with Hermitian connection $\nabla^{F_{\fb,\eta}}$. 

  Note that  the K\"ahler form $\omega^{Y_\fb(k)}$ defines a volume form $dv_{Y_{\fb}(k)}$ on $Y_\fb(k)$.
  For a $U^0(k)$-invariant differential form $\beta$ on $Y_{\fb}(k)$, 
  as in \cite[(7-7)]{Shfried}, define $\[\beta\]^{\max}\in \bR$ such that
  \begin{align}
  \beta-[\beta]^{\max}dv_{Y_{\fb}(k)}
  \end{align}
has degree smaller than $\dim Y_{\fb}(k)$.   Recall that $a_0\in \fb$ is defined in \eqref{eq:655}.

%
%

\begin{thm}\label{thm:orbint}
Let $\gamma\in G$ be semisimple.  
If $\gamma$ can not be conjugated into $H$ by elements of $G$, then 
for $t>0$, we have 
\begin{align}\label{eq:G1=00}
  \Trs^{[\gamma]}\[\exp\(-t C^{\fg,X,{\widehat{\eta}}}/2\)\]=0.
\end{align}
If $\gamma=k^{-1}\in T$, then for $t>0$, we have 
\begin{multline}\label{eq:ellor}
\Trs^{[\gamma]}\left[\exp\(-tC^{\fg,X,\widehat{\eta}}/2\)\right]\\
=\frac{c_{Z^0(k)}}{\sqrt{2\pi}t} \exp\(\frac{t}{16}\Tr\[C^{\fu(\fb),\fu^{\bot}(\fb)}\]-\frac{t}{2}C^{\fu_\fm,\eta}\)
 \bigg[\exp\(-\frac{\omega^{{Y_\fb(k)},2}}{8\pi^2|a_0|^2 t}\)\\
 \widehat{A}_{k^{-1}}\(TY_{\fb}|_{Y_\fb(k)},\nabla^{TY_{\fb}|_{Y_\fb(k)}}\)\mathrm{ch}_{k^{-1}}\(F_{\fb,\eta}|_{Y_\fb(k)},\nabla^{F_{\fb,\eta}|_{Y_\fb(k)}}\)\bigg]^{\max}.\end{multline}
If $\gamma=e^ak^{-1}\in H$ with $a\neq0$, then for any $t>0$, we have 
\begin{multline}\label{eq:Trg}
  \Trs^{[\gamma]}\[\exp\(-t C^{\fg,X,\widehat{\eta}}/2\)\]
=\frac{1}{{\sqrt{2 \pi t}}}\[e\(TX_{M}(k),\nabla^{TX_{M}(k)}\)\]^{\max}\\
\exp\(-\frac{|a|^2}{2t}+\frac{t}{16}\Tr\[C^{\fu(\fb),\fu^{\bot}(\fb)}\]-\frac{t}{2}C^{\fu_\fm,\eta}\)
 \frac{\Tr^{E_\eta}\[\eta(k^{-1})\]}{\left|\det\big(1-\Ad(\gamma)\big)|_{\fz^\bot_0}\right|^{1/2}}.
\end{multline}
\end{thm}
\begin{proof}
Equations \eqref{eq:G1=00}, \eqref{eq:ellor} with $\gamma=1$, and 
\eqref{eq:Trg} are \cite[Theorem  5.15]{Shfried}. It remains to show 
\eqref{eq:ellor} for a non trivial  $\gamma=k^{-1}\in T$. 
Set
\begin{align}\label{eq:772}
&\fp_\fm^\bot(k)=\fp_\fm\cap \fz^\bot(k),&\fk_\fm^\bot(k)=\fk_\fm\cap \fz^\bot(k),&&\fm^\bot(k)=\fm\cap \fz^\bot(k).
\end{align}
By \eqref{eq:772}, we have
\begin{align}
&\fp_\fm=\fp_\fm(k)\oplus \fp_\fm^\bot(k), &\fk_\fm=\fk_\fm(k)\oplus \fk_\fm^\bot(k).
\end{align}
Similarly,  $k$ acts on $\fp^\bot(\fb)$ and $\fk^\bot(\fb)$. Set
\begin{align}\label{eq:7722}
&\fp^\bot_1(\fb)=\fp^\bot(\fb)\cap \fz(k), & \fk^\bot_1(\fb)=\fk^\bot(\fb)\cap \fz(k),&& \fz_1^\bot(k)=\fz^\bot(\fb)\cap \fz(k),\\
&\fp^\bot_2(\fb)=\fp^\bot(\fb)\cap \fz^\bot(k), & \fk^\bot_2(\fb)=\fk^\bot(\fb)\cap \fz^\bot(k),&& \fz_2^\bot(k)=\fz^\bot(\fb)\cap \fz^\bot(k).\notag
\end{align}
Then
\begin{align}
&\fp^\bot(\fb)=\fp^\bot_1(\fb)\oplus \fp_2^\bot(\fb), &\fk^\bot(\fb)=\fk^\bot_1(\fb)\oplus \fk_2^\bot(\fb).
\end{align}
By \eqref{eq:772} and \eqref{eq:7722}, we get
\begin{align}\label{eq:fan}
&\fp(k)=\fb\oplus \fp_\fm(k)\oplus \fp_1^\bot(\fb), &\fk(k)= \fk_\fm(k)\oplus \fk_1^\bot(\fb),\\
&\fp^\bot(k)= \fp^\bot_\fm(k)\oplus \fp_2^\bot(\fb), &\fk^\bot(k)= \fk^\bot_\fm(k)\oplus \fk_2^\bot(\fb).\notag
\end{align}
As in the case of \cite[(6-5)]{Shfried}, we have  isomorphisms of representations of $T$,
\begin{align}\label{eq:77iso}
\fp^\bot_1(\fb)\simeq \fk_1^\bot(\fb)\simeq \fn(k),
\end{align}
where the first isomorphism is given by $\ad(a_0)$. 
Moreover, $\ad(a_0)$ induces an isomorphism of representations of $T$,
\begin{align}\label{eq:78iso}
\fp^\bot_2(\fb)\simeq \fk_2^\bot(\fb).
\end{align}
Set
\begin{align}\label{eq:fan2}
&\fu_\fm(k)=\sqrt{-1}\fp_\fm(k)\oplus \fk_{\fm}(k),&\fu^\bot_\fm(k)=\sqrt{-1}\fp^\bot_\fm(k)\oplus \fk^\bot_{\fm}(k),\\
&\fu_1^\bot(\fb)=\sqrt{-1}\fp_1^\bot(\fb)\oplus \fk_1^\bot(\fb),& \fu_2^\bot(\fb)=\sqrt{-1}\fp_2^\bot(\fb)\oplus \fk_2^\bot(\fb). \notag
\end{align}

Proceeding as \cite[7-18]{Shfried}, by \eqref{eq:trrJr} and by the Weyl integral formula for Lie algebra 
\cite[ (7-17)]{Shfried}, we have 
\begin{multline}\label{eq:tr1si}
  \Trs^{[k^{-1}]}\[\exp\(-t C^{\fg,X,\widehat{\eta}}/2\)\]\\
  =\frac{1}{(2\pi t)^{\dim 
  \fz(k)/2}}\exp\(\frac{t}{16}\Tr^{\fp}\[C^{\fk,\fp}\]+\frac{t}{48}\Tr^{\fk}\[C^{\fk,\fk}\]\)\\
\frac{\vol(K^{0}(k)/T)}{|W(T,K^{0}(k))|}
\int_{Y\in\ft}\det\big(\ad(Y)\big)\big|_{\fk(k)/\ft}\ J_{k^{-1}}(Y)\\
\Trs^{E_{\widehat{\eta}}}\[\widehat{\eta}(k^{-1})\exp\big(-i\widehat{\eta}(Y)\big)\]\exp\(-|Y|^2/2t\)dY.
\end{multline}

As $\ft$ is also the Cartan subalgebra of $\fu_\fm(k)$, we will rewrite the integral on the right-hand side as an integral over $\fu_\fm(k)$.
%
By \eqref{eq:J}, \eqref{eq:fan}-\eqref{eq:78iso}, for $Y\in \ft$, we have
\begin{multline}\label{eq:J2}
  J_{k^{-1}}({Y})=\frac{\widehat{A}\big(i\ad({Y})|_{\fp
  _\fm(k)}\big)}{\widehat{A}\big(i\ad(Y)|_{\fk_\fm(k)}\big)}\[\det\big(1-\Ad(k^{-1})\big)|_{\fp_2^\bot(\fb)\oplus \fk_2^\bot(\fb)}\]^{-1/2}
\\
  \[\frac{1}{\det\big(1-\Ad(k^{-1})\big)|_{\fp_\fm^\bot(k)\oplus 
  \fk_\fm^\bot(k)}}\frac{\det\big(1-\exp(-i\ad(Y))\Ad(k^{-1})\big)|_{\fk_\fm^\bot(k)}}{\det\big(1-\exp(-i\ad(Y))\Ad(k^{-1})\big)|_{\fp_\fm^\bot(k)}}\]^{1/2}\!\!\!\!\!\!.\end{multline}
%
As in \cite[(7-22)]{Shfried}, by \eqref{eq:hatvar}, \eqref{eq:77iso}, and \eqref{eq:J2}, for $Y\in \ft$, we have
\begin{multline}\label{eq:kk1}
\frac{\det(\ad(Y))|_{\fk(k)/\ft}}{\det(\ad(Y))|_{\fu_\fm(k)/\ft}}J_{k^{-1}}(Y)\Trs^{E_{\widehat{\eta}}}\[\widehat{\eta}(k^{-1})\exp\big(-i\widehat{\eta}(Y)\big)\]\\
=(-1)^{\frac{\dim\fp_\fm(k)}{2}}
\det\big(\ad(Y)\big)\big|_{\fn(k)}
\widehat{A}^{-1}\(i\ad(Y)|_{\fu_\fm(k)}\)\\\Trs^{E_{\eta}}\[\eta(k^{-1})\exp(-i\eta(Y))\]
\[\det\big(1-\Ad(k^{-1})\big)|_{\fu_2^\bot(\fb)}\]^{-1/2}\\
\[\frac{\det\big(1-\exp(-i\ad(Y))\Ad(k^{-1})\big)|_{\fu_\fm^\bot(k)}}{\det\big(1-\Ad(k^{-1})\big)|_{\fu_\fm^\bot(k)}}\]^{1/2}. 
 \end{multline}
 
Let $U_M(k)$ be the centralizer of $k$ in $U_M$, and
 let $U^0_M(k)$ be the connected component of the identity in $U_M(k)$.
 The right-hand side of \eqref{eq:kk1} is $\Ad(U^0_M(k))$-invariant.
By \eqref{eq:cG}, \eqref{eq:tr1si} and \eqref{eq:kk1}, and using 
again the Weyl integral formula \cite[ (7-17)]{Shfried}, as in \cite[(7-24)]{Shfried}, we get
\begin{multline}\label{590}
\Trs^{[k^{-1}]}\[\exp\(-t 
C^{\fg,X,\widehat{\eta}}/2\)\]=\frac{(-1)^{\frac{\dim 
\fn(k)}{2}}}{(2\pi t)^{\dim \fz(k)/2}}c_{Z^0(k)}\\\frac{\exp\(\frac{t}{16}\Tr^{\fp}\[C^{\fk,\fp}\]+\frac{t}{48}\Tr^{\fk}\[C^{\fk,\fk}\]\)}{\[\det\big(1-\Ad(k^{-1})\big)|_{\fu_2^\bot(\fb)}\]^{1/2}}
\int_{Y\in\fu_{\fm}(k)}\det\big(\ad(Y)\big)|_{\fn(k)}\\
\widehat{A}^{-1}\(i\ad(Y)|_{\fu_\fm(k)}\)
\Tr^{E_{\eta}}\[\eta(k^{-1})\exp(-i\eta(Y))\]
\\
\[\frac{\det\big(1-\exp(-i\ad(Y))\Ad(k^{-1})\big)|_{\fu_\fm^\bot(k)}}{\det\big(1-\Ad(k^{-1})\big)|_{\fu_\fm^\bot(k)}}\]^{1/2}\exp(-|Y|^2/2t)dY.
 \end{multline}
 
Proceeding as in \cite[ (7-25)-(7-44)]{Shfried}, by \eqref{590},  
 we get
 \begin{multline}\label{eq:782}
\Trs^{[\gamma]}\[\exp\(-t 
C^{\fg,X,\widehat{\eta}}/2\)\]=\frac{c_{Z^0(k)}}{\sqrt{2\pi t}}\\
\exp\(\frac{t}{16}\Tr\[C^{\fu(\fb),\fu^{\bot}(\fb)}\]-\frac{t}{2}C^{\fu_\fm,\eta}\)
\bigg[\exp\(-\frac{\omega^{Y_\fb(k),2}}{8\pi^2|a_0|^2t}\)\widehat{A}_{k^{-1}}(0)\\\widehat{A}^{-1}_{k^{-1}}\(N_\fb|_{Y_\fb(k)},\nabla^{N_\fb|_{Y_\fb(k)}}\){\rm ch}_{k^{-1}}\(F_{\fb,\eta}|_{Y_\fb(k)},\nabla^{F_{\fb,\eta}|_{Y_\fb(k)}} \)\bigg]^{\max}.
 \end{multline} 
 By \eqref{eq:AAA} and \eqref{eq:782}, we get \eqref{eq:ellor}. 
\end{proof}


\subsection{Selberg zeta functions}\label{sec:selzeta}
We follow \cite[Section 7C]{Shfried}. 
Recall that $\rho:\Gamma\to \mathrm{U}(r)$ is a unitary 
representation of $\Gamma$.

\begin{defin}
For $\sigma\in \bC$,	we define a formal sum 
\begin{multline}\label{eq:Xisel}
\Xi_{\eta,\rho}(\sigma)=-\sum_{\footnotesize\substack{[\gamma]\in [\Gamma_+]\\\gamma=g_\gamma e^a k^{-1} g_\gamma^{-1}}}\Tr[\rho(\gamma)]\frac{\chi_{\rm 
orb}\big(\mathbb{S}^{1}\backslash 
B_{[\gamma]}\big)}{m_{[\gamma]}}\\\frac{\Tr^{E_\eta}\[\eta(k^{-1})\]}{\left|\det\big(1-\Ad(e^{a}k^{-1})\big)|_{\fz_0^\bot}\right|^{1/2}}e^{-\sigma|a|}
\end{multline}	
and a formal Selberg zeta function 
\begin{align}
  Z_{\eta,\rho}(\sigma)=\exp\big(\Xi_{\eta,\rho}(\sigma)\big).
\end{align}
The formal Selberg zeta function is said to be well defined if 
the same conditions as in Definition \ref{def:ruelle} hold. 
\end{defin}

%
%

Recall that the Casimir operator $C^{g,Z,\widehat{\eta},\rho}$ acting on $C^\infty(Z,\cF_{\widehat{\eta}}\otimes_\bC F)$ is a formally self-adjoint second order elliptic operator, which is bounded from below. Set
\begin{align}
  m_{\eta,\rho}(\lambda)=\dim_\bC \ker \(C^{\fg,Z,\widehat{\eta}^+\!\!\!,\rho}-\lambda\)-\dim_\bC \ker \(C^{\fg,Z,\widehat{\eta}^-\!\!\!,\rho}-\lambda\).
\end{align}
As in  \cite[(7-59)]{Shfried},
 consider  the quotient of zeta regularized determinants 
\begin{align}
{\rm det}_{\rm gr}\(C^{\fg,Z,\widehat{\eta},\rho}+\sigma\)=  \frac{\det\big(C^{\fg,Z,\widehat{\eta}^+\!\!\!,\rho}+\sigma\big)}{\det\big(C^{\fg,Z,\widehat{\eta}^-\!\!\!,\rho}+\sigma\big)}. 
\end{align}
By Remark \ref{re:redet}, it is a meromorphic function on $\bC$. 
%
Its zeros and poles 
belong to the set $\{-\lambda:\lambda\in \Sp(C^{\fg,Z,\widehat{\eta},\rho})\}$. The order of zero  at $\sigma=-\lambda$ is $m_{\eta,\rho}(\lambda)$. 



Set 
\begin{align}
\sigma_\eta=\frac{1}{8}\Tr\[C^{\fu(\fb),\fu^{\bot}(\fb)}\]-C^{\fu_\fm,\eta}.
\end{align}
Let $P_{\eta,\rho}(\sigma)$ be the odd polynomial defined by 
\begin{multline}\label{eq:polyP}
P_{\eta,\rho}(\sigma)=\sum_{\footnotesize\substack{[\gamma]\in 
[\Gamma_e]\\\gamma=g_\gamma  k^{-1} g_\gamma^{-1}}} 
c_{Z^0(k)}\Tr[\rho(\gamma)]\frac{\vol(\Gamma(\gamma)\backslash 
X(\gamma))}{n_{[\gamma]}}\\
\bigg(
\sum^{\dim \fn(k)/2}_{j=0}(-1)^j 
\frac{\Gamma(-j-\frac{1}{2})}{j!(4\pi)^{2j+\frac{1}{2}}|a_0|^{2j}} \sigma^{2j+1}\\
 \[\omega^{{Y_\fb}(k),2j}\widehat{A}_{k^{-1}}\(TY_{\fb}|_{Y_\fb(k)},\nabla^{TY_{\fb}|_{Y_\fb(k)}}\)\mathrm{ch}_{k^{-1}}\(\cF_{\fb,\eta}|_{Y_\fb(k)},\nabla^{\cF_{\fb,\eta}|_{Y_\fb(k)}}\)\]^{\max} \!\!\bigg).
\end{multline}

\begin{thm}\label{thm:detfor}There is
	$\sigma_{0}>0$ such that 
		\begin{align}\label{eq:see}
   \sum_{\footnotesize\substack{[\gamma]\in 
   [\Gamma_+]\\\gamma=g_\gamma e^a k^{-1} 
   g_\gamma^{-1}}}\frac{\left|\chi_{\rm orb}\(\bbS^1\backslash 
   B_{[\gamma]}\)\right|}{m_{[\gamma]}}\frac{e^{-\sigma_0|a|}}{\left|\det\big(1-\Ad(e^{a}k^{-1})\big)|_{\fz^\bot_0}\right|^{1/2}}<\infty.
\end{align}
	The Selberg zeta function $Z_{\eta,\rho}(\sigma)$ has a meromorphic extension to $\sigma\in \bC$ such that  the following identity of meromorphic functions on $\bC$ holds:
\begin{align}\label{eq:detfor}
  Z_{\eta,\rho}(\sigma)={\rm det}_{\rm gr}\(C^{\fg,Z,\widehat{\eta},\rho}+\sigma_\eta+\sigma^2\)\exp\big(P_{\eta,\rho}(\sigma)\big).
\end{align}
The zeros and poles of $Z_{\eta,\rho}(\sigma)$
belong to  $\{\pm i\sqrt{\lambda+\sigma_\eta}:\lambda\in \Sp(C^{\fg,Z,\widehat{\eta},\rho})\}$. If $\lambda\in \Sp(C^{\fg,Z,\widehat{\eta},\rho})$ and $\lambda\neq-\sigma_\eta$, the order of zero  at $\sigma=\pm i\sqrt{\lambda+\sigma_\eta}$ is $m_{\eta,\rho}(\lambda)$. The order of zero at $\sigma=0$ is $2m_{\eta,\rho}(-\sigma_\eta)$.
Also,
\begin{align}\label{eq:funceq}
  Z_{\eta,\rho}(\sigma)=Z_{\eta,\rho}(-\sigma)\exp\big(2 P_{\eta,\rho}(\sigma)\big).
\end{align}
\end{thm}
\begin{proof}
Proceeding as in \cite[Theorem 7.6]{Shfried}, by Theorems 
\ref{thm:sel} and \ref{thm:orbint}, our 
 theorem  follows.
 \end{proof}

\subsection{The  proof of  Theorem \ref{Thm:3} when G has compact 
center and $\delta(G)=1$ }
\label{sec:proof}
We  apply the results of subsection \ref{sec:selzeta} to $\eta_j$. 
Recall that $\alpha\in \fb^*$ is defined in \eqref{eq:655}. 
Proceeding as in 
\cite[Theorem 7.7]{Shfried}, by \eqref{eq:sumsum}, we find that 
$R_\rho(\sigma)$ is well-defined and holomorphic on the domain $\sigma\in \bC$ and $\Re(\sigma)\gg1$, and that 
\begin{align}\label{eq:RbyZ}
  R_\rho(\sigma)=\prod_{j=0}^{2l}Z_{\eta_j,\rho}\big(\sigma+(j-l)|\alpha|\big)^{(-1)^{j-1}}.
\end{align}
By Theorem \ref{thm:detfor} and \eqref{eq:RbyZ},  $R_\rho(\sigma)$ has a meromorphic extension to $\sigma\in \bC$.

For $0\le j\le 2l$, put 
\begin{align}\label{eq:rrs}
  r_{j}=m_{\eta_j,\rho}(0).
\end{align}
By the orbifold Hodge theorem \ref{thm:orbhodge}, as in \cite[(7-74)]{Shfried}, we have 
\begin{align}
\chi_{\rm top}'(Z,F)=2\sum_{j=0}^{l-1}(-1)^{j-1}r_j+(-1)^{l-1}r_l.
\end{align}
Set
\begin{align}\label{eq:Cr}
 & C_\rho=\prod_{j=0}^{l-1}\big(-4(l-j)^2|\alpha|^2\big)^{(-1)^{j-1}r_{j}},&r_\rho=2\sum_{j=0}^{l}(-1)^{j-1}r_{j}.
\end{align}

Proceeding as in \cite[(7-76)-(7-78)]{Shfried}, we get  
\eqref{eq:masi1}. If $H^{\cdot}(Z,F)=0$, proceeding as in 
\cite[Corollary 8.18]{Shfried}, for all $0\le j\le 2l$, we have 
\begin{align}
r_j=0.
\end{align}
By \eqref{eq:Cr}, we get \eqref{eq:masi2}, which completes the proof 
of Theorem \ref{Thm:3} in the case where $G$ has compact center and  
$\delta(G)=1$. \qed

\def\cprime{$'$}
\providecommand{\bysame}{\leavevmode\hbox to3em{\hrulefill}\thinspace}
\providecommand{\MR}{\relax\ifhmode\unskip\space\fi MR }
\providecommand{\MRhref}[2]{%
  \href{http://www.ams.org/mathscinet-getitem?mr=#1}{#2}
}
\providecommand{\href}[2]{#2}

\address{Institut de Mathématiques de Jussieu-Paris Rive Gauche, \\
Sorbonne Université,\\
4 place Jussieu, 75252 Paris Cedex 05, France\\
\email{shu.shen@imj-prg.fr
}}

\address{School of Mathematical Sciences, \\
University of Science and Technology of China,\\
96 Jinzhai Road, 
Hefei, Anhui 230026,
P. R. China.\\
\email{jianqing@ustc.edu.cn}\\
\received{}
}


\begin{thebibliography}{DKV79}


\bibitem{Ruan_Orbifold}
A.~Adem, J.~Leida, and Y.~Ruan, \emph{Orbifolds and stringy topology},
  Cambridge Tracts in Mathematics, vol. 171, Cambridge University Press,
  Cambridge, 2007. 
  
\bibitem{AtiyahBott67}
M.~F. Atiyah and R.~Bott, \emph{A {L}efschetz fixed point formula for elliptic
  complexes. {I}}, Ann. of Math. (2) \textbf{86} (1967), 374--407. 

\bibitem{AtiyahBott68}
M.~F. Atiyah and R.~Bott, \emph{A {L}efschetz fixed point formula for elliptic complexes. {II}.
  {A}pplications}, Ann. of Math. (2) \textbf{88} (1968), 451--491. 

\bibitem{BGV}
N.~Berline, E.~Getzler, and M.~Vergne, \emph{Heat kernels and {D}irac
  operators}, Grundlehren Text Editions, Springer-Verlag, Berlin, 2004,
  Corrected reprint of the 1992 original. 


\bibitem{B09}
J.-M. Bismut, \emph{Hypoelliptic {L}aplacian and orbital integrals}, Annals of
  Mathematics Studies, vol. 177, Princeton University Press, Princeton, NJ,
  2011. 
  
\bibitem{BG01}
J.-M. Bismut and S.~Goette, \emph{Families torsion and {M}orse functions},
  Ast\'erisque (2001), no.~275, x+293. 

\bibitem{BGdeRham}
J.-M. Bismut and S.~Goette, \emph{Equivariant de {R}ham torsions}, Ann. of Math. (2) \textbf{159}
  (2004), no.~1, 53--216. 
  
\bibitem{BLott}
J.-M. Bismut and J.~Lott, \emph{Flat vector bundles, direct images and higher
  real analytic torsion}, J. Amer. Math. Soc. \textbf{8} (1995), no.~2,
  291--363. 
  

\bibitem{BZ92}
J.-M. Bismut and W.~Zhang, \emph{An extension of a theorem by {C}heeger and
  {M}\"uller}, Ast\'erisque (1992), no.~205, 235, With an appendix by
  Fran{\c{c}}ois Laudenbach. 
  
\bibitem{BZ94}
J.-M. Bismut and W.~Zhang, \emph{Milnor and {R}ay-{S}inger metrics on the equivariant determinant
  of a flat vector bundle}, Geom. Funct. Anal. \textbf{4} (1994), no.~2,
  136--212. 
  
 \bibitem{Bridson_Haefliger}
M.~R. Bridson and A.~Haefliger, \emph{Metric spaces of non-positive curvature},
  Grundlehren der Mathematischen Wissenschaften [Fundamental Principles of
  Mathematical Sciences], vol. 319, Springer-Verlag, Berlin, 1999. 
  
\bibitem{BruningMa06}
J.~Br\"{u}ning and X.~Ma, \emph{An anomaly formula for {R}ay-{S}inger metrics
  on manifolds with boundary}, Geom. Funct. Anal. \textbf{16} (2006), no.~4,
  767--837. 


\bibitem{Ch79}
J.~Cheeger, \emph{Analytic torsion and the heat equation}, Ann. of Math. (2)
  \textbf{109} (1979), no.~2, 259--322. 
  

\bibitem{Daiyu}
  X.~Dai and J.~Yu, \emph{Comparison between two analytic torsions on orbifolds},
  Math. Z. \textbf{285} (2017), no.~3-4, 1269--1282.
  
\bibitem{DuistermaatKolkVaradarajan}
J.~J. Duistermaat, J.~A.~C. Kolk, and V.~S. Varadarajan, \emph{Spectra of
  compact locally symmetric manifolds of negative curvature}, Invent. Math.
  \textbf{52} (1979), no.~1, 27--93.   

\bibitem{DyatlovZworski}
S.~Dyatlov and M.~Zworski, \emph{Dynamical zeta functions for {A}nosov flows via microlocal
  analysis}, Ann. Sci. \'Ec. Norm. Sup\'er. (4) \textbf{49} (2016), no.~3,
  543--577. 
  
\bibitem{zworski_zero}
S.~Dyatlov and M.~Zworski, \emph{Ruelle zeta function at zero for surfaces}, Invent. Math.
  \textbf{210} (2017), no.~1, 211--229.

\bibitem{Farsi_eta}
C.~Farsi, \emph{Orbifold {$\eta$}-invariants}, Indiana Univ. Math. J.
  \textbf{56} (2007), no.~2, 501--521. 


\bibitem{Fedosova_orb_zeta}
K.~Fedosova, \emph{The twisted {S}elberg trace formula and the {S}elberg zeta function
  for compact orbifolds}, arXiv:1511.04208 (2015).

\bibitem{Fedosova_orb_hyper}
K.~Fedosova, \emph{On the asymptotics of the analytic torsion for compact
  hyperbolic orbifolds}, arXiv:1511.04281 (2015).

\bibitem{Fedosova_orb_finite}
K.~Fedosova, \emph{Analytic torsion of finite volume hyperbolic orbifolds},
  arXiv:1601.07873 (2016).


\bibitem{FriedRealtorsion}
D.~Fried, \emph{Analytic torsion and closed geodesics on hyperbolic
  manifolds}, Invent. Math. \textbf{84} (1986), no.~3, 523--540. 

\bibitem{Friedconj}
D.~Fried, \emph{Lefschetz formulas for flows}, The {L}efschetz centennial
  conference, {P}art {III} ({M}exico {C}ity, 1984), Contemp. Math., vol.~58,
  Amer. Math. Soc., Providence, RI, 1987, pp.~19--69. 
  
\bibitem{GLP2013}
P.~Giulietti, C.~Liverani, and M.~Pollicott, \emph{Anosov flows and dynamical
  zeta functions}, Ann. of Math. (2) \textbf{178} (2013), no.~2, 687--773.
  
\bibitem{GuruprasadHaefiger06}
K.~Guruprasad and A.~Haefliger, \emph{Closed geodesics on orbifolds}, Topology
  \textbf{45} (2006), no.~3, 611--641. 

\bibitem{Haefliger_orb}
A.~Haefliger, \emph{Orbi-espaces}, Sur les groupes hyperboliques d'apr\`es
  {M}ikhael {G}romov ({B}ern, 1988), Progr. Math., vol.~83, Birkh\"auser
  Boston, Boston, MA, 1990, pp.~203--213. 

\bibitem{Kawasaki_Orb_sign}
T.~Kawasaki, \emph{The signature theorem for {$V$}-manifolds}, Topology
  \textbf{17} (1978), no.~1, 75--83. 
  
\bibitem{Kawasaki_RR}
T.~Kawasaki, \emph{The {R}iemann-{R}och theorem for complex {$V$}-manifolds}, Osaka
  J. Math. \textbf{16} (1979), no.~1, 151--159. 

\bibitem{Knappsemi}
A.~W. Knapp, \emph{Representation theory of semisimple groups}, Princeton
  Mathematical Series, vol.~36, Princeton University Press, Princeton, NJ,
  1986, An overview based on examples. 

\bibitem{KnappLie}
A.~W. Knapp, \emph{Lie groups beyond an introduction}, second ed., Progress in
  Mathematics, vol. 140, Birkh\"auser Boston, Inc., Boston, MA, 2002.
  

\bibitem{Ma_Orbifold_immersion}
X.~Ma, \emph{Orbifolds and analytic torsions}, Trans. Amer. Math. Soc.
  \textbf{357} (2005), no.~6, 2205--2233 (electronic). 
  
\bibitem{Ma_bourbaki}
X.~Ma, \emph{Geometric hypoelliptic {L}aplacian and orbital integral, after
  {B}ismut, {L}ebeau and {S}hen}, no. 407, 2019, S\'{e}minaire Bourbaki. Vol.
  2016/2017. Expos\'{e}s 1120--1135, pp.~Exp. No. 1130, 333--389. 
  available at \url{http://www.bourbaki.ens.fr/TEXTES/1130.pdf}. Video available at 
  \url{https://www.youtube.com/watch?v=dCDwN-HqcJw}.
  
\bibitem{MaMa}
X.~Ma and G.~Marinescu, \emph{Holomorphic {M}orse inequalities and {B}ergman
  kernels}, Progress in Mathematics, vol. 254, Birkh\"auser Verlag, Basel,
  2007.   

  \bibitem{MckeanSinger}
H.~P. McKean and I.~M. Singer, \emph{Curvature and the eigenvalues of the
  {L}aplacian}, J. Differential Geometry \textbf{1} (1967), no.~1, 43--69.


\bibitem{MilnorZcover}
J.~Milnor, \emph{Infinite cyclic coverings}, Conference on the {T}opology of
  {M}anifolds ({M}ichigan {S}tate {U}niv., {E}. {L}ansing, {M}ich., 1967),
  Prindle, Weber \& Schmidt, Boston, Mass., 1968, pp.~115--133. 

\bibitem{MoPr_orbifold}
I.~Moerdijk and D.~A. Pronk, \emph{Orbifolds, sheaves and groupoids},
  $K$-Theory \textbf{12} (1997), no.~1, 3--21. 

  

\bibitem{MStorsion}
H.~Moscovici and R.~J. Stanton, \emph{{$R$}-torsion and zeta functions for
  locally symmetric manifolds}, Invent. Math. \textbf{105} (1991), no.~1,
  185--216. 

\bibitem{Muller78}
W.~M{\"u}ller, \emph{Analytic torsion and {$R$}-torsion of {R}iemannian
  manifolds}, Adv. in Math. \textbf{28} (1978), no.~3, 233--305. 

\bibitem{Muller2}
W.~M{\"u}ller, \emph{Analytic torsion and {$R$}-torsion for unimodular
  representations}, J. Amer. Math. Soc. \textbf{6} (1993), no.~3, 721--753.
  

\bibitem{Quillensuper}
D.~Quillen, \emph{Superconnections and the {C}hern character}, Topology
  \textbf{24} (1985), no.~1, 89--95. 
  
\bibitem{ReidemeisterTorsion}
K.~Reidemeister, \emph{Homotopieringe und {L}insenr\"aume}, Abh. Math. Sem.
  Univ. Hamburg \textbf{11} (1935), no.~1, 102--109. 


\bibitem{RSTorsion}
D.~B. Ray and I.~M. Singer, \emph{{$R$}-torsion and the {L}aplacian on
  {R}iemannian manifolds}, Advances in Math. \textbf{7} (1971), 145--210.
  

\bibitem{Satake_gene_mfd}
I.~Satake, \emph{On a generalization of the notion of manifold}, Proc. Nat.
  Acad. Sci. U.S.A. \textbf{42} (1956), 359--363. 

\bibitem{SatakeGaussB}
I.~Satake, \emph{The {G}auss-{B}onnet theorem for {$V$}-manifolds}, J. Math. Soc.
  Japan \textbf{9} (1957), 464--492. 
  
\bibitem{Seeley66}
R.~T. Seeley, \emph{Complex powers of an elliptic operator}, Singular
  {I}ntegrals ({P}roc. {S}ympos. {P}ure {M}ath., {C}hicago, {I}ll., 1966),
  Amer. Math. Soc., Providence, R.I., 1967, pp.~288--307. 
  
\bibitem{Selberg60}
A.~Selberg, \emph{On discontinuous groups in higher-dimensional symmetric
  spaces}, Contributions to function theory (internat. {C}olloq. {F}unction
  {T}heory, {B}ombay, 1960), Tata Institute of Fundamental Research, Bombay,
  1960, pp.~147--164. 


\bibitem{Shfried}
S.~Shen, \emph{Analytic torsion, dynamical zeta functions, and the {F}ried
  conjecture}, Anal. PDE \textbf{11} (2018), no.~1, 1--74. 

\bibitem{S61}
S.~Smale, \emph{On gradient dynamical systems}, Ann. of Math. (2) \textbf{74}
  (1961), 199--206. 

\bibitem{S67}
S.~Smale, \emph{Differentiable dynamical systems}, Bull. Amer. Math. Soc.
  \textbf{73} (1967), 747--817. 
  
\bibitem{Thurston_geo_3_maniflod}
W.~P. Thurston, \emph{The geometry and topology of three-manifolds},  available at \url{http://www.msri.org/publications/books/gt3m}, 
  unpublished manuscript, 1980.
  
\bibitem{Voros}
A.~Voros, \emph{Spectral functions, special functions and the {S}elberg zeta
  function}, Comm. Math. Phys. \textbf{110} (1987), no.~3, 439--465.
  
\bibitem{Waldron}
J.~Waldron, \emph{Lie algebroids over differentiable stacklie algebroids over
  differentiable stacks}, arXiv:1511.07366.

\end{thebibliography}
\end{document}